\RequirePackage{fix-cm}
\documentclass[10pt,reqno]{amsart}
\usepackage[T1]{fontenc}
\usepackage{mlmodern,amsmath,amssymb,amsthm,mathrsfs}
\usepackage[letterpaper,margin=1.1in]{geometry}
\usepackage{enumitem,microtype,graphicx,pgfplots}
\usepackage{fix-cm}
\usepackage[colorlinks=true,linkcolor=blue,citecolor=blue,urlcolor=blue,hypertexnames=false]{hyperref}
\usepackage{ulem}
\usepackage{esint}

\newtheorem{theorem}{Theorem}
\newtheorem{proposition}{Proposition}
\newtheorem{lemma}{Lemma}
\newtheorem{corollary}{Corollary}
\newtheorem{claim}{{\it Claim}}

\newtheorem*{theoremA}{Theorem A}
\newtheorem*{theorem*}{Theorem}
\theoremstyle{definition}

\newtheorem{remark}{Remark}[section]
\newtheorem*{remark*}{Remark}
\newtheorem*{remarks}{Remarks}
\newtheorem*{definition*}{Definition}
\newtheorem*{example*}{Example}
\newcommand{\loc}{\mathrm{loc}}

\numberwithin{equation}{section}
\makeatletter
\renewcommand{\normalsize}{%
  \@xsetfontsize\normalsize 6%
  \fontsize{10.4pt}{4.4mm}\selectfont
  \@adjustvertspacing\let\@listi\@listI}
\makeatother
\AtBeginDocument{\normalsize}

\begin{document}

\title[Heat kernels of Keller-Segel particle systems]{Upper bound on heat kernels of finite particle systems of Keller-Segel type}

\author{S.E.\,Boutiah and D.\,Kinzebulatov}

\address{Universit\'{e} Laval, D\'{e}partement de math\'{e}matiques et de statistique, Qu\'{e}bec, QC, Canada and Laboratoire de Math\'{e}matiques Appliqu\'{e}es, Universit\'{e} Ferhat Abbas, S\'{e}tif 1, Campus El Bez, S\'{e}tif, Algeria}

\email{sallah-eddine.boutiah.1@ulaval.ca}

\address{Universit\'{e} Laval, D\'{e}partement de math\'{e}matiques et de statistique, Qu\'{e}bec, QC, Canada}
\email{damir.kinzebulatov@mat.ulaval.ca}

\thanks{The research of D.K. is supported by NSERC grant RGPIN-2024-04236.}


\begin{abstract}
We obtain an upper bound on the heat kernel of the finite particle approximation of the Keller-Segel model of chemotaxis. The proof exploits a connection with certain non-local operators. We also obtain upper heat kernel bounds for related systems of interacting Brownian and stable particles that exhibit similar local blow-up effects.
\end{abstract}


\maketitle

\setcounter{tocdepth}{1}
\tableofcontents

\section{Introduction}\label{main_result}
The subject of this paper is the interacting particle system
\begin{equation}\label{ks}
 dX_t^i=-\frac\nu N\sum_{j\ne i}\frac{X_t^i-X_t^j}{|X_t^i-X_t^j|^2}\,dt
           +\sqrt2\,dB_t^i,\qquad 1\le i\le N,
\end{equation}
and related particle systems. Here $X_t^i$ is the position of the $i$-th particle in $\mathbb R^2$ at time $t$, the $B^i$ are independent two-dimensional Brownian motions, and $N\ge2$ is the number of particles. This system is a finite particle approximation of the Keller-Segel model of chemotaxis \eqref{ks_chem}, where the non-negative constant $\nu$ measures the strength of attraction between the particles. In the absence of noise, \eqref{ks} is a system of ordinary differential equations, and so whenever $\nu>0$ the particles collide  and stay glued up due to the strong attraction between them, i.e.\,there is a blow up. Introducing Brownian noise (or thermal excitation) moves the blow-up threshold:
$$ 
\nu_\star=0 \quad \longrightarrow \quad \nu_\star=4.
$$
i.e.\,now for every $\nu <\nu_\star:=4$ the evolution of the particles continues indefinitely provided that $N$ is large and the initial distribution has no atoms. More precisely, under these assumptions the particle system \eqref{ks} has a global in time weak solution in the sense of stochastic differential equations (SDEs). If, on the other hand, $\nu \geq \nu_\star$, then all the particles collide a.s.\,and stay glued up; this can be seen upon noting that $R_t:=\frac{1}{4N}\sum_{i,j=1}^N|X_t^i-X_t^j|^2$ is a local squared Bessel process of dimension $(N-1)(2-\frac{\nu}{2})$. See \cite{F} for detailed discussion and references (we comment on the existing literature further below).

What makes \eqref{ks} and the Keller-Segel model interesting from the theoretical perspective is that both of these equations are strictly at the limit of the general theory: the drift and the diffusion terms are of comparable strength. Furthermore, by testing the diffusion by such singular drifts we get a deeper insight into the regularizing effect of the Brownian motion.

Our goal in this paper is to obtain an upper bound on the density of the law of \eqref{ks} ($=:$\,heat kernel) $p(t,x,y)$. Having such a bound provides occupation estimates for the solution  to \eqref{ks}:
$$
\mathbb E_x\int_0^T \mathbf{1}_{A}(X_t)dt = \int_{A}p(t,x,y)dydt, \quad X=(X^1,\dots,X^N),\;\; A \subset [0,T] \times \mathbb R^{dN}.
$$
The singularities of the drift in \eqref{ks} make invalid the standard Gaussian upper heat kernel bound, and require a weight that describes the behaviour of the interaction kernel near the collision sets. We will introduce such a weight, with the time dependence dictated by the diffusive scaling. Along the way, we will need to establish non-classical Sobolev regularity estimates for the corresponding Kolmogorov operator; the ``non-classical'' is related to dimension two, see below. 

The Keller-Segel finite system \eqref{ks} exhibits critical behaviour in two important ways. First, there are blow-ups. The blow-ups, however, also occur in the higher-dimensional counterpart of \eqref{ks} that we consider in Section \ref{high_dim_sect}. The two-dimensionality of the particles in system \eqref{ks}  is the other reason that makes it difficult to handle. Namely, the drift in the Kolmogorov backward operator corresponding to \eqref{ks} 
$$
\Lambda=-\Delta + \frac{\nu}{N}\sum_{i=1}^N \sum_{j=1, j \neq i}^N \frac{x^i-x^j}{|x^i-x^j|^2} \cdot \nabla_{x_i}
$$
is not in $L^2_{\loc}=L^2_{\loc}(\mathbb R^{2N})$, the Cauchy problem for the corresponding parabolic equation is not well-posed in the standard Hilbert triple of Sobolev spaces $W^{1,2}(\mathbb R^{2N}) \rightarrow L^2(\mathbb R^{2N})  \rightarrow W^{-1,2}(\mathbb R^{2N}) $, and the use of the energy methods to develop regularity theory of \eqref{ks} in the standard (unweighted) Sobolev spaces is problematic. These difficulties do not appear in the higher-dimensional counterparts of \eqref{ks}. 

The source of these analytic difficulties is, one can argue, the lack of the Hardy inequality in $\mathbb R^2$. If $d \geq 3$, then the (usual) Hardy inequality 
\begin{equation*}
\frac{(d-2)^2}{4}\int_{\mathbb R^d} \frac{f^2}{|x|^2}dx \leq \|\nabla f\|_2^2, \quad f \in W^{1,2}(\mathbb R^{d}),
\end{equation*}
allows us to control the drift term in $\Lambda$ in terms of the quadratic form of the Laplacian, and thus allows us to prove the usual energy inequality. There is no non-trivial Hardy inequality in dimension $d=2$. There is, however, a non-trivial \textit{fractional} Hardy inequality in $\mathbb R^2$:
\begin{equation}\label{frac2}
 2^\beta\frac{\Gamma((2+\beta)/4)^2}{\Gamma((2-\beta)/4)^2}
 \int_{\mathbb R^2}\frac{|f|^2}{|x|^\beta}dx
 \le\|(-\Delta)^{\beta/4}f\|_2^2,
 \qquad 1\le\beta<2,\quad f\in W^{\beta/2,2}(\mathbb R^2).
\end{equation}
We will exploit that in order to estimate the heat kernel $p(t,x,y)$ of \eqref{ks}, even though a priori there is nothing non-local about  the Keller-Segel system \eqref{ks}. 
Furthermore, \eqref{frac2}, combined with the particle structure of the drift, will provide us with a many-particle fractional Hardy inequality. Together with the fractional resolvent representation of the type introduced in \cite{S,Ki_super} - crucial for handling $d=2$ - it will yield the regularity estimates needed to verify conditions of the abstract desingularization theorem (Theorem A) obtained earlier in \cite{KSS}. The latter then will yield
\[
 p(t,x,y)\le Ct^{-N}\prod_{i<j}
            \left(1+\frac{t}{|y_i-y_j|^2}\right)^{\nu/(2N)},
\]
where the product describes the behaviour of the heat kernel along the collision sets, see Theorem~\ref{thm1}. In Corollary~\ref{cor:planar-gaussian} we will upgrade this upper bound to 
\[
 p(t,x,y)\le Ct^{-N}\prod_{i<j}
            \left(1+\frac{t}{|y_i-y_j|^2}\right)^{\nu/(2N)}e^{-\frac{|x-y|^2}{ct}}
\]
using a result of Coulhon-Sikora \cite{CS}. 

An analogous upper heat kernel bound holds in dimensions $d \geq 3$ (Theorem \ref{thm2}). We handle them using Moser's iterations run in the Dirichlet form setting, which allows us to relax the constraint on $\nu$ compared to Theorem \ref{thm1} (the latter also can be extended, using the same method, to higher dimensions). In the follow-up paper \cite{BKM} we have established the corresponding uniform in $N$ spectral gap inequality, and derived the lower heat kernel bound  of the same form. So, at least for $d \geq 3$, the upper bound is sharp (up to a choice of constants $C$, $c$ that is different from the ones in the lower bound).

In the first version of this paper, the admissible strength of attraction $\nu$ in dimension $d=2$ decreased, after allowing a choice of $\beta=\beta(N)$ in the fractional Hardy inequality, as $1/(e\log N)$. Even this slow decrease of the admissible value of $\nu$ as $N$ tends to infinity is not what one would expect from an optimal result for the Keller-Segel particle system. 

In the present version, using some ideas of \cite{HHLT} to obtain a fractional many-particle  Hardy inequality, we find a better way to exploit the many-particle structure of the drift. As a consequence, the condition in Theorem~\ref{thm1} no longer degenerates as $N\to\infty$: its admissible upper bound tends to $1$.
For finite particle systems, it remains important to select $\beta$ as a function of $N$ in order to obtain a condition on $\nu$ that is as best as possible provided by our approach - this is the reason for the maximum in Theorem~\ref{thm1}. See  also the comment on Sturm's result below.

In Section \ref{nonlocal_sect} we apply the same approach to the stable-driven counterpart of \eqref{ks}, i.e.
\begin{equation}
\label{ks_stable}
 dX_t^i=-\frac\nu N\sum_{j\ne i}\frac{X_t^i-X_t^j}{|X_t^i-X_t^j|^\alpha}\,dt + dZ_t^i,\qquad 1\le i\le N,
\end{equation}
where $1<\alpha<2$ and $Z^i$ are independent isotropic $\alpha$-stable processes in $\mathbb R^d$, $d \geq 2$. The corresponding Kolmogorov operator is 
$$
\Lambda_\alpha= A_\alpha + \frac{\nu}{N}\sum_{i=1}^N \sum_{j=1, j \neq i}^N \frac{x^i-x^j}{|x^i-x^j|^\alpha} \cdot \nabla_{x_i},
$$
where
$$
A_\alpha=\sum_{i=1}^N (-\Delta)^{\frac{\alpha}{2}}_{x^i}.
$$
See Theorem~\ref{nl:main}. The condition on $\nu$ again does not degenerate as $N\to\infty$. The proof follows the same approach as the proof of Theorem~\ref{thm1}, although there are some non-trivial differences in the calculations.

Some time after the first version of the present paper appeared on arXiv, we became aware
that the Brownian heat kernel bound in Theorem~\ref{thm1} (or rather in Corollary \ref{cor:planar-gaussian}, i.e.\,with the Gaussian factor) can also be
obtained from Sturm's 1996 result on local Dirichlet spaces
\cite[Corollary~4.2]{Sturm}. Indeed, the identity
$b_\varepsilon=-\nabla\log\psi_\varepsilon$, with the density defined
in \eqref{eps_reg}, makes the regularized
operator symmetric in $L^2(\psi_\varepsilon\,dx)$. Under the
assumptions of Theorem~\ref{thm1}, these densities are Muckenhoupt
$A_2$ weights, uniformly in $\varepsilon$, so their measures satisfy
the doubling and weighted spectral gap inequalities, see \cite{FKS}. Furthermore, specifying Sturm's result to particle system \eqref{ks}, one arrives at the condition $\nu<\nu_\star=4$, i.e.\,it covers the entire range of strength of attraction up to the critical value. In Theorem \ref{thm1} we only reach $\nu<1$ as $N \rightarrow \infty$.
Our approach nevertheless has its own advantages. In particular, it produces a heat kernel in the physical space,  provides  Sobolev regularity results for the corresponding Kolmogorov equations and, perhaps most importantly, works for stable-driven particles as well. Our arguments do not depend on the gradient form of the many-particle drift, but rather exploit the cancellation properties of the interaction kernel, captured by various forms of many-particle Hardy inequality that we prove below.

\subsection{Literature}
Let us start with heat kernel bounds for singular particle systems. We refer to the results of Graczyk-Sawyer \cite{GS1,GS2,GS3} and Dziuba\'nski-Hejna \cite{DH} concerning two-sided heat kernel bounds for the Dunkl Laplacian. Their situation is different: the focus is on repulsive interactions, and they have explicit, although non-elementary, formulas for the heat kernel are available. We also mention the work of Giunti-Gu-Mourrat \cite{GGM} on the upper heat kernel bound for the symmetric simple exclusion process, although this particle system is quite different from ours.

There is a substantial literature on heat kernel bounds for local and non-local operators with coefficients having critical singularities at the origin, for which the standard two-sided Gaussian bounds do not hold. We refer, in particular, to Milman-Sem\"enov \cite{MS}, Metafune-Sobajima-Spina \cite{MSS}, Metafune-Negro-Spina \cite{MNS}, and Kinzebulatov-Sem\"enov-Szczypkowski \cite{KSS}; this list is far from exhaustive. Non-standard heat kernel bounds also appear in other settings with unbounded coefficients, see, e.g., Boutiah-Rhandi-Tacelli \cite{BRT} and the references therein.

Regarding the weak well-posedness of the Keller-Segel finite particle system \eqref{ks}, we refer to Cattiaux-P\'ed\`eches \cite{CP} and Fournier-Jourdain \cite{FJ}, who obtained detailed results using rather different approaches. See also the more recent advances in Fournier-Tardy \cite{FT} and Tardy \cite{T}. The point of departure for Cattiaux-P\'ed\`eches is a Dirichlet-form construction with test functions supported away from pair collisions, which addresses the lack of local square integrability of the drift discussed above. They solve the corresponding martingale problem and construct the heat kernel, among other results. The argument of Fournier-Jourdain appeals directly to the stochastic equations and uses in an essential way the special form of the interaction kernel. Their analysis of collisions also permits a treatment of the classical martingale problem, without excluding the collision sets from the supports of the test functions.

We also refer to Cattiaux \cite{Ca2}, where entropy and Orlicz-space methods are used to study the singularities of the Keller-Segel system. For particle systems with other singular interactions, such as those of Bessel type, see Graczyk-Ma\l{}ecki \cite{GM}, Hufnagel-Andraus \cite{HA}, and the references therein.

We also mention that the theory of SDEs with general singular drifts that can introduce strong attraction to submanifolds, was brought up to the task only recently, i.e.\,one can now handle the Keller-Segel finite particle system \eqref{ks} as a special case of what is now known about general SDEs, see \cite{KiS_brownian,Ki_Morrey}, see also Krylov \cite{Kr} and references therein, albeit with losses in the assumptions on $\nu$ compared to \cite{CP,FJ,FT,T} whose methods are tailored to \eqref{ks} (this is partially remedied in \cite{Ki_multi,KV}). As a by-product of the new general theory of SDEs with singular drifts, one now has flexible means for modifying the interaction kernel in the particle system \eqref{ks}.

There is a rich literature on the Keller-Segel model of chemotaxis, described by a distribution-dependent SDE
\begin{equation}
\label{ks_chem}
dY_t=(K \star \eta)(Y_t) dt+\sqrt{2} dB_t \quad \text{ in } \mathbb R^2,
\end{equation}
where $K(y)=- \nu |y|^{-2}y$, $y \in \mathbb R^2$, is the interaction kernel in \eqref{ks}, $\eta(t,y)$ is the law of $Y_t$ and $\star$ is the convolution in the spatial variables, see, in particular, \cite{Ca2, CPZ} and references therein. If $(X_t^{1,N},\dots,X_t^{N,N})$ denotes the solution of \eqref{ks} then, under the exchangeability hypothesis on the initial condition, the process $Y_t$ is obtained as the limit of (sub-)sequence $X_t^{1,N}$ as $N \rightarrow \infty$, see \cite{FJ}, i.e.\,$Y_t$ describes the behaviour of a typical particle in the limit as the number of particles goes to infinity. Regarding the propagation of chaos in general critical settings, see Bresch-Jabin-Wang \cite{BJW}, Jabin-Wang \cite{JW} and Hao-R\"{o}ckner-Zhang \cite{HRZ}, see also references therein. 

Assuming that some mild regularity conditions are imposed on the density of $Y_0$, the regularizing effect of the convolution in \eqref{ks_chem} makes the drift in \eqref{ks_chem} more regular than the drift in the finite particle system \eqref{ks_chem} and thus opens up a way for the use of other methods such as De Giorgi's method \cite{JL}. In this setting, the problem with the lack of the Hardy inequality in dimension $2$ does not arise.

Let us also add that our heat kernel bounds are independent of $N$ only in the sense that the condition on admissible $\nu$ is $N$-independent. The constants in these bounds are $N$-dependent.

\subsection{Notation}
Given a sequence $\{T_n\}$ of bounded linear operators $X \rightarrow Y$ between Banach spaces $X$, $Y$, endowed with the operator norm $\|\cdot\|_{X \rightarrow Y}$, we write $$T=s\mbox{-} Y \mbox{-}\lim_n T_n$$ if $$\lim_n\|Tf- T_nf\|_Y=0 \quad \text{ for every $f \in X$}.
$$ 

Let $L^p=L^p(\mathbb R^{dN},dx)$, $W^{1,p}=W^{1,p}(\mathbb R^{dN},dx)$ denote the usual Lebesgue and Sobolev spaces, respectively. 

Set $\|\cdot\|_p:=\|\cdot\|_{L^p}$
and denote operator norm
$
\|\cdot\|_{p \rightarrow q}:=\|\cdot\|_{L^p \rightarrow L^q}
$.

Given $1 < p < \infty$, we set $p':=\frac{p}{p-1}$.

Let $W^{\alpha,p}$ ($\alpha>0$) denote the Bessel potential space endowed with norm $\|u\|_{p,\alpha}:=\|g\|_p$,  
$u=(1-\Delta)^{-\frac{\alpha}{2}}g$, $g \in L^p$, and $W^{-\alpha,p'}$, $p'=p/(p-1)$, the anti-dual of $W^{\alpha,p}$. 

Set
\[
 \operatorname{dist}(\operatorname{supp}f,\operatorname{supp}g)
 :=\inf\bigl\{|x-y|:\ x\in\operatorname{supp}f,\ 
                         y\in\operatorname{supp}g\bigr\}.
\]

Put
$$
\langle f,g\rangle = \langle f \bar{g}\rangle :=\int_{\mathbb R^{dN}}f \bar{g} dx.$$ 
For vector fields $b$, $\mathsf{f}$, we put
$$
\langle b,\mathsf{f}\rangle:=\langle b \cdot \mathsf{f}\rangle \qquad \text{($\cdot$ is the scalar product)}.
$$

Let us fix the coefficient in the Fourier transform:
\[
 \widehat f(\xi)
 =\int_{\mathbb R^d}e^{-iz\cdot\xi}f(z)\,dz.
\]

Throughout the Brownian particles parts of the paper (Sections \ref{main_sect} and \ref{high_dim_sect}) we will use the following regularized weights (in every dimension):
\begin{equation}\label{eps_reg}
\begin{aligned}
 \varphi_{s,\varepsilon}& :=\frac{\psi_\varepsilon}{\psi_{\varepsilon+s}} \\
 &\text{where} \\
 &
 \psi_\varepsilon(x):=\prod_{i<j}|x_i-x_j|_\varepsilon^{-\nu/N}, \quad |z|_\varepsilon:=(|z|^2+\varepsilon)^{1/2}.
 \end{aligned}
\end{equation}
We set $\psi=\psi_0$ and $\varphi_s=\varphi_{s,0}$ outside of the collision sets. For the stable-driven particles (Section \ref{nonlocal_sect}) we define the regularized weights in a similar way.

We use the following notation for the Gaussian density on $\mathbb R^{dN}$:
$$
 \Gamma_t(x)=(4\pi t)^{-dN/2}e^{-\frac{|x|^2}{4t}}.
$$

\subsection{Organization}
In Section~\ref{main_sect} we state the two-dimensional Brownian particles result, Theorem \ref{thm1}, as well as the abstract desingularization theorem Theorem~A which is used in its proof. The stable-driven case is handled in Section~\ref{nonlocal_sect} (Theorem \ref{nl:main}), also using Theorem A. In Section~\ref{high_dim_sect}, we return to Brownian particles in dimensions $d \geq 3$  and give Theorem \ref{thm2} whose proof is rather classical by now, and uses Moser's iterations ran in the Dirichlet form setting. As is mentioned, its advantage, compared to Theorems \ref{thm1}, \ref{nl:main}, is that it allows to handle larger strength of attraction $\nu$, albeit still losing to Sturm \cite{Sturm} (that is, $\nu<2(d-2)$ against $\nu<2d$). However, in the proof of Theorem \ref{thm2} we provide a construction of the limiting Kolmogorov operator and its heat kernel in the physical space (and not associated to an abstract Dirichlet generator). The proofs of Theorems~\ref{thm1} and~\ref{nl:main} are given in Sections~\ref{proof:planar} and~\ref{proof:nonlocal}. The proofs of Corollary~\ref{cor:planar-gaussian} and Theorem~\ref{thm2} are given in Sections~\ref{proof:planar-gaussian} and~\ref{proof:higher}. 

\bigskip

\section{Two-dimensional Brownian particles}\label{main_sect}
We start with a particle system with regularized interaction kernels. The corresponding backward Kolmogorov operator with domain $W^{2,2}$ is
\begin{equation}\label{planar:operator}
 \Lambda_\varepsilon=-\Delta+b_\varepsilon\cdot\nabla,
 \qquad b_{\varepsilon,i}(x)=\frac\nu N\sum_{j=1,j\ne i}^N
          \frac{x_i-x_j}{|x_i-x_j|_\varepsilon^2},
 \qquad\varepsilon>0.
\end{equation}
The formal limiting operator as $\varepsilon$ tends to zero is $\Lambda=-\Delta+b\cdot\nabla$, i.e., the backward Kolmogorov operator of \eqref{ks}. We introduce two weights:
\[
 \psi_\varepsilon(x)=\prod_{i<j}|x_i-x_j|_\varepsilon^{-\nu/N}
\]
and
\begin{equation*}
 \varphi_{s,\varepsilon}(x)
 =\prod_{1 \leq i<j \leq N}\left(1+\frac{s}{|x_i-x_j|_\varepsilon^2}\right)^{\nu/(2N)}
 \geq 1,\qquad s>0.
\end{equation*}
The latter is a modification of $\psi_\varepsilon$, i.e.\,rescaled (to account for the diffusive scaling later) and bounded from below by a positive constant. The last property is needed in Theorem A. More generally, this property is needed for any weight that appears in the upper heat kernel bound - ``Nash initial estimate'' - of Theorem \ref{thm1}. 

In the proof, we are going to exploit the following algebraic relation between the two weights:
\begin{equation}
\label{phi_psi}
 \varphi_{s,\varepsilon}=\frac{\psi_\varepsilon}{\psi_{\varepsilon+s}}.
\end{equation}
Let us add that $\psi_\varepsilon$  is the density of a formal invariant measure of $\Lambda_\varepsilon$, that is,  $\Lambda_\varepsilon^*\psi_\varepsilon=0$.

For every fixed $\varepsilon>0$, the drift in \eqref{planar:operator} is bounded and smooth. Therefore, by the classical theory, the semigroup $e^{-t\Lambda_\varepsilon}$ has a (smooth) integral kernel, which we denote by $p_\varepsilon(t,x,y)$. 

\begin{theorem}\label{thm1}
Let $N\ge2$, and assume that the strength of attraction $\nu\ge0$ satisfies
\begin{equation}\label{alpha_max}
 \nu<\max_{1\le\beta<2}
 \left[
 \frac{N^{1-\beta}(N-1)^{1-\beta/2}}{2^{1+\beta}}
 \frac{\Gamma(\frac12-\frac\beta4)^2}
      {\Gamma(\frac12+\frac\beta4)^2}
 +\left(\frac{N-2}{N}\right)^{\beta/2}
 \right]^{-1/\beta}.
\end{equation}
Then the following assertions hold.
\begin{enumerate}[label=(\roman*)]
\item For every $T>0$, there is a constant $C=C(N,\nu,T)$, independent of $\varepsilon>0$, such that the a priori Nash initial estimate is valid:
\begin{equation*}
 0< p_\varepsilon(t,x,y)\le Ct^{-N}\varphi_{t,\varepsilon}(y),
 \qquad 0<t\le T,\quad x,y\in\mathbb R^{2N}.
\end{equation*}
\item The strong limit
\[
 T_t=s\mbox{-}L^2\mbox{-}\lim_{\varepsilon\downarrow 0}e^{-t\Lambda_\varepsilon}
\]
exists locally uniformly for $t\ge0$ and defines a positivity preserving holomorphic semigroup. We write $T_t=:e^{-t\Lambda}$, where the generator $\Lambda$ is thus a realization of the Kolmogorov operator of \eqref{ks} as the generator of a holomorphic semigroup in $L^2$. 

The semigroups $e^{-t\Lambda_\varepsilon}$, $e^{-t\Lambda}$  are $L^\infty$ contractions, so by an interpolation argument they admit extension from $L^2 \cap L^r$, $2 \leq r<\infty$, to $L^r$, as strongly continuous semigroups there.

\item For every $t>0$, the limiting semigroup has integral kernel $p(t,x,y)$,
$$
e^{-t\Lambda}f(x)=:\int_{\mathbb R^{2N}}p(t,x,y)f(y)dy,
$$
that satisfies a posteriori Nash initial estimate
\begin{equation*}
p(t,x,y)\le Ct^{-N}\varphi_t(y),\qquad 0<t\le T,
\end{equation*}
for a.e. $x,y \in \mathbb R^{2N}$, where
$$
 \varphi_{s}(x)
 =\prod_{1 \leq i<j \leq N}\left(1+\frac{s}{|x_i-x_j|^2}\right)^{\nu/(2N)},\qquad s>0.
$$
\end{enumerate}
\end{theorem}

\begin{remark}
The holomorphy of the approximating semigroups in Theorem \ref{thm1} is a part of the construction of the limiting semigroup: we need to venture in the complex plane since there is no uniform in $\varepsilon$ contractivity property of the approximating semigroups in any $L^p$, and so we cannot apply the Hille-Yosida theorem directly. The approximating semigroups are, in fact, contraction uniformly in $\varepsilon$ on the Bessel space $W^{\frac{1}{2},2}$, but this takes us out of the Lebesgue scale $L^p$ where we have to work in order to estimate the heat kernel ($=$\,integral kernel of the semigroup).
\end{remark}

Let us note that \eqref{alpha_max} implies $\nu<4$, and so the limiting weight $\varphi_s$ is locally summable. 

The previous bound allows us to apply a result of Coulhon-Sikora \cite{CS}, which then yields the following refinement of the upper bound in Theorem \ref{thm1}.

\begin{corollary}\label{cor:planar-gaussian}
Under the hypotheses of Theorem~\ref{thm1},
for every $T>0$, there is a constant $C_1$ independent of
$\varepsilon>0$ such that
\begin{equation}\label{pg:gaussian-polynomial-eps}
 p_\varepsilon(t,x,y)
 \le C_1t^{-N}\varphi_{t,\varepsilon}(y)
 \left(1+\frac{|x-y|^2}{4t}\right)^{N+\nu(N-1)/4}
 \exp\left(-\frac{|x-y|^2}{4t}\right)
\end{equation}
for $0<t\le T$ and $x,y\in\mathbb R^{2N}$. Therefore, for every
$c>4$, there is a constant $C$, independent of $\varepsilon$, such that
\begin{equation*}
 p_\varepsilon(t,x,y)
 \le C t^{-N}\varphi_{t,\varepsilon}(y)
       \exp\left(-\frac{|x-y|^2}{ct}\right).
\end{equation*}
The limiting heat kernel satisfies, for each $0<t\le T$ and a.e.\,$x,y \in \mathbb R^{2N}$,
\begin{equation*} 
 p(t,x,y) \le C_1t^{-N}\varphi_t(y)  \left(1+\frac{|x-y|^2}{4t}\right)^{N+\nu(N-1)/4}  \exp\left(-\frac{|x-y|^2}{4t}\right),
\end{equation*}
and 
\begin{equation*} 
p(t,x,y)\le C t^{-N}\varphi_t(y) \exp\left(-\frac{|x-y|^2}{ct}\right).
\end{equation*}
Here the constants in principle depend on $N,\nu,T$.
\end{corollary}

\subsection{Comments}\label{comment_sect}
1.~The condition on $\nu$ in Theorem~\ref{thm1} does not degenerate
as $N\to\infty$. In fact, if $\nu_N$ denotes the maximum in
\eqref{alpha_max}, then
\[
 \lim_{N\to\infty}\nu_N=1.
\]
To see this, let us fix $\beta\in ]4/3,2[$. The first term in the brackets in
\eqref{alpha_max} is of order $N^{2-3\beta/2}$ and tends to zero,
while the second tends to $1$. Hence
$\liminf_{N\to\infty}\nu_N\ge1$. Conversely, for $N\ge3$,
discarding the first (positive) term gives
$\nu_N\le\sqrt{N/(N-2)}$, establishing the reverse inequality for the
upper limit. Thus, every fixed $\nu<1$ is admissible for all
sufficiently large $N$.

There is also an explicit sufficient condition for \eqref{alpha_max} valid for all $N\ge2$. Taking $\beta=4/3$ in \eqref{alpha_max} and using
$(N-1)/N\le1$ and $(N-2)/N\le1$, we obtain
\[
 \nu<\left[1+2^{-7/3}
          \frac{\Gamma(1/6)^2}{\Gamma(5/6)^2}\right]^{-3/4}
        \approx 0.267.
\]
That said, we are still quite far from the critical value of $\nu$, which is $\nu_\star=4$, see \cite{FJ,FT,T}. We will partially rectify this situation in higher dimensions $d \geq 3$, see Theorem \ref{thm2}.

2.~The time dependence of the weight takes into account the diffusive scaling. In particular, $\varphi_t(y)\to1$ as $t\downarrow0$ whenever $y$ lies outside the collision sets, and so the upper bound of Corollary \ref{cor:planar-gaussian} recovers the delta-function in $y$ at $t=0$. 

3.~In the case of two particles, the law of the relative displacement $z=(x_1-x_2)/\sqrt2$ is described by the Kolmogorov operator $-\Delta+\frac{\nu}{2}|z|^{-2}z\cdot\nabla$. In the case $d \geq 3$, the two-sided heat kernel bounds for this operator were obtained, among many other results, in  \cite{MS,MSS,MNS}, see also \cite{Sturm} and \cite{FKS}. We discuss this in greater detail in Appendix \ref{app:two-particles}.

4.~In the proof of Theorem \ref{thm1} we develop a Sobolev regularity theory of $\Lambda_\varepsilon$, $\Lambda$ in the spirit of \cite{S, Ki_super,KiS-theory}.

5.~The process
\[
 R_t=\frac1{4N}\sum_{i,j}|X_t^i-X_t^j|^2
\]
is the squared Bessel equation
\[
 dR_t=2\sqrt{R_t}\,dW_t+(N-1)(2-\nu/2)\,dt;
\]
see \cite{F,FJ}. If $\delta=(N-1)(2-\nu/2)>0$, the corresponding Bessel transition density for positive arguments is
\begin{equation*}
 \frac yt\left(\frac yx\right)^{\delta/2-1}
 e^{-(x^2+y^2)/(2t)}I_{\delta/2-1}(xy/t),
 \qquad x,y>0,
\end{equation*}
where $I_{\delta/2-1}$ is the modified Bessel function, see \cite[Chapter XI]{RevuzYor}.  We are not aware of any explicit formulas for the heat kernel of the Keller-Segel particle system \eqref{ks}. It should be added that having an explicit formula for a heat kernel does not necessarily mean that it is easy to derive from it practical elementary estimates, such as the one in Theorem \ref{thm1}.

\subsection{An abstract desingularization theorem}
The a priori estimate in Theorem~\ref{thm1} is a consequence of the following general desingularization theorem. Let $X$ be a locally compact topological space, let $\mu$ be a $\sigma$-finite Radon measure, and let $L^\infty_{\mathrm{com}}$ denote the bounded compactly supported functions.
\begin{theoremA}[\cite{KSS}]
Let $e^{-t\Lambda}$ be a strongly continuous semigroup on $L^r(X,\mu)$ for some $r>1$, such that, for some $a,c>0$,
\begin{equation}\label{S1}\tag{$S_1$}
 \|e^{-t\Lambda}\|_{r\to\infty}\le ct^{-a/r},\qquad0<t\le T.
\end{equation}
Assume that there exists a weight $\varphi$ such that
\begin{equation}\label{S2}\tag{$S_2$}
 \varphi,\ \varphi^{-1}\in L^1_{\mathrm{loc}}(X,\mu),
\end{equation}
\begin{equation}\label{S3}\tag{$S_3$}
 \varphi\ge c_0>0,
\end{equation}
and
\begin{equation}\label{S4}\tag{$S_4$}
 \|\varphi e^{-t\Lambda}\varphi^{-1}f\|_1
 \le c_1\|f\|_1,
 \qquad 0<t\le T,\quad f\in L^\infty_{\mathrm{com}}.
\end{equation}
Then the semigroup has an integral kernel satisfying
\begin{equation}\label{nie}\tag{$N$}
 |e^{-t\Lambda}(x,y)|\le Ct^{-a}\varphi(y),
 \qquad0<t\le T,
\end{equation}
for $\mu$-a.e.\,$x,y \in X$.
\end{theoremA}

To make the paper self-contained, we give the proof of Theorem A in Appendix~\ref{app_A}. The proof uses a weighted version of the Coulhon-Raynaud extrapolation theorem, which is of independent interest. 

The weight in Theorem A is fixed and is time-independent. So, in the proof of Theorem \ref{thm1}, we will apply Theorem A with weight $\varphi_{s,\varepsilon}$ for fixed $s>0$ satsifying $0<t\le s$. In the end of the proof we set $s=t$ to arrive at the upper bound in Theorem \ref{thm1}.

In the proof of Theorem \ref{thm1} via Theorem A most of the work is in the verification of condition \eqref{S1}. The situation in the stable particles Section~\ref{nonlocal_sect} is different. Namely, to prove Theorem \ref{nl:main}, we will verify \eqref{S1} is a similar manner via a fractional resolvent representation, but the non-locality of the diffusion also makes the verification of \eqref{S4} non-trivial. 

\bigskip

\section{Stable-driven particles}\label{nonlocal_sect}

We now turn to the non-local counterpart \eqref{ks_stable}, in which the particles are driven by
independent isotropic $\alpha$-stable noises. Let $d\ge2$, $N\ge2$, $1<\alpha<2$, and
$\nu\ge0$. The corresponding backward Kolmogorov 
operator is
\begin{equation}\label{nl:operator}
 \Lambda_\alpha=A_\alpha+b_\alpha\cdot\nabla,
 \qquad A_\alpha=\sum_{i=1}^N(-\Delta_{x_i})^{\alpha/2},
 \qquad b_{\alpha,i}(x)=\frac\nu N\sum_{j\ne i}
                \frac{x_i-x_j}{|x_i-x_j|^\alpha}.
\end{equation}
As before, the sign of the drift in \eqref{nl:operator} corresponds to
attraction between the particles.

We will first work with regularized interaction kernels, i.e.\,with many-particle drift 
\begin{equation*}
 b_{\alpha,\varepsilon,i}(x)=\frac\nu N\sum_{j\ne i}
             \frac{x_i-x_j}{|x_i-x_j|_\varepsilon^\alpha},
 \qquad \Lambda_{\alpha,\varepsilon}=A_\alpha+b_{\alpha,\varepsilon}\cdot\nabla,
\end{equation*}
and obtain estimates with constants that are independent of the regularization, i.e.\,of $\varepsilon$

Let us recall the fractional
Hardy constant 
\[
 \kappa_{d,\alpha}
 =2^\alpha\left[\frac{\Gamma((d+\alpha)/4)}
                         {\Gamma((d-\alpha)/4)}\right]^2
\]
 and define a new constant that will enter the condition of Theorem \ref{thm2}:
 \[
 D_{d,\alpha}
 =\sup_{r\ge0}\frac{d-\alpha+dr}{(1+r)^{1+\alpha/2}}
 =\begin{cases}
 2\bigl(d/(\alpha+2)\bigr)^{1+\alpha/2},&d<\alpha+2,\\
 d-\alpha,&d\ge\alpha+2.
 \end{cases}
\]
For $s>0$, define the desingularizing weight
\begin{equation}\label{nl:weights}
 \varphi_{\alpha,s,\varepsilon}(x)
 =\prod_{1\le i<j\le N}
 \left(1+
 \left(\frac{s^{1/\alpha}}{|x_i-x_j|_\varepsilon}\right)^{(d-\alpha)/2}
 \right)^{\frac{3\nu(d-\alpha)}{2N\kappa_{d,\alpha}}},
 \qquad
 \varphi_{\alpha,s}=\varphi_{\alpha,s,0}.
\end{equation}
Let us emphasize that the weight depends on the strength of attraction $\nu$ between the particles.

\begin{theorem}\label{nl:main}
Assume that
\begin{equation}\label{nl:smallness}
 \frac{\nu(N-1)}{N\kappa_{d,\alpha}}
 \left(D_{d,\alpha}+\frac{d-\alpha}{2}\right)<1.
\end{equation}
Then the following assertions hold.
\begin{enumerate}[label=(\roman*)]
\item For every $\varepsilon>0$, the operator
$\Lambda_{\alpha,\varepsilon}$ with domain $W^{\alpha,2}$ generates a positivity preserving
strongly continuous contraction semigroup
$e^{-t\Lambda_{\alpha,\varepsilon}}$ on $L^2$.
It is contractive on $L^\infty$ and has a nonnegative integral kernel
satisfying the a priori Nash initial estimate
\begin{equation}\label{nl:kernel-eps}
 p_{\alpha,\varepsilon}(t,x,y)
 \le C t^{-dN/\alpha}\varphi_{\alpha,t,\varepsilon}(y),
 \qquad t>0
\end{equation}
for every $x,y \in \mathbb R^{dN}$, with $C$ independent of $\varepsilon>0$.

\item The strong limit
\[
 T_{\alpha,t}=s\mbox{-}L^2\mbox{-}\lim_{\varepsilon\downarrow0}
                    e^{-t\Lambda_{\alpha,\varepsilon}}
\]
exists locally uniformly in $t\ge0$ and defines a positivity preserving strongly continuous
contraction semigroup on $L^2$. Write $T_{\alpha,t}=:e^{-t\Lambda_\alpha}$. It is $L^\infty$ contraction. Similarly to Theorem \ref{thm1}, we also have semigroup convergence and strong continuity in $L^r$, $2 \leq r<\infty$.

\item The limiting semigroup has a nonnegative integral kernel satisfying the a posteriori Nash initial estimate
\begin{equation*}
 p_\alpha(t,x,y)\le C t^{-dN/\alpha}\varphi_{\alpha,t}(y),
 \qquad t>0
\end{equation*}
for a.e. $x,y \in \mathbb R^{dN}$. 
\end{enumerate}
\end{theorem}

\begin{remarks}
1.~It is possible to prove that both approximating and limiting semigroups in Theorem \ref{nl:main} are holomorphic (using \eqref{nl:drift-estimate}). But, in contrast to Theorem \ref{thm1}, in Theorem \ref{nl:main} the holomorphy is not needed to construct the semigroup. One can use instead the fact that the approximating semigroups are contractions in $L^2$, and so the usual Hille-Yosida theorem applies.

2.~Taking $N=2$, Theorem \ref{nl:main} recovers the $\alpha$-stable Nash initial estimate of \cite[Theorem 2]{KSS}, albeit under more restrictive condition on $\nu$. However, the present proof of this Nash initial estimate is more elementary than the proof in \cite{KSS}. Namely, \cite{KSS} used an equivalent  construction of weight $\varphi$, but its regularization required some operator-theoretic ideas, i.e.\,the mollifier defined in terms of $\Lambda_{\alpha,\varepsilon}$ in order to be able to commute operators in the verification of \eqref{S4}.

3. It has been demonstrated in \cite{KSS} ($N=2$) that in the non-local case the Nash initial estimate arguably plays a more important role than in the local case. Specifically, due to  the polynomial decay of the heat kernel of $(-\Delta)^{\alpha/2}$ ($1<\alpha<2$) at infinity, and its different scaling properties compared to the Gaussian density, in the case $0<\alpha<2$ only needs to supplement the Nash initial estimate with careful but elementary analysis of the Duhamel series in order to deduce both upper and lower heat kernel bounds.

4.~The proof of Theorem \ref{thm1} is not a special case ($\alpha=2$, $d=2$) of the proof of Theorem \ref{nl:main}. Indeed, although we use a fractional resolvent representation, it is different from the one in Theorem \ref{thm1} to account for a different fractional many-particle Hardy inequality (in fact, Lemma \ref{nl:drift-lemma}) that is, in turn, needed to keep the constraint on $\nu$ independent of $N$ when $1<\alpha<2$.
\end{remarks}

\bigskip

\section{Higher-dimensional Brownian particles}
\label{high_dim_sect}

Let $d\geq3$. We now turn to the higher-dimensional counterpart of the
two-dimensional Brownian system, keeping the same inverse-square interaction
between the particles:
\begin{equation*}
dX_t^i=-\frac{\nu}{N}\sum_{j\ne i}
 \frac{X_t^i-X_t^j}{|X_t^i-X_t^j|^2}\,dt+\sqrt2\,dB_t^i,
 \qquad 1\leq i\leq N.
\end{equation*}
where the $B^i$ are independent Brownian motions in $\mathbb R^d$. This is no longer the Coulomb interaction kernel, but there is still present the phenomenon of clustering of the particles if the strength of attraction $\nu$ is too large.

The Kolmogorov operator corresponding to this particle system has the
same expression as before:
\[
 \Lambda=-\Delta+b\cdot\nabla,
 \qquad b_i(x)=\frac{\nu}{N}\sum_{j\ne i}
                  \frac{x_i-x_j}{|x_i-x_j|^2}, \quad x^i,x^j \in \mathbb R^d.
\]
Set
\[
 \psi_\varepsilon(x)=\prod_{i<j}|x_i-x_j|_\varepsilon^{-\nu/N},
 \qquad |z|_\varepsilon=(|z|^2+\varepsilon)^{1/2}.
\]
The  approximating operators are  given by
\begin{equation*}
 \Lambda_\varepsilon=-\Delta+b_\varepsilon\cdot\nabla,
 \qquad b_{\varepsilon,i}(x)=\frac{\nu}{N}\sum_{j\ne i}
                    \frac{x_i-x_j}{|x_i-x_j|_\varepsilon^2}
             =-\nabla_i\log\psi_\varepsilon(x).
\end{equation*}
 We denote the heat kernel
of $\Lambda_\varepsilon$ by $p_\varepsilon$ and search for an upper bound in
terms of the same time-dependent weights as in the two-dimensional Brownian case:
\begin{equation*}
 \varphi_{t,\varepsilon}(y)
 =\frac{\psi_\varepsilon(y)}{\psi_{\varepsilon+t}(y)}
 =\prod_{i<j}\left(1+\frac{t}{|y_i-y_j|_\varepsilon^2}\right)^{\nu/(2N)},
 \qquad \varphi_t=\varphi_{t,0}.
\end{equation*}
The proof of the next theorem uses a
weighted Sobolev inequality and Moser's iterations, following \cite{MS}. We could apply again the result of Coulhon-Sikora as in Section \ref{main_sect}, but prefer to give a direct proof to keep the paper self-contained.

\begin{theorem}\label{thm2}
Let $d\geq3$, $N\geq2$, and $0\leq\nu<2(d-2)$.
Then the following assertions hold.
\begin{enumerate}
\item[(i)] For every $T>0$ there are constants $C,c>0$, independent
of $\varepsilon>0$, such that
\[
 p_\varepsilon(t,x,y)\leq 
                            Ct^{-\frac{dN}{2}} \varphi_{t,\varepsilon}(y)e^{-\frac{|x-y|^2}{ct}},
 \qquad 0<t\leq T,\quad x,y\in\mathbb R^{dN}.
\]
\item[(ii)] For every $r\geq2$ satisfying
\[
 r>\frac{4}{4-2\nu/(d-2)},
\]
the strong limit
\[
 T_t=s\mbox{-}L^r\mbox{-}\lim_{\varepsilon\downarrow0}
                                e^{-t\Lambda_\varepsilon}
\]
exists locally uniformly in $t\geq0$ and defines a strongly continuous
semigroup on $L^r(\mathbb R^{dN})$. These semigroups agree on intersections
of their $L^r$ spaces. We write $T_t=e^{-t\Lambda}$.
For every $t>0$ it has a nonnegative  integral kernel $p(t,x,y)$,
and for every $0<t\leq T$,
\begin{equation*}
 p(t,x,y)\leq Ct^{-\frac{dN}{2}} \varphi_{t}(y)e^{-\frac{|x-y|^2}{ct}}
 \quad\text{for a.e.} x,y \in\mathbb R^{dN}.
\end{equation*}
\end{enumerate}
Here the constants in principle depend on $d,N,\nu$ and $T$. 
\end{theorem}

\subsection{Comments}
1. Let $d\ge3$. We prove in the follow-up paper \cite{BKM} the lower heat kernel bound of the same form as the upper bound in Theorem \ref{thm2}:
$$
C_0t^{-\frac{dN}{2}} \varphi_{t,\varepsilon}(y)e^{-\frac{|x-y|^2}{c_0t}} \leq  p_\varepsilon(t,x,y), \quad x,y \in \mathbb R^{dN}, \quad t \in ]0,T].
$$
The proof in \cite{BKM} requires $\nu<(d-2)^2/5d$ for all $N \geq 2$. So, at least for small $\nu$, the upper bound in Theorem \ref{thm2} is sharp.

For $N=2$ the two-sided heat kernel bounds of \cite{MS,MSS,MNS} hold for large values of $\nu$, see the discussion in the appendix on the two-particle case.

2. Since $d\geq3$, we can use in the proof of Theorem \ref{thm2} the following many-particle Hardy inequality \cite{HHLT}:
\begin{equation}\label{hardy}
 \frac{(d-2)^2}{N}\sum_{i<j}\int_{\mathbb R^{dN}}
      \frac{|f(x)|^2}{|x_i-x_j|^2}\,dx
 \leq\int_{\mathbb R^{dN}}|\nabla f(x)|^2\,dx,
 \qquad f\in W^{1,2}(\mathbb R^{dN}).
\end{equation}
One can obtain this inequality, albeit with a constant smaller by a
factor of two, by adding up the ordinary Hardy inequalities, each in
its own particle variable. There is no nontrivial counterpart of this inequality in the
two-dimensional case. This is the reason why we cannot run the proof of Theorem \ref{thm2}, which uses rather classical arguments, to obtain Theorem \ref{thm1}.  

The constant in \eqref{hardy} is not the best possible one (still unknown at the moment of writing of this paper). More precisely, if $C_{d,N}$ denotes the best
constant in
\begin{equation}\label{multi_hardy}
 C_{d,N}\sum_{i<j}\int_{\mathbb R^{dN}}\frac{|f(x)|^2}{|x_i-x_j|^2}\,dx
 \leq\int_{\mathbb R^{dN}}|\nabla f|^2\,dx,
\end{equation}
then the authors of \cite{HHLT} prove
\begin{equation*}
 C_{d,N}\geq(d-2)^2\max\left\{\frac1N,
 \frac1{1+\sqrt{1+\frac{3(d-2)^2}{2(d-1)^2}(N-1)(N-2)}}\right\}.
\end{equation*}
This constant is not far from the best possible one, as is demonstrated by the upper bound $C_{d,N}\leq d(d-2)/N$ proved in
\cite{KV} using a probabilistic argument. 

3. The proof of Theorem~\ref{thm2} also extends to operators with a
bounded skew-symmetric matrix in the diffusion term, namely, to
operators of the form
\begin{equation}\label{ks7}
 -\nabla\cdot(I+C)\nabla
 -\frac{\nabla\psi}{\psi}\cdot(I+C)\nabla,
 \qquad C^{\mathsf T}=-C,
 \qquad C\in[L^\infty(\mathbb R^{dN})]^{dN\times dN}.
\end{equation}
We discuss the details and the corresponding conditions in
Appendix~\ref{C_app}. In the absence of interaction between the
particles, i.e.\ when $\nu=0$, the remaining distributional drift is
$\nabla C$, and the operator was studied by Osada~\cite{O}.
The result was subsequently extended by Qian-Xi~\cite{QX} to
skew-symmetric matrices with entries in $\mathrm{BMO}$. The weighted
argument allows us to retain the particle interaction in this setting.

\bigskip

\section{Proof of Theorem \ref{thm1}}\label{proof:planar}

\subsection{Preliminaries}

We first establish two preliminary lemmas.

By \eqref{alpha_max}, we may choose $\beta\in[1,2[$ such that
\begin{equation}\label{planar:delta}
 \delta:=\nu^\beta\left[
 \frac{N^{1-\beta}(N-1)^{1-\beta/2}}{2^{1+\beta}}
 \frac{\Gamma(\frac12-\frac\beta4)^2}
      {\Gamma(\frac12+\frac\beta4)^2}
 +\left(\frac{N-2}{N}\right)^{\beta/2}\right]<1.
\end{equation}
This inequality will ensure convergence of the Neumann series in the
resolvent representation in Lemma \ref{planar:resolvent-lemma}. (Recall that $\delta$ does not go to zero as the number of particles $N$ goes to infinity, see the discussion after Theorem \ref{thm1}.)

In what follows, we write
$$
b_0:=b,
$$
i.e.\,the limit $b_\varepsilon$, $\varepsilon \downarrow 0.$

\begin{lemma}[A fractional many-particle Hardy inequality]
\label{planar:hardy-lemma}
Let $N\ge2$, $\nu\ge0$, $1\le\beta<2$, and let $\delta$ be defined by \eqref{planar:delta} (not necessarily less than $1$).
Then, for every $\varepsilon\ge0$, 
\begin{equation}\label{b_g0}
 \big\||b_\varepsilon|^{\beta/2}(-\Delta)^{-\beta/4}
      \big\|_{2\to2}\le\sqrt\delta.
\end{equation}
\end{lemma}

Regarding other types of many-particle fractional Hardy inequalities, see \cite{FSe,FHLS}.

The proof of this lemma uses some intermediate results from \cite{HHLT} that are valid in dimension $d=2$ (the end result in \cite{HHLT}, i.e.\,the usual many-particle Hardy inequality, cf.\,\eqref{multi_hardy}, requires $d \geq 3$). Specifically, we follow 
\cite[Lemmas 3.2-3.6 and (4.8)]{HHLT}. 

\begin{proof}
We will prove the equivalent inequality
\begin{equation}\label{b_g}
 \int_{\mathbb R^{2N}}|b_\varepsilon|^\beta|g|^2\,dx
 \le\delta\|(-\Delta)^{\beta/4}g\|_2^2,
 \qquad g\in W^{\beta/2,2}(\mathbb R^{2N}).
\end{equation}
It suffices to prove this inequality on
$g\in C_c^\infty(\mathbb R^{2N})$, then the usual density argument yields \eqref{b_g}.

It will actually be easier to first prove \eqref{b_g} for the singular drift $b$. From there we will deduce \eqref{b_g}  for every $\varepsilon>0$ using an argument that is of interest on its own.

Recall the usual 
fractional Hardy inequality in $\mathbb R^2$:
\[
 \int_{\mathbb R^2}|z|^{-\beta}|h(z)|^2\,dz
 \le\kappa_{2,\beta}^{-1}
       \|(-\Delta)^{\beta/4}h\|_{L^2(\mathbb R^2)}^2,
 \qquad
 \kappa_{2,\beta}=2^\beta
 \frac{\Gamma(\frac12+\frac\beta4)^2}
      {\Gamma(\frac12-\frac\beta4)^2};
\]
see \cite[Lemma 2.7]{KPS}. 

\smallskip
Step 1. Write $r_{ij}=|x_i-x_j|$, and, for $i<j<k$, let
$A_{ijk}$ be the area of the triangle with vertices
$x_i,x_j,x_k$. 

\begin{claim}
\label{cl_b2}
\[
 |b|^2=P+V,
\]
where
\[
 \begin{split}
 P:=\frac{2\nu^2}{N^2}\sum_{i<j}r_{ij}^{-2},\quad V:=\frac{\nu^2}{N^2}\sum_{i<j<k}
       \frac{16 A_{ijk}^2}{r_{ij}^2r_{ik}^2r_{jk}^2}.
 \end{split}
\]
\end{claim}
\begin{proof}[Proof of Claim \ref{cl_b2}]
Indeed, recalling that
\[
 b_i=\frac{\nu}{N}\sum_{j\ne i}
       \frac{x_i-x_j}{r_{ij}^2},
\]
we evaluate
\[
\begin{aligned}
 |b|^2=\sum_{i=1}^N|b_i|^2=\frac{\nu^2}{N^2}
   \sum_{i=1}^N\sum_{j\ne i}\frac1{r_{ij}^2}+
   \frac{2\nu^2}{N^2}
   \sum_{i=1}^N
   \sum_{\substack{j<k\\j,k\ne i}}
   \frac{(x_i-x_j)\cdot(x_i-x_k)}
        {r_{ij}^2r_{ik}^2}.
\end{aligned}
\]
In the first sum each unordered pair $i,j$ occurs twice,
once with $i$ as the outer summation index and once with $j$.
This is $P$.

Regarding the second sum in the expansion of $|b|^2$, which we denote by $V$, we have
\[
\begin{aligned}
 V=\frac{2\nu^2}{N^2}\sum_{i<j<k}
 \biggl[\frac{(x_i-x_j)\cdot(x_i-x_k)}
       {r_{ij}^2r_{ik}^2}+\frac{(x_j-x_i)\cdot(x_j-x_k)}
        {r_{ij}^2r_{jk}^2}+\frac{(x_k-x_i)\cdot(x_k-x_j)}
        {r_{ik}^2r_{jk}^2}.
 \biggr].
\end{aligned}
\]
Indeed, to see this, fix an unordered triple of parameters in the original definition of $V$, say, $\{p,q,r\}$,  such that $p<q<r$. If the outer index $i$ is $p$, then the other indices are $(j,k)=(q,r)$. If $i$ is $q$ , then $(j,k)=(p,r)$. If $i$ is $r$, then $(j,k)=(p,q)$. (Note that the order $j<k$ is preserved.) Thus, there are three terms in the original definition with these parameters, which yields the three-term representation for $V$.

Our task now is to express $V$ in terms of $A_{ijk}$.
Let us write
\[
 a_1=r_{ij}^2,\qquad
 a_2=r_{ik}^2,\qquad
 a_3=r_{jk}^2.
\]
It is easily seen that
\[
\begin{aligned} 
2(x_i-x_j)\cdot(x_i-x_k)&=a_1+a_2-a_3\\
2(x_j-x_i)\cdot(x_j-x_k)&=a_1+a_3-a_2,\\
 2(x_k-x_i)\cdot(x_k-x_j)&=a_2+a_3-a_1.
\end{aligned}
\]
Therefore, the bracket in the previous three-term representation of $V$, times two, equals
\[
\begin{aligned}
 &\frac{a_1+a_2-a_3}{a_1a_2}
 +\frac{a_1+a_3-a_2}{a_1a_3}
 +\frac{a_2+a_3-a_1}{a_2a_3}\\
 &\quad=
 \frac{a_3(a_1+a_2-a_3)
       +a_2(a_1+a_3-a_2)
       +a_1(a_2+a_3-a_1)}
      {a_1a_2a_3}\\
 &\quad=
 \frac{2(a_1a_2+a_1a_3+a_2a_3)-a_1^2-a_2^2-a_3^2}
      {a_1a_2a_3}.
\end{aligned}
\]
By
\begin{equation}
\label{heron_id}
2(a_1a_2+a_1a_3+a_2a_3)-a_1^2-a_2^2-a_3^2=16 A_{ijk}^2
\end{equation}
(Heron's identity),
the numerator in the last expression is
$16A_{ijk}^2$, and so we arrive at the required form of $V$.
\end{proof}

\smallskip

Step 2. We now establish a three particle estimate. Put $$S_{ijk}=r_{ij}^2+r_{ik}^2+r_{jk}^2.$$

\begin{claim}
\label{cl1} We have, outside of the collision sets,
\[
 \frac{16 A_{ijk}^2}{r_{ij}^2r_{ik}^2r_{jk}^2}
 \le\frac9{S_{ijk}}.
\]
\end{claim}
\begin{proof}[Proof of Claim~\ref{cl1}]
We put as above $a_1=r_{ij}^2$, $a_2=r_{ik}^2$, $a_3=r_{jk}^2$,
and write $S=a_1+a_2+a_3=S_{ijk}$.
By our assumption, $a_1,a_2,a_3>0$.
The sought inequality is equivalent to
\[
 16 A_{ijk}^2S\le 9a_1a_2a_3.
\]
By \eqref{heron_id}
\[
 16 A_{ijk}^2
 =2(a_1a_2+a_1a_3+a_2a_3)-a_1^2-a_2^2-a_3^2
 =4(a_1a_2+a_1a_3+a_2a_3)-S^2,
\]
so we need to show that
\[
\biggl( 9a_1a_2a_3-16 A_{ijk}^2S
 =\biggr) \;\;S^3-4S(a_1a_2+a_1a_3+a_2a_3)+9a_1a_2a_3
 \ge0.
\]
This expression is symmetric in $a_1,a_2,a_3$, so we may assume
that $a_1\ge a_2\ge a_3$. Expanding and then grouping the first two
terms gives
\[
\begin{aligned}
 &S^3-4S(a_1a_2+a_1a_3+a_2a_3)+9a_1a_2a_3\\
 &\quad=
 a_1(a_1-a_2)(a_1-a_3)+a_2(a_2-a_1)(a_2-a_3)+a_3(a_3-a_1)(a_3-a_2)\\
 &\quad=
 (a_1-a_2)\bigl[a_1(a_1-a_3)-a_2(a_2-a_3)\bigr]
 +a_3(a_1-a_3)(a_2-a_3)\\
 &\quad=
 (a_1-a_2)^2(a_1+a_2-a_3)+a_3(a_1-a_3)(a_2-a_3).
\end{aligned}
\]
Both terms in the last line are nonnegative: indeed,
$a_1-a_3\ge0$, $a_2-a_3\ge0$, and $a_1+a_2-a_3\ge a_1>0$.
This proves the required inequality. 
\end{proof}

For a fixed triple, introduce the orthonormal coordinates
\[
 u=\frac{x_i-x_j}{\sqrt2},\qquad
 v=\frac{x_i+x_j-2x_k}{\sqrt6},\qquad
 z=\frac{x_i+x_j+x_k}{\sqrt3}.
\]
The variables $u,v$ together range over $\mathbb R^4$, and
$S_{ijk}=3(|u|^2+|v|^2)$. Applying the usual Hardy inequality in
$\mathbb R^4$, whose constant is $1$, and integrating over $z$ and
the other particle variables, we obtain
\[
 \begin{split}
 \int_{\mathbb R^{2N}}\frac{|g|^2}{S_{ijk}}\,dx
 &\le\frac13\bigl(\|\nabla_u g\|_2^2+\|\nabla_v g\|_2^2\bigr)\\
 &=\frac19\left(
 \| (\nabla_{x_i}-\nabla_{x_j})g\|_2^2
 +\| (\nabla_{x_i}-\nabla_{x_k})g\|_2^2
 +\| (\nabla_{x_j}-\nabla_{x_k})g\|_2^2\right).
 \end{split}
\]
Note that each unordered pair belongs to exactly $N-2$ triples. Summing the
preceding estimate and applying Claim \ref{cl1} for each triple
therefore gives
\[
 \begin{split}
 \int_{\mathbb R^{2N}}V|g|^2\,dx
 \le\frac{\nu^2(N-2)}{N^2}
       \sum_{i<j}\|(\nabla_{x_i}-\nabla_{x_j})g\|_2^2\le\nu^2\frac{N-2}{N}\|\nabla g\|_2^2,
 \end{split}
\]
where the last step follows from 
$
 \sum_{i<j}\|(\nabla_{x_i}-\nabla_{x_j})g\|_2^2
 =N\sum_i\|\nabla_{x_i}g\|_2^2
   -\left\|\sum_i\nabla_{x_i}g\right\|_2^2.
$
(If $N=2$, then the triple sum is empty and so $V=0$.)

We next pass to the fractional power needed in \eqref{b_g}. Using a density argument and Fatou's lemma, we extend the preceding estimate 
to $g \in W^{1,2}(\mathbb R^{2N})$. Thus, regarding $V$ as a
nonnegative self-adjoint multiplication operator on
$L^2(\mathbb R^{2N})$, we have
\[
 V\le \nu^2\frac{N-2}{N}(-\Delta)
\]
in the sense of quadratic forms. Next, we will need a well-known operator inequality.

\begin{lemma}[Heinz inequality, see e.g.\,{\cite[Corollary~12.15]{SSV}}]
\label{lem:heinz}
Let $A$ and $B$ be nonnegative self-adjoint operators. If
\[
 D(B^{1/2})\subset D(A^{1/2}),
 \qquad
 \|A^{1/2}f\|^2\le \|B^{1/2}f\|^2,
 \quad f\in D(B^{1/2}),
\]
then, for every $0<\theta<1$,
\[
 D(B^{\theta/2})\subset D(A^{\theta/2}),
 \qquad
 \|A^{\theta/2}f\|^2\le \|B^{\theta/2}f\|^2,
 \quad f\in D(B^{\theta/2}).
\]
\end{lemma}

Thus, $A\le B$ implies $A^\theta\le B^\theta$, where both
inequalities are in the sense of quadratic forms.
Applying the Heinz
inequality
therefore gives
\begin{equation}
\label{V_est_final}
 \int_{\mathbb R^{2N}}V^{\beta/2}|g|^2\,dx
 \le
 \nu^\beta\left(\frac{N-2}{N}\right)^{\beta/2}
 \|(-\Delta)^{\beta/4}g\|_2^2,
 \qquad
 g\in W^{\beta/2,2}(\mathbb R^{2N}).
\end{equation}

\medskip

Step 3. We now establish a two-particle estimate. The previous step was purely local. Now we will take advantage of the fact that, if one works in the non-local setting (still on the Euclidean space), then dimension $d=2$ becomes accessible - one of the main observations of this paper.
Namely, since $0<\beta/2<1$, we obtain by subadditivity of the corresponding power function
\begin{equation}
\label{P_ineq}
 P^{\beta/2}
 \le 2^{\beta/2}\left(\frac\nu N\right)^\beta
       \sum_{i<j}r_{ij}^{-\beta}.
\end{equation}
For a fixed pair, let us use the orthonormal coordinates
$u=(x_i-x_j)/\sqrt2$, $z=(x_i+x_j)/\sqrt2$, while keeping the other
particle coordinates unchanged. The two-dimensional fractional Hardy
inequality applied in $u$ yields
\[
 \int_{\mathbb R^{2N}}r_{ij}^{-\beta}|g|^2\,dx
 \le\frac{2^{-\beta/2}}{\kappa_{2,\beta}}
       \|(-\Delta_u)^{\beta/4}g\|_2^2
 =\frac{2^{-\beta}}{\kappa_{2,\beta}}  \frac{1}{(2\pi)^{2N}}
       \int_{\mathbb R^{2N}}
          |\xi_i-\xi_j|^\beta|\widehat g(\xi)|^2\,d\xi.
\]
There are $N(N-1)/2$  pairs in \eqref{P_ineq}, so
\[
 \begin{split}
 \sum_{i<j}|\xi_i-\xi_j|^\beta
 &\le\left(\frac{N(N-1)}2\right)^{1-\beta/2}
       \left(\sum_{i<j}|\xi_i-\xi_j|^2\right)^{\beta/2}\\
 &\le\left(\frac{N(N-1)}2\right)^{1-\beta/2}
       N^{\beta/2}|\xi|^\beta,
 \end{split}
\]
where we have used
$\sum_{i<j}|\xi_i-\xi_j|^2
 =N|\xi|^2-|\sum_i\xi_i|^2\le N|\xi|^2$.
Combining these estimates and using Plancherel's theorem yields
\[
 \begin{split}
 \int_{\mathbb R^{2N}}P^{\beta/2}|g|^2\,dx
 &\le\frac{2^{-\beta/2}\nu^\beta}{N^\beta\kappa_{2,\beta}}
       \left(\frac{N(N-1)}2\right)^{1-\beta/2}N^{\beta/2}
       \|(-\Delta)^{\beta/4}g\|_2^2\\
 &=\frac{\nu^\beta}{2\kappa_{2,\beta}}
       N^{1-\beta}(N-1)^{1-\beta/2}
       \|(-\Delta)^{\beta/4}g\|_2^2.
 \end{split}
\]
Finally, noting that
\[
 |b|^\beta
 =(P+V)^{\beta/2}
 \le P^{\beta/2}+V^{\beta/2}
\]
and adding up the previous estimate and \eqref{V_est_final}, we arrive at inequality \eqref{b_g} for $b$.

\smallskip

Step 4. We now deduce estimate \eqref{b_g} for $b_\varepsilon$ from what we have just proved for the singular vector field
$b$. (Above we could also work directly with $b_\varepsilon$, but then we would have to handle $\varepsilon$ everywhere throughout Steps 1-3.) We will exploit a nice fact that the regularization in
\eqref{planar:operator} is a time average of the De Giorgi mollification of drift $b$. That is, with $-\Delta$
acting on $\mathbb R^{2N}$, we have

\begin{claim}
\label{cl3}
\[
 b_\varepsilon=\int_0^\infty e^{-s}
                 e^{\frac{\varepsilon}{8s}\Delta}b\,ds.
\]
\end{claim}

We give proof in Appendix \ref{app:elementary}.

Since $\beta\ge1$, the function $v\mapsto|v|^\beta$ on
$\mathbb R^{2N}$ is convex, so Jensen's inequality yields
\begin{equation}
\label{b_eps}
 |b_\varepsilon|^\beta
 \le\int_0^\infty e^{-s}
          e^{\frac{\varepsilon}{8s}\Delta}|b|^\beta\,ds.
\end{equation}
Let us note that, for every $t>0$, 
\[
 \begin{split}
 &\int_{\mathbb R^{2N}}
       (e^{t\Delta}|b|^\beta)(x)|g(x)|^2\,dx\\
 &\quad=\int_{\mathbb R^{2N}}\Gamma_t(y)
       \int_{\mathbb R^{2N}}|b(z)|^\beta|g(z+y)|^2\,dz\,dy\\
 &\quad\le\delta\int_{\mathbb R^{2N}}\Gamma_t(y)
       \|(-\Delta)^{\beta/4}g(\,\cdot+y)\|_2^2\,dy
       =\delta\|(-\Delta)^{\beta/4}g\|_2^2.
 \end{split}
\]
Here we used translation invariance of the fractional Sobolev norm
and $\int_{\mathbb R^{2N}}\Gamma_t=1$. Now, applying the last estimate in \eqref{b_eps} and averaging in $s$ yields
\eqref{b_g} for $\varepsilon>0$. A standard approximation argument extends \eqref{b_g} to $W^{\beta/2,2}$ and ends the proof of the fractional Hardy inequality for $|b_\varepsilon|$, i.e.\,Lemma \ref{planar:hardy-lemma}.
\end{proof}

\begin{remark}
Note that in order  to prove Lemma \ref{planar:hardy-lemma} we cannot simply apply the Heinz inequality to the many-particle Hardy inequality of \cite{HHLT}, cf.\,\eqref{multi_hardy}, since the latter requires $d \geq 3$.
\end{remark}

Having Lemma \ref{planar:hardy-lemma} at hand, we can establish a fractional representation of the
resolvent that will be used both in the verification of the dispersion estimate \eqref{S1} in Theorem A and for
the construction of the a posteriori (limiting) heat kernel. This fractional representation is a variant of the representations used in  \cite{S,Ki_super}; recall that we keep $\beta$ variable (will chose it as a function of $N$) in order to make the constraint on $\nu$ less restrictive, see discussion in the introduction.

The two
operators entering the representation are
\[
 \begin{split}
 R_\varepsilon&=
 b_\varepsilon|b_\varepsilon|^{-1+\beta/2}
       \cdot\nabla (\mu-\Delta)^{-1/2-\beta/4},\\
 Q_\varepsilon&=(\mu-\Delta)^{-1/2+\beta/4}|b_\varepsilon|^{1-\beta/2}.
 \end{split}
\]

\begin{lemma}\label{planar:resolvent-lemma}
For every $\varepsilon>0$ and $\operatorname{Re}\mu>0$, the operators
$R_\varepsilon,Q_\varepsilon$ extend to bounded operators on $L^2$ and satisfy
\begin{equation}\label{planar:RQ-bounds}
 \|R_\varepsilon\|_{2\to2}\le\delta^{1/2},\qquad
 \|Q_\varepsilon\|_{2\to2}
       \le\delta^{(2-\beta)/(2\beta)},\qquad
 \|Q_\varepsilon R_\varepsilon\|_{2\to2}
       \le\delta^{1/\beta}<1.
\end{equation}
The resolvent of $\Lambda_\varepsilon$, with domain $W^{2,2}$ in $L^2$,
has the representation
\begin{equation}\label{lambda_repr}
 (\mu+\Lambda_\varepsilon)^{-1}
 =(\mu-\Delta)^{-1/2-\beta/4}
   (I+Q_\varepsilon R_\varepsilon)^{-1}
   (\mu-\Delta)^{-1/2+\beta/4},
\end{equation}
and obeys the bound
\begin{equation}\label{eta_est}
 \|(\mu+\Lambda_\varepsilon)^{-1}\|_{2\to2}
 \le\frac1{1-\delta^{1/\beta}}\frac1{|\mu|},
 \qquad \operatorname{Re}\mu>0.
\end{equation}
Moreover, the semigroups $e^{-t\Lambda_\varepsilon}$
are bounded and holomorphic on $L^2$, and satisfy
\begin{equation}\label{planar:holomorphic}
 \|e^{-t\Lambda_\varepsilon}\|_{2\to2}\le C,
 \qquad
 \|\Lambda_\varepsilon e^{-t\Lambda_\varepsilon}\|_{2\to2}\le Ct^{-1},
\end{equation}
with $C$ independent of $\varepsilon>0$ and $t>0$.
\end{lemma}

\begin{remark}
If the energy argument in $L^2$ (i.e.\,multiply the parabolic equation by the solution itself and integrate by parts) would apply to our setting, then the first $L^2 \rightarrow L^2$ bound in \eqref{planar:holomorphic} would be straightforward. But the energy argument in $L^2$ does not apply, and this estimate is not straightforward: it forces us to use the non-classical fractional resolvent representation and, furthermore, to work in the complex plane; this follows from the observations in \cite{S,Ki_super}.
\end{remark}

\begin{proof}
We first note that the
heat-kernel representation of a negative fractional power gives
\[
 |(\mu-\Delta)^{-\beta/4}(x,y)|
 \le(-\Delta)^{-\beta/4}(x,y).
\]
Consequently \eqref{b_g0} holds with $(\mu-\Delta)^{-\beta/4}$ in place of
$(-\Delta)^{-\beta/4}$. Applying this scalar estimate to each component
of $\nabla (\mu-\Delta)^{-1/2}f$, and using
\[
 \|\nabla (\mu-\Delta)^{-1/2}f\|_2^2
 = \frac{1}{(2\pi)^{2N}} \int_{\mathbb R^{2N}}\frac{|\xi|^2}{|\mu+|\xi|^2|}|\widehat f(\xi)|^2\,d\xi
 \le\|f\|_2^2,
\]
gives the asserted bound for $R_\varepsilon$. To estimate
$Q_\varepsilon$, we use the Heinz inequality to reduce the power of
the drift in \eqref{b_g0}. With exponent
$(2-\beta)/\beta\in(0,1]$, it yields
\[
 \big\||b_\varepsilon|^{1-\beta/2}
           (-\Delta)^{-1/2+\beta/4}\big\|_{2\to2}
 \le\delta^{(2-\beta)/(2\beta)}.
\]
The same domination of the heat kernels, followed by duality, now
gives the bound for $Q_\varepsilon$. Multiplication of the two bounds
completes the proof of \eqref{planar:RQ-bounds}.

The inverse in \eqref{lambda_repr} is well defined by a convergent
Neumann series, since \eqref{planar:RQ-bounds} gives a bound strictly
less than one for $Q_\varepsilon R_\varepsilon$. We also need to verify
that this bounded expression is the resolvent of the operator on
$W^{2,2}$ (rather than just a formal factorization). If $u$ denotes its
value at $f\in L^2$, then the leftmost factor shows that
$u\in W^{1+\beta/2,2}\subset W^{1,2}$, and the defining identity is
\[
 (\mu-\Delta)^{1/2+\beta/4}u
 +(\mu-\Delta)^{-1/2+\beta/4}(b_\varepsilon\cdot\nabla u)
 =(\mu-\Delta)^{-1/2+\beta/4}f.
\]
Because $b_\varepsilon$ is bounded, we may multiply this identity by
$(\mu-\Delta)^{1/2-\beta/4}$ in the sense of distributions and obtain
$(\mu-\Delta)u+b_\varepsilon\cdot\nabla u=f$. The regularity already established
for $u$ makes the drift term an $L^2$ function. It follows that
$(\mu-\Delta)u\in L^2$, and hence $u\in W^{2,2}$. Conversely, the same identity
applied to a function in $W^{2,2}$ yields injectivity. This verifies
\eqref{lambda_repr} on the operator domain.

Estimate \eqref{eta_est} follows from \eqref{lambda_repr} and
\eqref{planar:RQ-bounds} since
\[
 \|(\mu-\Delta)^{-1/2-\beta/4}\|_{2\to2}
 \|(\mu-\Delta)^{-1/2+\beta/4}\|_{2\to2}\le |\mu|^{-1}.
\]
In detail,
\begin{align*}
\|(\mu+\Lambda_\varepsilon)^{-1}\|_{2 \rightarrow 2} &\le \|(\mu-\Delta)^{-1}\|_{2\rightarrow 2} \nonumber \\ 
& \qquad +\|(\mu-\Delta)^{-\frac{1}{2}-\frac{\beta}{4}}\|_{2 \rightarrow 2}
\|Q_\varepsilon\|_{2 \rightarrow 2}\|(I+Q_\varepsilon R_\varepsilon)^{-1}\|_{2 \rightarrow 2}
\|R_\varepsilon\|_{2 \rightarrow 2} \|(\mu-\Delta)^{-\frac{1}{2}+\frac{\beta}{4}}\|_{2 \rightarrow 2}, \nonumber \\
&\le \frac{1}{|\mu|} +|\mu|^{-\frac{1}{2}-\frac{\beta}{4}}
\delta^{\frac{2-\beta}{2\beta}} \frac{1}{1-\delta^{1/\beta}} \delta^{\frac{1}{2}}\, |\mu|^{-\frac{1}{2}+\frac{\beta}{4}},  \nonumber \\
&=\frac{1}{1-\delta^{1/\beta} } \frac{1}{|\mu|}.    
\end{align*}
The standard representation of a holomorphic semigroup as a contour integral of the resolvent now gives
\eqref{planar:holomorphic}, see \cite[Ch.\,IX, Sect\,10, implications III $\Rightarrow$ I,II]{Y}, or see \cite{Ki_super}.
\end{proof}

In the next two sections we prove assertion (\textit{i}) of the theorem, i.e.\,for the regularized interaction kernels, by verifying conditions of Theorem A.

\subsection{Verification of \eqref{S1}}

We are going to combine an intermediate $L^2$ to $L^r$ estimate with the
$L^\infty$ contraction property, using the Coulhon-Raynaud
extrapolation theorem (Theorem~\ref{thmE2} in Appendix \ref{app_A}). The intermediate estimate that we need is
\begin{equation}\label{planar:intermediate}
 \|e^{-t\Lambda_\varepsilon}\|_{2\to r}\le Ct^{-1/4},
 \qquad r=\frac{4N}{2N-1}.
\end{equation}
To obtain it, we will use, following ideas of \cite{S}, the representation \eqref{lambda_repr}. So, the
regularity of the semigroup will be deduced from that of the
resolvent. We have
$$
 \|(\mu-\Delta)^{-1/2-\beta/4} \|_{2 \rightarrow r} \le C\mu^{-\frac{1}{4}-\frac{\beta}{4}}.
$$
Also,
$$
 \|(\mu-\Delta)^{-1/2+\beta/4} \|_{2 \rightarrow 2} \le \mu^{-\frac{1}{2}+\frac{\beta}{4}}.
$$
Applying this in \eqref{lambda_repr} yields
\begin{equation}\label{planar:resolvent-smoothing}
 \|(\mu+\Lambda_\varepsilon)^{-1}\|_{2\to r}
 \le C\mu^{-1/4-\beta/4}\mu^{-1/2+\beta/4}
 =C\mu^{-3/4}.
\end{equation}
We now pass from the resolvent to the semigroup.
We have
\[
 e^{-t\Lambda_\varepsilon}f
 =(\mu+\Lambda_\varepsilon)^{-1}
   \bigl(\mu e^{-t\Lambda_\varepsilon}f
         +\Lambda_\varepsilon e^{-t\Lambda_\varepsilon}f\bigr).
\]
By 
\eqref{planar:holomorphic}, the norm of the expression in the brackets is bounded by
$C(\mu+t^{-1})\|f\|_2$. Hence
\eqref{planar:resolvent-smoothing} implies
\[
 \|e^{-t\Lambda_\varepsilon}f\|_r
 \le C\mu^{-3/4}(\mu+t^{-1})\|f\|_2.
\]
Choosing $\mu=t^{-1}$ now yields
\[
 \|e^{-t\Lambda_\varepsilon}f\|_r
 \le Ct^{3/4}t^{-1}\|f\|_2
 =Ct^{-1/4}\|f\|_2,
\]
which proves \eqref{planar:intermediate}.

To obtain the sought $L^2$ to $L^\infty$ estimate, we apply the
Coulhon-Raynaud theorem (Appendix~\ref{app_A}), to the adjoint semigroup: the
$L^\infty$ contraction property of $e^{-t\Lambda_\varepsilon}$ gives an $L^1$
bound equal to one, while the dual of \eqref{planar:intermediate} is
\[
 \|(e^{-t\Lambda_\varepsilon})^*\|_{r'\to2}\le Ct^{-1/4},
 \qquad r'=\frac{4N}{2N+1}.
\]
Thus the exponents in the Coulhon-Raynaud theorem are $p=1$, $q=r'$, and its final
exponent is $2$. It follows that
\[
 \|(e^{-t\Lambda_\varepsilon})^*\|_{1\to2}
 \le Ct^{-\frac{1/4}{1/2N}}=Ct^{-N/2}.
\]
Returning to the original
semigroup, we obtain by duality
\begin{equation*}
 \|e^{-t\Lambda_\varepsilon}\|_{2\to\infty}\le Ct^{-N/2}.
\end{equation*}
We have therefore verified \eqref{S1} with $r=2$ and $a=N$.

\subsection{Verification of \eqref{S4}}

The scale $s>0$ is fixed.
Recall that
\[
 \varphi_{s,\varepsilon}
 =\frac{\psi_\varepsilon}{\psi_{\varepsilon+s}},
 \qquad
 \nabla\log\psi_\varepsilon=-b_\varepsilon.
\]
For fixed $s,\varepsilon>0$, both the weight and its derivatives are
bounded, and $\varphi_{s,\varepsilon}\ge1$. The following lemma is the key to verifying condition \eqref{S4}:

\begin{lemma}
\label{planar:weight-lemma}
Let $N\ge2$ and $\nu\ge0$. For all $s,\varepsilon>0$, the weight
$\varphi_{s,\varepsilon}=\psi_\varepsilon/\psi_{\varepsilon+s}$ satisfies the following
pointwise bound on $\mathbb R^{2N}$:
\begin{equation}\label{planar:weight-supersolution}
 \Lambda_\varepsilon^*\varphi_{s,\varepsilon}
 \ge-\frac Cs\varphi_{s,\varepsilon},
 \qquad C=(N-1)\left(2\nu+\frac{\nu^2}{2}\right).
\end{equation}
The constant $C$ is independent of $s$ and $\varepsilon$.
\end{lemma}

\begin{proof}
We use the following identities:
$\Lambda_\varepsilon^*=-\Delta-\operatorname{div}(b_\varepsilon\,\cdot)$
and 
\[
 \frac{\Delta\varphi_{s,\varepsilon}}{\varphi_{s,\varepsilon}}
 =\Delta\log\varphi_{s,\varepsilon}
   +|\nabla\log\varphi_{s,\varepsilon}|^2,
 \qquad
 \frac{\operatorname{div}(b_\varepsilon\varphi_{s,\varepsilon})}
      {\varphi_{s,\varepsilon}}
 =\operatorname{div}b_\varepsilon
   +b_\varepsilon\cdot\nabla\log\varphi_{s,\varepsilon}.
\]
Since $\varphi_{s,\varepsilon}=\psi_\varepsilon/\psi_{\varepsilon+s}$,
we also have
\[
 \nabla\log\varphi_{s,\varepsilon}=-b_\varepsilon+b_{\varepsilon+s},
 \qquad
 \Delta\log\varphi_{s,\varepsilon}
 =-\operatorname{div}b_\varepsilon+\operatorname{div}b_{\varepsilon+s},
\]
where, throughout
the proof, $\operatorname{div}$ is the divergence in all $2N$ variables.
We obtain
\[
 \begin{split}
 \frac{\Lambda_\varepsilon^*\varphi_{s,\varepsilon}}
      {\varphi_{s,\varepsilon}}
 &=-\Delta\log\varphi_{s,\varepsilon}
   -|\nabla\log\varphi_{s,\varepsilon}|^2
   -\operatorname{div}b_\varepsilon
   -b_\varepsilon\cdot\nabla\log\varphi_{s,\varepsilon}\\
 &=\operatorname{div}b_\varepsilon-\operatorname{div}b_{\varepsilon+s}
   -|b_\varepsilon-b_{\varepsilon+s}|^2
   -\operatorname{div}b_\varepsilon
   +b_\varepsilon\cdot(b_\varepsilon-b_{\varepsilon+s}).
 \end{split}
\]
The two terms containing $\operatorname{div}b_\varepsilon$ cancel out, while
\[
 -|b_\varepsilon-b_{\varepsilon+s}|^2
 +b_\varepsilon\cdot(b_\varepsilon-b_{\varepsilon+s})
 =(b_\varepsilon-b_{\varepsilon+s})\cdot b_{\varepsilon+s}.
\]
Thus
\begin{equation}\label{planar:weight-identity}
 \frac{\Lambda_\varepsilon^*\varphi_{s,\varepsilon}}
      {\varphi_{s,\varepsilon}}
 =-\operatorname{div}b_{\varepsilon+s}
   +(b_\varepsilon-b_{\varepsilon+s})\cdot b_{\varepsilon+s}.
\end{equation}
We will estimate the two terms on the right separately. We note that
$
 \operatorname{div}b_{\varepsilon+s}
 =\sum_{i=1}^N\operatorname{div}_{x_i}b_{\varepsilon+s,i}.
$
For the first term, differentiation of a single interaction in
$\mathbb R^2$ gives
\[
 \operatorname{div}_{z}\frac{z}{|z|^2+\varepsilon+s}
 =\frac{2}{|z|^2+\varepsilon+s}
  -\frac{2|z|^2}{(|z|^2+\varepsilon+s)^2}
 =\frac{2(\varepsilon+s)}{(|z|^2+\varepsilon+s)^2}.
\]
Each expression on the right is between zero and
$2/(\varepsilon+s)$. There are $N(N-1)$ ordered pairs $(i,j)$ with
$i\ne j$. Therefore,
\[
 \begin{split}
 0\le\operatorname{div}b_{\varepsilon+s}
 &=\frac\nu N\sum_i\sum_{j\ne i}
   \frac{2(\varepsilon+s)}{(|x_i-x_j|^2+\varepsilon+s)^2}\\
 &\le\frac\nu N\,N(N-1)\frac{2}{\varepsilon+s}
 \le\frac{2\nu(N-1)}s.
 \end{split}
\]
This gives the required lower bound for the negative divergence term
in \eqref{planar:weight-identity}.

The second term in \eqref{planar:weight-identity} contains $b_\varepsilon$, which may become large near
a collision as $\varepsilon$ tends to zero. To control its contribution, we first evaluate
\[
 b_{\varepsilon,i}-b_{\varepsilon+s,i}
 =\frac\nu N\sum_{j\ne i}
 \frac{s(x_i-x_j)}
      {(|x_i-x_j|^2+\varepsilon)(|x_i-x_j|^2+\varepsilon+s)}.
\]
In the scalar product with $b_{\varepsilon+s}$, the two ordered
pairs have the same denominator. Their numerators can be combined by
using
\[
 (x_i-x_j)\cdot b_{\varepsilon+s,i}
 +(x_j-x_i)\cdot b_{\varepsilon+s,j}
 =(x_i-x_j)\cdot
       (b_{\varepsilon+s,i}-b_{\varepsilon+s,j}).
\]
Thus the sum over ordered pairs becomes a sum over $i<j$:
\begin{equation}
\label{ord_sum}
 \begin{split}
 &(b_\varepsilon-b_{\varepsilon+s})\cdot b_{\varepsilon+s}\\
 &\quad=\frac\nu N\sum_{i<j}
 \frac{s(x_i-x_j)\cdot
       (b_{\varepsilon+s,i}-b_{\varepsilon+s,j})}
      {(|x_i-x_j|^2+\varepsilon)(|x_i-x_j|^2+\varepsilon+s)}.
 \end{split}
\end{equation}

\begin{claim}
\label{cl4}
\[
 \left|\frac{z}{|z|^2+\varepsilon+s}
       -\frac{w}{|w|^2+\varepsilon+s}\right|
 \le\frac{|z-w|}{\varepsilon+s}.
\]
\end{claim}

\begin{proof}[Proof of Claim \ref{cl4}]
We use identity
\[
\begin{aligned}
\bigl|\,|w|^2z-|z|^2w\,\bigr|^2
&=|w|^4|z|^2+|z|^4|w|^2
  -2|z|^2|w|^2(z\cdot w)\\
&=|z|^2|w|^2\bigl(|z|^2+|w|^2-2z\cdot w\bigr)\\
&=|z|^2|w|^2|z-w|^2
\end{aligned}
\]
to estimate:
\[
\begin{aligned}
\left|\frac{z}{|z|^2+\varepsilon+s}
      -\frac{w}{|w|^2+\varepsilon+s}\right|
&=\frac{|(\varepsilon+s)(z-w)+|w|^2z-|z|^2w|}
 {(|z|^2+\varepsilon+s)(|w|^2+\varepsilon+s)}
\\
&\le\frac{(\varepsilon+s+|z|\,|w|)|z-w|}
 {(|z|^2+\varepsilon+s)(|w|^2+\varepsilon+s)}
\\
&\le\frac{(\varepsilon+s+|z|\,|w|)|z-w|}
 {(\varepsilon+s)(\varepsilon+s+|z|\,|w|)}
\\
&=\frac{|z-w|}{\varepsilon+s},
\end{aligned}
\]
where the last inequality follows from 
\[
\begin{aligned}
(|z|^2+\varepsilon+s)(|w|^2+\varepsilon+s)
&=(\varepsilon+s)^2
  +(\varepsilon+s)(|z|^2+|w|^2)
  +|z|^2|w|^2\\
&\ge(\varepsilon+s)(\varepsilon+s+|z|\,|w|).
\end{aligned}
\]
\end{proof}

To apply Claim \ref{cl4}, we write the difference of the two drift
components explicitly:
\[
 \begin{split}
 b_{\varepsilon+s,i}-b_{\varepsilon+s,j}
 =\frac\nu N\bigg(&
 \frac{2(x_i-x_j)}{|x_i-x_j|^2+\varepsilon+s}\\
 &+\sum_{k\ne i,j}\left[
 \frac{x_i-x_k}{|x_i-x_k|^2+\varepsilon+s}
 -\frac{x_j-x_k}{|x_j-x_k|^2+\varepsilon+s}
 \right]\bigg).
 \end{split}
\]
The first term comes from the interaction of particles $i$ and $j$, and
occurs twice since the interaction kernel is odd. For every
remaining particle $k$, the preceding Lipschitz estimate applies with
$z=x_i-x_k$ and $w=x_j-x_k$, whose difference is $x_i-x_j$.
There are $N-2$ such particles, so
\[
 \begin{split}
 |b_{\varepsilon+s,i}-b_{\varepsilon+s,j}|
 &\le\frac\nu N\left(
 \frac{2|x_i-x_j|}{\varepsilon+s}
 +(N-2)\frac{|x_i-x_j|}{\varepsilon+s}\right)\\
 &=\frac{\nu|x_i-x_j|}{\varepsilon+s}.
 \end{split}
\]
Applying this in \eqref{ord_sum}, we obtain
\[
 \begin{split}
 \big|(b_\varepsilon-b_{\varepsilon+s})\cdot b_{\varepsilon+s}\big|
 &\le\frac{\nu^2}N\sum_{i<j}
 \frac{s|x_i-x_j|^2}
 { (\varepsilon+s)(|x_i-x_j|^2+\varepsilon)
                    (|x_i-x_j|^2+\varepsilon+s)}.
 \end{split}
\]
The extra factor $|x_i-x_j|$ provided by the difference of the drift
components removes the singularity, namely,
\[
 \frac{s|x_i-x_j|^2}
 { (\varepsilon+s)(|x_i-x_j|^2+\varepsilon)
                    (|x_i-x_j|^2+\varepsilon+s)}
 \le\frac{s}{(\varepsilon+s)^2}\le\frac1s.
\]
Here we used $|x_i-x_j|^2/(|x_i-x_j|^2+\varepsilon)\le1$ and
$|x_i-x_j|^2+\varepsilon+s\ge\varepsilon+s\ge s$.
Since there are $N(N-1)/2$ pairs with $i<j$, this yields
\[
 \big|(b_\varepsilon-b_{\varepsilon+s})\cdot b_{\varepsilon+s}\big|
 \le\frac{\nu^2}N\frac{N(N-1)}{2s}
 =\frac{\nu^2(N-1)}{2s}.
\]
Combining this estimate with the divergence bound in
\eqref{planar:weight-identity}, we conclude that
\[
 \frac{\Lambda_\varepsilon^*\varphi_{s,\varepsilon}}
      {\varphi_{s,\varepsilon}}
 \ge-\frac{2\nu(N-1)}s-\frac{\nu^2(N-1)}{2s}.
\]
Multiplication by the positive function $\varphi_{s,\varepsilon}$
yields \eqref{planar:weight-supersolution} and ends the proof of Lemma \ref{planar:weight-lemma}.
\end{proof}

Let us now verify the ``desingularizing'' estimate \eqref{S4}. 

\begin{lemma}\label{planar:weighted-lemma}
For all $s,\varepsilon>0$, we have
\begin{equation*}
 \|\varphi_{s,\varepsilon}e^{-t\Lambda_\varepsilon}
                     \varphi_{s,\varepsilon}^{-1}\|_{1\to1}
 \le e^{Ct/s},\qquad t\ge0,
\end{equation*}
where $C$ is the constant in
\eqref{planar:weight-supersolution}. Thus, \eqref{S4} holds
for $0<t\le s$ with $c_1=e^C$, independently of $s$ and $\varepsilon$.
\end{lemma}

\begin{proof}
We fix $s,\varepsilon>0$ and write
$\varphi=\varphi_{s,\varepsilon}$. Let
$0\le f\in C_c^\infty(\mathbb R^{2N})$ and put
$u(t)=e^{-t\Lambda_\varepsilon}\varphi^{-1}f$. Thus $u(t)\ge0$ and
\[
 \partial_tu=\Delta u-b_\varepsilon\cdot\nabla u,
 \qquad u(0)=\varphi^{-1}f.
\]
Our goal is to estimate the weighted integral
$\int_{\mathbb R^{2N}}\varphi u(t)\,dx$ whose value at $t=0$ is
$\|f\|_1$. In what follows, we integrate by parts, taking into account that, for each fixed $\varepsilon>0$, the Gaussian  bounds on $u$ and $|\nabla u|$ and the fact that $\varphi^{-1}f$ has compact support ensure the latter rapidly converge to zero as $|x| \rightarrow \infty$ locally uniformly in $t>0$. Namely, multiplying the equation by
$\varphi$ and integrating, we obtain
\[
 \begin{split}
 \frac d{dt}\int_{\mathbb R^{2N}}\varphi u(t)\,dx
 &=\int_{\mathbb R^{2N}}\varphi
          (\Delta u(t)-b_\varepsilon\cdot\nabla u(t))\,dx\\
 &=-\int_{\mathbb R^{2N}}\nabla\varphi\cdot\nabla u(t)\,dx
   +\int_{\mathbb R^{2N}}u(t)
                         \operatorname{div}(\varphi b_\varepsilon)\,dx\\
 &=\int_{\mathbb R^{2N}}u(t)
       \bigl(\Delta\varphi+\operatorname{div}(\varphi b_\varepsilon)\bigr)\,dx\\
 &=-\int_{\mathbb R^{2N}}u(t)\Lambda_\varepsilon^*\varphi\,dx.
 \end{split}
\]
 We can now apply
\eqref{planar:weight-supersolution}: since $u(t)\ge0$, 
\[
 \frac d{dt}\int_{\mathbb R^{2N}}\varphi u(t)\,dx
 \le\frac Cs\int_{\mathbb R^{2N}}\varphi u(t)\,dx.
\]
Therefore, by Gronwall's inequality,
\[
 \|\varphi e^{-t\Lambda_\varepsilon}\varphi^{-1}f\|_1
 =\int_{\mathbb R^{2N}}\varphi u(t)\,dx
 \le e^{Ct/s}\|f\|_1,
\]
which is the sought estimate on $0\leq f \in C_c^\infty(\mathbb R^{2N})$. Using $
 |\varphi e^{-t\Lambda_\varepsilon}\varphi^{-1}f|
 \le\varphi e^{-t\Lambda_\varepsilon}\varphi^{-1}|f|$ and applying a standard density argument, we extend this estimate to all $f \in L^1$, as needed.
\end{proof}

\subsection{Application of Theorem A}
Conditions \eqref{S2} and \eqref{S3} follow from boundedness of $\varphi_{s,\varepsilon}$ and from $\varphi_{s,\varepsilon}\ge1$.
Together with \eqref{S1} and \eqref{S4}, they verify conditions of Theorem A, which yields a priori heat kernel bound
\[
 p_\varepsilon(t,x,y)
 \le Ct^{-N}\varphi_{s,\varepsilon}(y),\qquad 0<t\le s.
\]
It remains to take $s:=t$. This ends the proof of assertion (\textit{i}) of Theorem \ref{thm1}.

\subsection{A posteriori heat kernel bound}

We turn to assertions (\textit{ii}) and (\textit{iii}) of the theorem. We will use
Trotter's approximation theorem, see Appendix~\ref{app_B}, by arguing as in \cite{S,Ki_super,KiS-theory}.

$1^\circ)$ We have
\[
 (\mu+\Lambda_\varepsilon)^{-m}f
 =\frac1{(m-1)!}\int_0^\infty
       t^{m-1}e^{-\mu t}e^{-t\Lambda_\varepsilon}f\,dt,
 \qquad m=1,2,\dots
\]
Indeed, the case $m=1$ needs no comment. Assuming the formula for $m$, we apply
$(\mu+\Lambda_\varepsilon)^{-1}$ once more:
\[
\begin{aligned}
 (\mu+\Lambda_\varepsilon)^{-(m+1)}f
 &=\frac1{(m-1)!}
   \int_0^\infty\int_0^\infty
   t^{m-1}e^{-\mu(t+s)}
   e^{-s\Lambda_\varepsilon}e^{-t\Lambda_\varepsilon}f
   \,ds\,dt\\
 &=\frac1{(m-1)!}
   \int_0^\infty\int_0^\infty
   t^{m-1}e^{-\mu(t+s)}
   e^{-(t+s)\Lambda_\varepsilon}f\,ds\,dt\\
 &=\frac1{(m-1)!}
   \int_0^\infty e^{-\mu r}e^{-r\Lambda_\varepsilon}f
        \left(\int_0^r t^{m-1}\,dt\right)\,dr\\
 &=\frac1{m!}\int_0^\infty
       r^m e^{-\mu r}e^{-r\Lambda_\varepsilon}f\,dr.
\end{aligned}
\]
In the third line we use $r=t+s$, so that $0<t<r$, and also employed identity
\[
 \int_0^\infty\int_0^\infty
 t^{m-1}e^{-\mu(t+s)}\,ds\,dt
 =\frac{(m-1)!}{\mu^{m+1}}<\infty.
\]

Hence, using the first estimate in \eqref{planar:holomorphic}, we obtain
\[
\begin{aligned}
 \|(\mu+\Lambda_\varepsilon)^{-m}f\|_2
 &\le\frac{C}{(m-1)!}
       \int_0^\infty t^{m-1}e^{-\mu t}\,dt\,\|f\|_2\\
 &=C\mu^{-m}\|f\|_2.
\end{aligned}
\]
Therefore,
\[
 \sup_{\varepsilon>0}
 \|(\mu+\Lambda_\varepsilon)^{-m}\|_{2\to2}
 \le C\mu^{-m},
 \qquad \mu>0,\quad m=1,2,\dots
\]

$2^\circ)$ Let us verify the strong convergence condition of Trotter's approxmation theorem. We fix $\mu>0$. We have
$b_\varepsilon\to b$ pointwise away from the collision sets. It is easily seen that
\begin{equation*}
 |b_{\varepsilon,i}(x)|
 \le\frac\nu N\sum_{j\ne i}\frac1{|x_i-x_j|}.
\end{equation*}
The right-hand side has locally integrable powers of
every order strictly less than two. The powers of the drift appearing
in $Q_\varepsilon$ and $R_\varepsilon$ were chosen so that this
integrability is sufficient. In particular, the dominated convergence theorem
yields
\[
 |b_\varepsilon|^{1-\beta/2}f
 \longrightarrow |b|^{1-\beta/2}f\quad\text{in }L^2,
 \qquad f\in C_c^\infty.
\]
Applying the free resolvent factor shows that
$Q_\varepsilon\to Q$ strongly on this dense set. The corresponding
verification for $R_\varepsilon$ is made on the range of
$(\mu-\Delta)^{1/2+\beta/4}$, where the negative power in its definition cancels:
we have
\[
 R_\varepsilon (\mu-\Delta)^{1/2+\beta/4}f
 =b_\varepsilon|b_\varepsilon|^{-1+\beta/2}\cdot\nabla f
 \longrightarrow b|b|^{-1+\beta/2}\cdot\nabla f
 \quad\text{in }L^2.
\]
Here the squared majorant is locally integrable because $\beta<2$.
The space $(\mu-\Delta)^{1/2+\beta/4}(C_c^\infty)$ is dense in $L^2$, since
$C_c^\infty$ is dense in $W^{1+\beta/2,2}$ in its graph norm.
The uniform bounds \eqref{planar:RQ-bounds} therefore extend both
convergences to all of $L^2$. They also allow us to pass to the limit
in the products. Moreover, since the bound for these products is
strictly less than one, the tails of the Neumann series tend to zero
in operator norm uniformly in $\varepsilon$. It follows that
$(I+Q_\varepsilon R_\varepsilon)^{-1}$ converges strongly to
$(I+QR)^{-1}$, and \eqref{lambda_repr} gives the required strong
resolvent convergence.

$3^\circ)$ 
We next verify the ``approximation of the identity'' condition of Trotter's theorem, i.e.
\begin{equation}\label{planar:approximate-identity}
 \lim_{\mu\to\infty}\sup_{\varepsilon>0}
 \|\mu(\mu+\Lambda_\varepsilon)^{-1}f-f\|_2=0,
 \qquad f\in L^2.
\end{equation}
A crucial feature of the resolvent representation here is that the
gradient in its final drift factor can be moved onto the initial
function. That is, the operator identity
\[
 (I+Q_\varepsilon R_\varepsilon)^{-1}
 =I-Q_\varepsilon(I+R_\varepsilon Q_\varepsilon)^{-1}R_\varepsilon
\]
can be used to rewrite \eqref{lambda_repr} as
\[
 (\mu+\Lambda_\varepsilon)^{-1}-(\mu-\Delta)^{-1}
 =-(\mu-\Delta)^{-1/2-\beta/4}Q_\varepsilon
   (I+R_\varepsilon Q_\varepsilon)^{-1}
   R_\varepsilon (\mu-\Delta)^{-1/2+\beta/4}.
\]
For $f\in C_c^\infty$, the gradient in the last factor commutes with
the free resolvent and can therefore be placed on $f$. We may then
apply \eqref{b_g0}, with $\mu-\Delta$ in place of $-\Delta$, to obtain
\[
 \begin{split}
 \|R_\varepsilon (\mu-\Delta)^{-1/2+\beta/4}f\|_2
 &=\|b_\varepsilon|b_\varepsilon|^{-1+\beta/2}
                       \cdot (\mu-\Delta)^{-1}\nabla f\|_2\\
 &\le\sqrt\delta\,\mu^{-1+\beta/4}\|\nabla f\|_2.
 \end{split}
\]
The other factors have the bounds in \eqref{planar:RQ-bounds}, so
\begin{equation*}
 \|\mu[(\mu+\Lambda_\varepsilon)^{-1}-(\mu-\Delta)^{-1}]f\|_2
 \le C\mu^{1/2-\beta/4}\mu^{-1+\beta/4}\|\nabla f\|_2
 =C\mu^{-1/2}\|\nabla f\|_2.
\end{equation*}
The factor $\mu^{-1/2}$ is what makes the approximation uniform in
$\varepsilon$. Since $\mu(\mu-\Delta)^{-1}f\to f$ in $L^2$, we have
proved \eqref{planar:approximate-identity} on $C_c^\infty$. The 
bound \eqref{eta_est} then can be used to extend \eqref{planar:approximate-identity} to all of $L^2$ by a density argument.

Trotter's approximation theorem  now yields a strongly continuous on $L^2$ semigroup
$e^{-t\Lambda}$ such that
\[
 e^{-t\Lambda_\varepsilon}f\longrightarrow e^{-t\Lambda}f\quad\text{in }L^2,
\]
locally uniformly in $t \geq 0$. The properties established for the approximating semigroups
are inherited by this limit: positivity preservation, the $L^\infty$ contraction
property and the uniform $L^2$ bound follow from the above strong convergence.
The holomorphy follows from the corresponding resolvent bound with $\frac{1}{|\mu|}$ (let us add that at this stage, i.e.\,after the limiting semigroup has been constructed, the holomorphy is only a property, i.e.\,we do not use it in the rest of the proof).
Interpolation  gives consistent bounded semigroups on $L^r$,
$2\le r<\infty$. In fact, for bounded $f$ in $L^2$,
\[
 \|e^{-t\Lambda_\varepsilon}f-e^{-t\Lambda}f\|_r^r
 \le(2\|f\|_\infty)^{r-2}
       \|e^{-t\Lambda_\varepsilon}f-e^{-t\Lambda}f\|_2^2,
\]
then the above convergence extends to all $f \in L^r$ by density. The same interpolation argument yields strong continuity in $L^r$.

To end the proof, we need to construct and estimate the integral kernel of $e^{-t\Lambda}$. By the a priori estimate in (\textit{i}), we have
$$
\|e^{-t\Lambda_\varepsilon }f\|_{\infty} \leq Ct^{-N}\|\varphi_{t,\varepsilon} f\|_1, \quad t \in ]0,T], \quad f \in C_c(\mathbb R^{2N}).
$$
It follows that
$$
\langle e^{-t\Lambda_\varepsilon}f,g\rangle \leq Ct^{-N}\langle \varphi_{t,\varepsilon} f\rangle ||g||_1, \quad f,g \text{ are additionally non-negative}.
$$
Using the strong convergence of the semigroups as $\varepsilon \downarrow 0$ on the left, and the convergence $\varphi_{t,\varepsilon} \uparrow \varphi$ on the right, we obtain
\begin{equation}
\label{ineq:e_varphi}
\langle e^{-t\Lambda}f,g\rangle \leq Ct^{-N}\langle \varphi_{t} f\rangle||g||_1.
\end{equation}
It follows that $$
\|e^{-t\Lambda }f\|_{\infty} \leq Ct^{-N}\|\varphi_t f\|_1, \quad t \in ]0,T], 
$$
so by the Dunford-Pettis theorem there exists Borel measurable integral kernel $p(t,x,y)$ of $e^{-t\Lambda}$. Now, choosing in \eqref{ineq:e_varphi} functions $f=f_n$, $g=g_n$ that constitute $\delta$-sequences, and applying the Lebesgue differentiation theorem, we arrive at the a posteriori Nash initial estimate $p(t,x,y) \leq Ct^{-N}\varphi_t(y)$ for a.e.\,$x \in \mathbb R^{2N}$. \hfill \qed

\bigskip

\section{Proof of Corollary~\ref{cor:planar-gaussian}}
\label{proof:planar-gaussian}

Armed with the Nash initial estimate on the heat kernel of Theorem \ref{thm1}, we now apply a result of Coulhon-Sikora \cite[Theorem 4.4]{CS} to make the upper bound more precise, i.e.\,to replace $t^{-N}$ by the Gaussian density in $\mathbb R^{2N}$.

\begin{theorem}[Coulhon-Sikora]\label{pg:cs-variable}
Let $H$ be a nonnegative self-adjoint operator on
$L^2(\mathbb R^m)$ satisfying the Davies estimate:
\begin{equation}\label{pg:DG-abstract}
 |\langle e^{-tH}f,g\rangle|
 \le \exp\left(-\frac{\operatorname{dist}(\operatorname{supp}f,
                         \operatorname{supp}g)^2}{4t}\right)
       \|f\|_2\|g\|_2.
\end{equation}
Suppose the integral kernel $h_z(x,y)$ of $e^{-zH}$, $\operatorname{Re}z>0$, is continuous
in $x,y \in \mathbb R^m$. Let $\mathcal V(x,r)>0$ be a function that is continuous in $x$,
nondecreasing in $r>0$, and satisfying the doubling property
\begin{equation*}
 \frac{\mathcal V(x,R)}{\mathcal V(x,r)}
 \le K\left(\frac Rr\right)^D,
 \qquad R\ge r>0,
\end{equation*}
with some $K\ge1$ and $D\ge0$. Then the ``Nash initial estimate''
\begin{equation*}
 h_t(x,x)\le C \frac{1}{\mathcal V(x,\sqrt t)},
 \qquad t>0,
\end{equation*}
yields, whenever $|x-y|^2\ge4t$,
\begin{equation}\label{pg:CS-conclusion}
 |h_t(x,y)|
 \le \frac{eCK}
 {\sqrt{\mathcal V(x,|x-y|/2)\mathcal V(y,|x-y|/2)}}
 \left(\frac{|x-y|^2}{4t}\right)^D
 \exp\left(-\frac{|x-y|^2}{4t}\right).
\end{equation}
\end{theorem}

Fix $\varepsilon>0$. We are going to apply this theorem to operator
\begin{equation}\label{pg:ground-state-operator}
 H_\varepsilon
 :=-\Delta+\frac{\Delta\sqrt{\psi_\varepsilon}}
                       {\sqrt{\psi_\varepsilon}}
 =\sqrt{\psi_\varepsilon}\,\Lambda_\varepsilon
                         \frac1{\sqrt{\psi_\varepsilon}},
\end{equation}
where we have used $b_\varepsilon=-\nabla\log\psi_\varepsilon$
(the first and the third are operators of multiplication by a function). 

1.~The potential in the resulting Schr\"{o}dinger operator is real-valued and bounded (and smooth), so
 $H_\varepsilon$ is self-adjoint on $W^{2,2}(\mathbb R^{2N})$.
Its nonnegativity follows from the identity
\begin{equation}\label{pg:ground-state-form}
 \int_{\mathbb R^{2N}}\left(|\nabla v|^2+
       \frac{\Delta\sqrt{\psi_\varepsilon}}{\sqrt{\psi_\varepsilon}}
         |v|^2\right)dx
 =\int_{\mathbb R^{2N}}\psi_\varepsilon
       \left|\nabla\left(\frac{v}{\sqrt{\psi_\varepsilon}}\right)\right|^2dx
 \ge0.
\end{equation}
Integration by parts give us this for $v\in C_c^\infty$. It extends
to $W^{1,2}$ by a density argument since the integrand on the right equals
$|\nabla v-\tfrac12v\nabla\log\psi_\varepsilon|^2$ and
$\nabla\log\psi_\varepsilon$ is bounded.

For
$f\in C_c^\infty$, the expression
$\sqrt{\psi_\varepsilon}\,e^{-t\Lambda_\varepsilon}
(f/\sqrt{\psi_\varepsilon})$ solves the Cauchy problem for
$H_\varepsilon$ with initial value $f$, by
\eqref{pg:ground-state-operator}. Indeed, multiplication by
$\sqrt{\psi_\varepsilon}$ preserves $W^{2,2}$, and
$f/\sqrt{\psi_\varepsilon}$ is smooth and compactly supported.
Hence
\begin{equation*}
 e^{-tH_\varepsilon}f
 =\sqrt{\psi_\varepsilon}\,e^{-t\Lambda_\varepsilon}
                           (f/\sqrt{\psi_\varepsilon}).
\end{equation*}
Writing $h_\varepsilon$ for the kernel of $e^{-tH_\varepsilon}$,
 we obtain
\begin{equation}\label{pg:kernel-identification}
 h_\varepsilon(t,x,y)
 =\sqrt{\frac{\psi_\varepsilon(x)}{\psi_\varepsilon(y)}}
       p_\varepsilon(t,x,y),
 \qquad
 h_\varepsilon(t,x,x)=p_\varepsilon(t,x,x).
\end{equation}

For fixed $\varepsilon>0$, the smooth bounded potential gives, by the classical theory, a
continuous heat kernel.
The identity
\[
 h_\varepsilon(t+i\sigma,x,y)
 =\left\langle
 e^{-i\sigma H_\varepsilon}h_\varepsilon(t/2,y,\cdot),
 h_\varepsilon(t/2,x,\cdot)\right\rangle,
 \qquad t>0,
\]
gives continuity in $(x,y)$ at complex times as well (using the unitarity of $ e^{-i\sigma H_\varepsilon}$; since $\varepsilon>0$, we are still within the classical theory).

2.~Let us verify the Davies estimate for $H_\varepsilon$. Let $F$ be a bounded real Lipschitz function, and let $f\in L^2$.
Put
\[
 u(t)=e^{-tH_\varepsilon}e^{-F}f,\qquad v(t)=e^Fu(t).
\]
For $t>0$, we have $u(t)\in W^{2,2}$ and
$u'(t)=-H_\varepsilon u(t)$. Since $e^F$ is a bounded (in $L^2$) multiplier,
we can differentiate in $L^2$ and obtain
\[
 \frac12\frac d{dt}\|v(t)\|_2^2
 =-\operatorname{Re}\langle H_\varepsilon u(t),e^{2F}u(t)\rangle.
\]
The term on the right can be calculated using integration by parts.
Indeed, writing
$u=e^{-F}v$, we have
\[
 \begin{split}
 \nabla u\cdot\nabla(e^{2F}\overline u)
 &=(\nabla v-v\nabla F)\cdot
                         (\nabla\overline v+\overline v\nabla F)\\
 &=|\nabla v|^2-|v|^2|\nabla F|^2
   +(\overline v\nabla v-v\nabla\overline v)\cdot\nabla F.
 \end{split}
\]
The last term is purely imaginary. Also,
$u\,e^{2F}\overline u=|v|^2$, so the potential contributes exactly
as in \eqref{pg:ground-state-form}. Taking real parts therefore gives
\[
 \begin{split}
 \operatorname{Re}\langle H_\varepsilon u,e^{2F}u\rangle
 &=\int_{\mathbb R^{2N}}\left(|\nabla v|^2+
       \frac{\Delta\sqrt{\psi_\varepsilon}}{\sqrt{\psi_\varepsilon}}
            |v|^2\right)dx
       -\int_{\mathbb R^{2N}}|\nabla F|^2|v|^2\,dx\\
 &=\int_{\mathbb R^{2N}}\psi_\varepsilon
       \left|\nabla\left(\frac{v}{\sqrt{\psi_\varepsilon}}\right)\right|^2dx
       -\int_{\mathbb R^{2N}}|\nabla F|^2|v|^2\,dx\\
 &\ge-\|\nabla F\|_\infty^2\|v\|_2^2.
 \end{split}
\]
We thus have
\[
 \frac d{dt}\|v(t)\|_2^2
 \le2\|\nabla F\|_\infty^2\|v(t)\|_2^2.
\]
Integrating from a positive time and then letting that time tend
to zero, using strong continuity and $v(0)=f$, yields
\begin{equation}\label{pg:twisted-L2}
 \|e^Fe^{-tH_\varepsilon}e^{-F}\|_{2\to2}
 \le e^{t\|\nabla F\|_\infty^2}.
\end{equation}

Above we were assuming that $F$ is additionally bounded, in order to justify the integration by parts.
In the general case, if $F$ is a globally Lipschitz but not necessarily bounded, set
\[
 F_n(x)=\max\{-n,\min\{F(x),n\}\},\qquad n\ge1.
\]
Then $F_n$ is bounded and Lipschitz,
$F_n(x)\to F(x)$ for every $x$, and
$\|\nabla F_n\|_\infty\le\|\nabla F\|_\infty$.
Fix $t>0$ and first take $f\in L^2$ with compact support.
Since $F$ is bounded on this support,
$e^{-F_n}f=e^{-F}f$ for all sufficiently large $n$.
Thus $u(t)=e^{-tH_\varepsilon}e^{-F}f$ is well defined and,
by \eqref{pg:twisted-L2} applied to $F_n$,
\[
 \|e^{F_n}u(t)\|_2^2
 \le e^{2t\|\nabla F\|_\infty^2}\|f\|_2^2
\]
for all sufficiently large $n$. Applying Fatou's lemma, we obtain \eqref{pg:twisted-L2} for an arbitrary globally Lipschitz $F$.

We now choose $F$ depending on the supports of $f,g\in L^2$. Let
$$r=\operatorname{dist}(\operatorname{supp}f,\operatorname{supp}g)>0$$
and, for $\lambda\ge0$, take
\[
 F(z)=\lambda\min\{\operatorname{dist}(z,\operatorname{supp}f),r\}.
\]
The distance function has global Lipschitz constant equal to $1$. Hence $F$ is bounded and
$\|\nabla F\|_\infty\le\lambda$. Moreover, $F=0$ on the support of
$f$ and $F=\lambda r$ on the support of $g$. Consequently
$e^{-F}f=f$ and $e^{-F}g=e^{-\lambda r}g$, and
\eqref{pg:twisted-L2} gives
\[
 \begin{split}
 |\langle e^{-tH_\varepsilon}f,g\rangle|
 &=|\langle e^Fe^{-tH_\varepsilon}e^{-F}f,e^{-F}g\rangle|\\
 &\le e^{t\lambda^2-\lambda r}\|f\|_2\|g\|_2.
 \end{split}
\]
Finally,
$t\lambda^2-\lambda r=t(\lambda-r/(2t))^2-r^2/(4t)$.
Selecting $\lambda=r/(2t)$ proves the Davies estimate
\eqref{pg:DG-abstract}. (If $r=0$, the estimate follows directly from
$L^2$ contractivity.)

3.~We can now apply the Coulhon-Sikora theorem.
Fix $T>0$. By \eqref{pg:kernel-identification} and Theorem~\ref{thm1} (use \eqref{phi_psi}),
\begin{equation*}
 h_\varepsilon(t,x,x)
 \le Ct^{-N}\frac{\psi_\varepsilon(x)}{\psi_{\varepsilon+t}(x)},
 \qquad 0<t\le T,
\end{equation*}
with $C$ independent of $\varepsilon$. Strictly speaking, to apply Theorem \ref{pg:cs-variable} directly, we need global upper bound in Theorem \ref{thm1}, although the proofs in \cite{CS} work on finite time intervals as well, see remark of Coulhon and Sikora after Example 4.19 in their paper. Still, we can apply Theorem \ref{pg:cs-variable} by adding positive constant term $\frac{N}{T}$ to operator $H_\varepsilon$, in which case we get
\begin{equation*}
 e^{-Nt/T}h_\varepsilon(t,x,x)
 \le Ct^{-N}\frac{\psi_\varepsilon(x)}{\psi_{\varepsilon+t}(x)},
 \qquad t>0.
\end{equation*}
The
Davies estimate for $H_\varepsilon + \frac{N}{T}$ does not change. Take
\begin{equation}\label{pg:volume-function}
 \mathcal V_\varepsilon(x,r)
 =r^{2N}\frac{\psi_{\varepsilon+r^2}(x)}{\psi_\varepsilon(x)}.
\end{equation}
This function is positive and continuous. The doubling estimate
\[
 \frac{\mathcal V_\varepsilon(x,R)}{\mathcal V_\varepsilon(x,r)}
 \le(R/r)^{2N},\qquad R\ge r>0.
\]
 follows from the fact that $r \mapsto \psi_{\varepsilon+r^2}$ is a monotonically decreasing function.
Next, the monotonicity of $r \mapsto \mathcal V_\varepsilon(x,r)$ follows from the inequality:
\[
 r\frac{d}{dr}\log\mathcal V_\varepsilon(x,r)
 =2N-\frac\nu N\sum_{i<j}
       \frac{r^2}{|x_i-x_j|^2+\varepsilon+r^2}
 \ge2N-\frac{\nu(N-1)}2>0,
\]
where we used $\nu<4$ - this is the threshold for local integrability of the limiting $\varepsilon \downarrow 0$ weight $\psi=\psi_0$, and it follows from the condition on $\nu$ in Theorem \ref{thm1}. 

Thus, the conditions of Theorem \ref{pg:cs-variable} are satisfied
with $K=1$ and $D=2N$. Now, for $0<t\le T$ and $|x-y|^2\ge4t$, we put
$\tau:=|x-y|^2/4$ and recall that the kernel of
$e^{-t(H_\varepsilon+N/T)}$ is
$e^{-Nt/T}h_\varepsilon(t,x,y)$, so
\eqref{pg:CS-conclusion}, with $K=1$ and $D=2N$, gives
\[
 e^{-Nt/T}h_\varepsilon(t,x,y)
 \le
 \frac{eC}
 {\sqrt{\mathcal V_\varepsilon(x,\sqrt{\tau})
              \mathcal V_\varepsilon(y,\sqrt{\tau})}}
 \left(\frac{\tau}{t}\right)^{2N}e^{-\frac{\tau}{t}}.
\]
By \eqref{pg:volume-function},
\[
 \mathcal V_\varepsilon(x,\sqrt{\tau})=\tau^N\frac{\psi_{\varepsilon+\tau}(x)}
                    {\psi_\varepsilon(x)},
\]
so
\[
 \frac{1}
 {\sqrt{\mathcal V_\varepsilon(x,\sqrt{\tau})
              \mathcal V_\varepsilon(y,\sqrt{\tau})}}
 =\tau^{-N}
   \frac{\sqrt{\psi_\varepsilon(x)\psi_\varepsilon(y)}}
        {\sqrt{\psi_{\varepsilon+\tau}(x)
                         \psi_{\varepsilon+\tau}(y)}}.
\]
Let us return to $p_\varepsilon$. We use
\eqref{pg:kernel-identification}, note that
\[
 \sqrt{\frac{\psi_\varepsilon(y)}{\psi_\varepsilon(x)}}
 \sqrt{\psi_\varepsilon(x)\psi_\varepsilon(y)}
 =\psi_\varepsilon(y).
\]
and apply $e^{1+Nt/T}\le e^{N+1}$ for $0<t\le T$.
Absorbing this factor into $C$, we thus obtain
\begin{equation}\label{pg:symmetric-bound}
 p_\varepsilon(t,x,y)
 \le C\tau^{-N}
 \frac{\psi_\varepsilon(y)}
      {\sqrt{\psi_{\varepsilon+\tau}(x)
                         \psi_{\varepsilon+\tau}(y)}}
 \left(\frac{\tau}{t}\right)^{2N}e^{-\tau/t}.
\end{equation}
Next, we note that the two factors in the denominator are comparable at this scale.
Indeed, for each pair $(i,j)$,
\[
 |(x_i-x_j)-(y_i-y_j)|\le\sqrt2|x-y|=2\sqrt{2\tau},
\]
and therefore
\[
 \sqrt{|x_i-x_j|^2+\varepsilon+\tau}
 \le\sqrt{|y_i-y_j|^2+\varepsilon+\tau}+2\sqrt{2\tau}
 \le(1+2\sqrt2)\sqrt{|y_i-y_j|^2+\varepsilon+\tau}.
\]
Interchanging $x,y$ and multiplying over all pairs gives us
$
 \frac1{\sqrt{\psi_{\varepsilon+\tau}(x)
                         \psi_{\varepsilon+\tau}(y)}}
 \le\frac C{\psi_{\varepsilon+\tau}(y)}.
$
Moreover, since $\tau\ge t$,
\begin{equation*}
 \frac{\psi_{\varepsilon+t}(y)}{\psi_{\varepsilon+\tau}(y)}
 =\prod_{i<j}
 \left(\frac{|y_i-y_j|^2+\varepsilon+\tau}
            {|y_i-y_j|^2+\varepsilon+t}\right)^{\nu/(2N)}
 \le\left(\frac\tau t\right)^{\nu(N-1)/4}
\end{equation*}
(use $\frac{a+\tau}{a+t} \leq \frac{\tau}{t}$ if $a>0$, $\tau \geq t>0$).
We substitute this into \eqref{pg:symmetric-bound} and use 
$\tau^{-N}=t^{-N}(\tau/t)^{-N}$ to obtain
\[
 p_\varepsilon(t,x,y)
 \le Ct^{-N}\varphi_{t,\varepsilon}(y)
       \left(\frac\tau t\right)^{N+\nu(N-1)/4}e^{-\tau/t}.
\]
This is the sought bound for $|x-y|^2 \ge 4t$.

For $|x-y|^2<4t$, the bound in Theorem~\ref{thm1} suffices, since
the extra factor in \eqref{pg:gaussian-polynomial-eps} is bounded
below by $e^{-1}$. 

This proves \eqref{pg:gaussian-polynomial-eps}.

The passage to the limiting kernel is the same as in the proof of 
Theorem \ref{thm1}.  \hfill \qed

\bigskip

\section{Proof of Theorem \ref{nl:main}}\label{proof:nonlocal}

The proof will again be based on Theorem A. There are two non-trivial estimates
to establish: the dispersion estimate \eqref{S1} and the desingularizing $L^1 \rightarrow L^1$ estimate \eqref{S4}, for which we will use the
product structure of \eqref{nl:weights}. The other conditions of Theorem A, \eqref{S2}, \eqref{S3}, are straightforward in our setting.

\subsection{Preliminaries}

Define quadratic form
\[
 \mathcal E_\alpha(u)=
 \sum_{i=1}^N\|(-\Delta_{x_i})^{\alpha/4}u\|_2^2.
\]
This is the quadratic form of operator $A_\alpha$.
The inequalities between Fourier multipliers
\begin{equation}\label{nl:symbol}
 |\xi|^\alpha\le\sum_i|\xi_i|^\alpha
       \le N^{1-\alpha/2}|\xi|^\alpha
\end{equation}
show that $\sqrt{\|u\|_2^2+\mathcal E_\alpha(u)}$ is equivalent to the norm on the Bessel space  $W^{\alpha/2,2}$.

In the following lemma, exploiting the structure of the drift, we control the gradient by $\alpha/2$ derivatives of each test function. This can be viewed as fractional many-particle Hardy-type inequality.

\begin{lemma}\label{nl:drift-lemma}
For $\varepsilon>0$ and $u,v\in C_c^\infty(\mathbb R^{dN})$,
\begin{equation}\label{nl:drift-estimate}
 |\langle b_{\alpha,\varepsilon}\cdot\nabla u,v\rangle|
 \le C_{d,\alpha}\nu\frac{N-1}{N}
       \mathcal E_\alpha(u)^{1/2}\mathcal E_\alpha(v)^{1/2}.
\end{equation}
Furthermore, for $u\in W^{\alpha,2}$,
\begin{equation}\label{nl:coercivity}
 \operatorname{Re}\langle\Lambda_{\alpha,\varepsilon}u,u\rangle
 \ge\frac12\mathcal E_\alpha(u).
\end{equation}
The constants are independent of $\varepsilon$.
\end{lemma}

\begin{proof}
1.~We prove the first inequality via Fourier transform and Schur's test. We will use the fact that the regularized one-particle drift $F_\varepsilon(z)=z(|z|^2+\varepsilon)^{-\alpha/2}$, $z \in \mathbb R^d$, has Fourier transform that is rather easy to deal with. (Recall our Fourier transform convention: $\widehat u(\xi)=\int_{\mathbb R^d}e^{-iz\cdot\xi}u(z)\,dz$.)

\begin{claim}
\label{cl_fourier}
\begin{equation}\label{nl:Fourier-drift}
 |\widehat F_\varepsilon(\xi)|
 \le C_{d,\alpha}|\xi|^{\alpha-d-1},
 \qquad \xi\ne0.
\end{equation}

\end{claim}

\begin{proof}[Proof of Claim \ref{cl_fourier}]
Indeed, recall that
\[
 F_\varepsilon(z)
 =\frac1{\Gamma(\alpha/2)}
   \int_0^\infty s^{\alpha/2-1}e^{-\varepsilon s}
                  z e^{-s|z|^2}\,ds.
\]
Using identity
$
 \widehat{e^{-s|\cdot|^2}}(\xi)
 =\left(\frac{\pi}{s}\right)^{d/2}
      e^{-|\xi|^2/(4s)},
$
which we differentiate, we obtain
\[
 \widehat{z e^{-s|z|^2}}(\xi)
 =-\frac{i\pi^{d/2}}2\,\xi\,s^{-d/2-1}
      e^{-|\xi|^2/(4s)}.
\]
We apply the Fourier transform to the above representation of $F_\varepsilon$ and use the previous identity to obtain
\[
 \widehat F_\varepsilon(\xi)
 =-\frac{i\pi^{d/2}\xi}{2\Gamma(\alpha/2)}
   \int_0^\infty
      s^{(\alpha-d)/2-2}
      e^{-\varepsilon s-|\xi|^2/(4s)}\,ds.
\]
Now, we take absolute values, apply $e^{-\varepsilon s}\le1$ and evaluate the remaining integral as follows:
\[
 \begin{split}
 &\int_0^\infty
      s^{(\alpha-d)/2-2}e^{-|\xi|^2/(4s)}\,ds\\
	  & (\text{put $q=|\xi|^2/(4s)$, then
$s=|\xi|^2/(4q)$ and
$ds=-|\xi|^2(4q^2)^{-1}\,dq$}) \\
 &=
 \left(\frac{|\xi|^2}{4}\right)^{(\alpha-d)/2-1}
 \int_0^\infty q^{(d-\alpha)/2}e^{-q}\,dq\\
 &=
 \left(\frac{|\xi|^2}{4}\right)^{(\alpha-d)/2-1}
 \Gamma\left(\frac{d+2-\alpha}{2}\right).
 \end{split}
\]
Therefore, multiplying by $|\xi|$, we arrive at \eqref{nl:Fourier-drift}.
\end{proof}

We now use \eqref{nl:Fourier-drift} to establish 
the one particle estimate, i.e.\,for $F_\varepsilon$. Let $u,v\in C_c^\infty(\mathbb R^d)$.
We have
\[
 \widehat{F_\varepsilon\cdot\nabla u}(\xi)
 =(2\pi)^{-d}\int_{\mathbb R^d}
       \bigl[\widehat F_\varepsilon(\xi-\zeta)\cdot i\zeta\bigr]
       \widehat u(\zeta)\,d\zeta,
\]
so
\begin{align*}\int_{\mathbb R^d}
       F_\varepsilon(z)\cdot\nabla u(z)\,
                         \overline{v(z)}\,dz=(2\pi)^{-2d}
       \int_{\mathbb R^d}\int_{\mathbb R^d}
       \bigl[\widehat F_\varepsilon(\xi-\zeta)\cdot i\zeta\bigr]
       \widehat u(\zeta)\overline{\widehat v(\xi)}
       \,d\zeta\,d\xi.
\end{align*}
Taking absolute values and using
\eqref{nl:Fourier-drift} yields
\[
 \left|\int_{\mathbb R^d}
       F_\varepsilon(z)\cdot\nabla u(z)\,
                         \overline{v(z)}\,dz\right|\le C_{d,\alpha}
       \int_{\mathbb R^d}\int_{\mathbb R^d}
       |\xi-\zeta|^{-(d+1-\alpha)}|\zeta|\,
       |\widehat u(\zeta)|\,|\widehat v(\xi)|
       \,d\zeta\,d\xi.
\]
Our goal is to estimate the right-hand side in terms of the $L^2$ norms of the fractional
derivatives of $u$, $v$. Define
$f(\zeta)=|\zeta|^{\alpha/2}|\widehat u(\zeta)|$,
$g(\xi)=|\xi|^{\alpha/2}|\widehat v(\xi)|$.
Then the integrand in the previous estimate becomes
$K(\xi,\zeta)f(\zeta)g(\xi)$, where
\[
 K(\xi,\zeta)
 :=|\xi|^{-\alpha/2}
   |\xi-\zeta|^{-(d+1-\alpha)}
   |\zeta|^{1-\alpha/2},
\]
so 
\[
 \left|\int_{\mathbb R^d}
       F_\varepsilon(z)\cdot\nabla u(z)\,
                         \overline{v(z)}\,dz\right|\le C_{d,\alpha}
       \int_{\mathbb R^d}
       \left(\int_{\mathbb R^d}
                    K(\xi,\zeta)f(\zeta)\,d\zeta\right)
       g(\xi)\,d\xi.
\]
It therefore suffices to show that the integral operator
with kernel $K$ is bounded on $L^2(\mathbb R^d)$.
 To this end, we apply the Schur test
with weight
 $h(\xi)=|\xi|^{-d/2}$.
We will need the standard composition property of Riesz potentials (see e.g.\,\cite{KPS}):
\[
 \int_{\mathbb R^d}|x-z|^{-a}|z|^{-b}\,dz
 = C|x|^{d-a-b}
\]
provided $0<a,b<d$ and $a+b>d$.
For the first Schur integral, the composition property gives
\[
 \begin{split}
 \int_{\mathbb R^d}K(\xi,\zeta)h(\zeta)\,d\zeta
 &=|\xi|^{-\alpha/2}
   \int_{\mathbb R^d}
      |\xi-\zeta|^{-(d+1-\alpha)}
      |\zeta|^{-(d+\alpha-2)/2}\,d\zeta\\
 &= C|\xi|^{-\alpha/2}|\xi|^{(\alpha-d)/2}
 =Ch(\xi).
 \end{split}
\]
Here the two exponents are $d+1-\alpha$ and
$(d+\alpha-2)/2$, and their sum exceeds $d$ by
$(d-\alpha)/2>0$.
For the second Schur integral, we  obtain in the same way
\[
 \begin{split}
 \int_{\mathbb R^d}K(\xi,\zeta)h(\xi)\,d\xi
 &=|\zeta|^{1-\alpha/2}
   \int_{\mathbb R^d}
      |\xi-\zeta|^{-(d+1-\alpha)}
      |\xi|^{-(d+\alpha)/2}\,d\xi\\
 &\le C|\zeta|^{1-\alpha/2}|\zeta|^{(\alpha-d-2)/2} = 
Ch(\zeta).
 \end{split}
\]
In this case the two exponents are $d+1-\alpha$ and
$(d+\alpha)/2$, and their sum exceeds $d$ by
$(d+2-\alpha)/2>0$.
The Schur test therefore gives
\[
 \int_{\mathbb R^d}\int_{\mathbb R^d}
       K(\xi,\zeta)f(\zeta)g(\xi)\,d\zeta\,d\xi
 \le C_{d,\alpha}\|f\|_2\|g\|_2,
\]
which then immediately yields
\begin{equation}\label{nl:one-drift}
 \left|\int_{\mathbb R^d}
       F_\varepsilon(z)\cdot\nabla u(z)\,
                         \overline{v(z)}\,dz\right|
 \le C_{d,\alpha}
       \|(-\Delta)^{\alpha/4}u\|_2
       \|(-\Delta)^{\alpha/4}v\|_2,
\end{equation}
as claimed. Recall that this is a ``one-particle'' estimate on $\mathbb R^d$.

We now return to the particle system. Let
$u,v\in C_c^\infty(\mathbb R^{dN})$ and fix $i\ne j$.
Write
$dx_{\widehat i}=\prod_{k\ne i}dx_k$.
For fixed $x_{\widehat i}$, the variable $x_j$ is fixed as
well. We can therefore apply \eqref{nl:one-drift} in the
variable $x_i$, after translating $x_i$ by $x_j$.
Since the fractional Laplacian commutes with translations,
this gives
\[
 \begin{split}
 &\left|\int_{\mathbb R^d}
    F_\varepsilon(x_i-x_j)\cdot\nabla_{x_i}u(x)\,
                  \overline{v(x)}\,dx_i\right|\\
 &\quad\le C_{d,\alpha}
   \left(\int_{\mathbb R^d}
       |(-\Delta_{x_i})^{\alpha/4}u(x)|^2\,dx_i\right)^{1/2}
   \left(\int_{\mathbb R^d}
       |(-\Delta_{x_i})^{\alpha/4}v(x)|^2\,dx_i\right)^{1/2}.
 \end{split}
\]
Next, we integrate this inequality in
$x_{\widehat i}\in\mathbb R^{d(N-1)}$.
The two square roots on its right-hand side are functions
of $x_{\widehat i}$, so Cauchy-Schwarz yields
\[
 \begin{split}
 &\left|\int_{\mathbb R^{dN}}
    F_\varepsilon(x_i-x_j)\cdot\nabla_{x_i}u(x)\,
                  \overline{v(x)}\,dx\right|\\
 &\quad\le C_{d,\alpha}
  \int_{\mathbb R^{d(N-1)}}
   \left(\int_{\mathbb R^d}
       |(-\Delta_{x_i})^{\alpha/4}u(x)|^2\,dx_i\right)^{1/2}
   \left(\int_{\mathbb R^d}
       |(-\Delta_{x_i})^{\alpha/4}v(x)|^2\,dx_i\right)^{1/2}
       dx_{\widehat i}\\
 &\quad\le C_{d,\alpha}
   \left(\int_{\mathbb R^{d(N-1)}}\int_{\mathbb R^d}
       |(-\Delta_{x_i})^{\alpha/4}u(x)|^2
                         \,dx_i\,dx_{\widehat i}\right)^{1/2}
   \left(\int_{\mathbb R^{d(N-1)}}\int_{\mathbb R^d}
       |(-\Delta_{x_i})^{\alpha/4}v(x)|^2
                         \,dx_i\,dx_{\widehat i}\right)^{1/2}\\
 &\quad=C_{d,\alpha}
       \|(-\Delta_{x_i})^{\alpha/4}u\|_2
       \|(-\Delta_{x_i})^{\alpha/4}v\|_2.
 \end{split}
\]
The resulting bound no longer depends on $j$. Therefore,
\[
 \begin{split}
 |\langle b_{\alpha,\varepsilon}\cdot\nabla u,v\rangle|
 &\le\frac{C_{d,\alpha}\nu}{N}
       \sum_{i=1}^N\sum_{j\ne i}
       \|(-\Delta_{x_i})^{\alpha/4}u\|_2
       \|(-\Delta_{x_i})^{\alpha/4}v\|_2\\
 &=\frac{C_{d,\alpha}\nu(N-1)}{N}
       \sum_{i=1}^N
       \|(-\Delta_{x_i})^{\alpha/4}u\|_2
       \|(-\Delta_{x_i})^{\alpha/4}v\|_2\\
 &\le\frac{C_{d,\alpha}\nu(N-1)}{N}
       \left(\sum_{i=1}^N
          \|(-\Delta_{x_i})^{\alpha/4}u\|_2^2\right)^{1/2}
       \left(\sum_{i=1}^N
          \|(-\Delta_{x_i})^{\alpha/4}v\|_2^2\right)^{1/2}\\
 &=\frac{C_{d,\alpha}\nu(N-1)}{N}
       \mathcal E_\alpha(u)^{1/2}\mathcal E_\alpha(v)^{1/2}.
 \end{split}
\]

2.~For the coercivity inequality \eqref{nl:coercivity}, we will need the following claim.

\begin{claim}[Standard fractional many-particle Hardy inequality]
\label{nl:hardy-lemma}
For every $u\in W^{\alpha/2,2}(\mathbb R^{dN})$,
\begin{equation}\label{nl:many-Hardy}
 \sum_{i=1}^N\sum_{j\ne i}
 \int_{\mathbb R^{dN}}
       \frac{|u(x)|^2}{|x_i-x_j|^\alpha}\,dx
 \le \frac{N-1}{\kappa_{d,\alpha}}\,
       \mathcal E_\alpha(u).
\end{equation}
\end{claim}

This inequality is obtained by summing up the one-dimensional fractional Hardy inequalities, see Appendix \ref{app:elementary} for details.

Next, a direct calculation yields
\begin{equation*}
 \operatorname{div}F_\varepsilon(z)
 =\frac{(d-\alpha)|z|^2+d\varepsilon}
        {|z|_\varepsilon^{\alpha+2}}
 \le D_{d,\alpha}|z|^{-\alpha}.
\end{equation*}
Thus, integration by parts and Claim \ref{nl:hardy-lemma}  give
\[
 -\operatorname{Re}\langle b_{\alpha,\varepsilon}\cdot\nabla u,u\rangle
 =\frac12\langle(\operatorname{div}b_{\alpha,\varepsilon})|u|^2\rangle
 \le\frac{\nu D_{d,\alpha}(N-1)}{2N\kappa_{d,\alpha}}
       \mathcal E_\alpha(u)
 \le\frac12\mathcal E_\alpha(u).
\]
In the last step we have used \eqref{nl:smallness}. Adding
$\mathcal E_\alpha(u)$ yields \eqref{nl:coercivity}.
\end{proof}

We now construct the semigroup for the regularized operator
$\Lambda_{\alpha,\varepsilon}=A_\alpha+b_{\alpha,\varepsilon}\cdot\nabla$
on $W^{\alpha,2}$. By \eqref{nl:coercivity}, this operator is
accretive. To apply the Lumer-Phillips theorem (see \cite[Ch.\,IX, Sect.\,8]{Y}), it remains to
verify the range condition for a fixed sufficiently large 
$\lambda>0$. That is, we will show the solvability of $(\lambda+\Lambda_{\alpha,\varepsilon})u=f$ for every $f \in L^2$. 
Indeed, note first that, using \eqref{nl:symbol}, we obtain
\[
 \|\nabla(\lambda+A_\alpha)^{-1}\|_{2\to2}
 \le C_\alpha\lambda^{1/\alpha-1}.
\]
Since $b_{\alpha,\varepsilon}$ is bounded and $\alpha>1$,
\[
 \|b_{\alpha,\varepsilon}\cdot\nabla
                (\lambda+A_\alpha)^{-1}\|_{2\to2}
 \le C_\alpha\|b_{\alpha,\varepsilon}\|_\infty
               \lambda^{1/\alpha-1}<1
\]
for sufficiently large $\lambda$ (depending on $\varepsilon$).
The factorization
\[
 \lambda+\Lambda_{\alpha,\varepsilon}
 =\bigl[I+b_{\alpha,\varepsilon}\cdot\nabla
                   (\lambda+A_\alpha)^{-1}\bigr]
      (\lambda+A_\alpha)
\]
reduces the problem to showing solvability of $v + b_{\alpha,\varepsilon}\cdot\nabla (\lambda+A_\alpha)^{-1}v=f$, where $v=(\lambda+A_\alpha)u$. This solvability is proved by writing Neumann series and using the above $<1$ operator norm estimate. It follows from this construction that $\lambda+\Lambda_{\alpha,\varepsilon}$ maps $W^{\alpha,2}$ onto
$L^2$. 
The Lumer-Phillips theorem now gives a strongly continuous
contraction semigroup $e^{-t\Lambda_{\alpha,\varepsilon}}$ on $L^2$.

For fixed $\varepsilon>0$, positivity and the $L^\infty$
contraction property of $e^{-t\Lambda_{\alpha,\varepsilon}}$
follow from the classical theory for smooth bounded drifts.
Interpolating the $L^2$ and $L^\infty$ contractions now gives
contractions on $L^r$, $2\le r<\infty$. Their strong continuity
follows from the $L^2$ continuity: indeed, for
$f\in L^2\cap L^\infty$,
\[
 \|e^{-t\Lambda_{\alpha,\varepsilon}}f-f\|_r^r
 \le(2\|f\|_\infty)^{r-2}
       \|e^{-t\Lambda_{\alpha,\varepsilon}}f-f\|_2^2,
\]
so it remains to apply the standard density argument.

Consider also the adjoint operator on $L^2$:
\[
 \Lambda_{\alpha,\varepsilon}^*
 =A_\alpha-b_{\alpha,\varepsilon}\cdot\nabla
             -\operatorname{div}b_{\alpha,\varepsilon},
 \qquad D(\Lambda_{\alpha,\varepsilon}^*)=W^{\alpha,2}.
\]
Note that $b_{\alpha,\varepsilon}$ and
$\operatorname{div}b_{\alpha,\varepsilon}$ are bounded and smooth for fixed
$\varepsilon>0$, so the previous construction via Lumer-Phillips can be repeated.
By duality,
$e^{-t\Lambda_{\alpha,\varepsilon}^*}
 =(e^{-t\Lambda_{\alpha,\varepsilon}})^*$ is a strongly continuous
contraction semigroup on $L^2$. Positivity preservation and the $L^\infty$
contraction property of the original semigroup yield positivity preservation
of the adjoint semigroup and its $L^1$ contraction property:
\[
 \|e^{-t\Lambda_{\alpha,\varepsilon}^*}\|_{1\to1}\le1,
 \qquad t\ge0.
\]

\subsection{Verification of \eqref{S1}}

For $f\in C_c^\infty$, using \eqref{nl:coercivity}, we obtain
\[
 \begin{split}
 \frac{d}{dt}\|e^{-t\Lambda_{\alpha,\varepsilon}}f\|_2^2
 &=-2\operatorname{Re}
   \langle\Lambda_{\alpha,\varepsilon}
                 e^{-t\Lambda_{\alpha,\varepsilon}}f,\,
                 e^{-t\Lambda_{\alpha,\varepsilon}}f\rangle
 \le-\mathcal E_\alpha(e^{-t\Lambda_{\alpha,\varepsilon}}f).
 \end{split}
\]
Therefore,
\[
 \int_0^t\mathcal E_\alpha(e^{-s\Lambda_{\alpha,\varepsilon}}f)\,ds
 \le\|f\|_2^2.
\]
Next, the semigroup property and $L^r$ contractivity
imply, for $0\le s\le t$,
\[
 \|e^{-t\Lambda_{\alpha,\varepsilon}}f\|_r
 =\|e^{-(t-s)\Lambda_{\alpha,\varepsilon}}
             e^{-s\Lambda_{\alpha,\varepsilon}}f\|_r
 \le\|e^{-s\Lambda_{\alpha,\varepsilon}}f\|_r.
\]
Therefore,
\[
 \begin{aligned}
 t\|e^{-t\Lambda_{\alpha,\varepsilon}}f\|_r^2
 &\le\int_0^t\|e^{-s\Lambda_{\alpha,\varepsilon}}f\|_r^2\,ds\\
 & (\text{take $r=2dN/(dN-\alpha)$ and use Sobolev embedding}) \\
 &\le C\int_0^t
       \mathcal E_\alpha(e^{-s\Lambda_{\alpha,\varepsilon}}f)\,ds
 \le C\|f\|_2^2.
 \end{aligned}
\]
By density in $L^2$, this proves
\begin{equation}\label{nl:sobolev-smoothing}
 \|e^{-t\Lambda_{\alpha,\varepsilon}}\|_{2\to 2dN/(dN-\alpha)}
 \le Ct^{-1/2},\qquad t>0,
\end{equation}
with $C$ independent of $\varepsilon$.

We now apply the Coulhon-Raynaud extrapolation theorem (Theorem~\ref{thmE2} in
Appendix~\ref{app_A}) in the same way as in the proof of
Theorem~\ref{thm1}. We have $L^1$ contractivity of the adjoint semigroup, the dual of
\eqref{nl:sobolev-smoothing} takes form
\[
 \|(e^{-t\Lambda_{\alpha,\varepsilon}})^*\|_{q\to2}\le Ct^{-1/2},
 \qquad q=\frac{2dN}{dN+\alpha}.
\]
We therefore take $p=1$, this value of $q$, and final exponent $2$
in Theorem~\ref{thmE2}, which then gives
\[
 \|(e^{-t\Lambda_{\alpha,\varepsilon}})^*\|_{1\to2}
 \le Ct^{-(1/2)/(\alpha/(dN))}
 =Ct^{-dN/(2\alpha)}.
\]
Returning to the original semigroup by duality, we obtain
\begin{equation*}
 \|e^{-t\Lambda_{\alpha,\varepsilon}}\|_{2\to\infty}
 \le C t^{-dN/(2\alpha)}.
\end{equation*}
This proves \eqref{S1} of Theorem A with $r=2$ and $a=dN/\alpha$,
with constants independent of $\varepsilon>0$.

\subsection{Verification of \eqref{S4}}

\begin{lemma}\label{nl:power-lemma}
For every $\varepsilon>0$,
\begin{equation}\label{nl:power-bound}
 (-\Delta)^{\alpha/2}|z|_\varepsilon^{-(d-\alpha)/2}
 \ge\kappa_{d,\alpha}|z|_\varepsilon^{-(d+\alpha)/2}
       \left(1+\frac{\alpha\varepsilon}
                        {(d-\alpha)|z|_\varepsilon^2}\right), \quad z \in \mathbb R^d.
\end{equation}
\end{lemma}

This result is a special case of \cite[Corollary 2]{DKK}, but we also give a direct proof in Appendix \ref{app:elementary}.

\begin{lemma}\label{nl:supersolution}
Provided that $\nu$ satisfies \eqref{nl:smallness}, there is a constant $C=C(N)$, independent of
$s,\varepsilon>0$, such that
\begin{equation}\label{nl:adjoint-bound}
 \Lambda_{\alpha,\varepsilon}^*\varphi_{\alpha,s,\varepsilon}
 \ge-\frac Cs\varphi_{\alpha,s,\varepsilon}
 \quad\text{pointwise on }\mathbb R^{dN}.
\end{equation}
\end{lemma}

\begin{proof}
The case $\nu=0$ is immediate, since the weight is $1$.
Assume $\nu>0$ and put
\[
 \theta=\frac{3\nu(d-\alpha)(N-1)}{2N\kappa_{d,\alpha}}.
\]
Since $D_{d,\alpha}\ge d-\alpha$, condition \eqref{nl:smallness}
implies $0<\theta<1$. The exponent of each pair factor in the
weight is $\theta/(N-1)$.

1.~Let us first handle $A_\alpha\varphi_{\alpha,s,\varepsilon}$. Recall the following representations of the fractional Laplacian.
For a bounded $C^2$ function $u$,
\[
 (-\Delta_{x_i})^{\alpha/2}u(x)
 =\lim_{\delta\downarrow0}L_{i,\delta}u(x), \quad  L_{i,\delta}u(x):=c_{d,\alpha}\int_{|h|>\delta}
       \frac{u(x)-u(x+h_i)}{|h|^{d+\alpha}}\,dh,
\]
where $x+h_i=(x_1,\ldots,x_i+h,\ldots,x_N)$. The symmetry under $h\mapsto-h$ yields
\[
 (-\Delta_{x_i})^{\alpha/2}u(x)
=\frac{c_{d,\alpha}}2\int_{\mathbb R^d}
 \frac{2u(x)-u(x+h_i)-u(x-h_i)}{|h|^{d+\alpha}}\,dh.
\]
In the last expression  the numerator is
$O(|h|^2)$ near the origin and bounded at infinity, so the integral is
therefore absolutely convergent (by $\alpha<2$ at the origin and
$\alpha>0$ at infinity).

We fix $s,\varepsilon>0$ and write
$$F(z)=1+s^{(d-\alpha)/(2\alpha)}|z|_\varepsilon^{-(d-\alpha)/2},$$ so that the weight $\varphi_{\alpha,s,\varepsilon}$
is a product of powers of $F$. Both $F$ and
$\varphi_{\alpha,s,\varepsilon}$ are bounded, positive $C^2$ functions
with bounded first and second derivatives, as required above. When
$x$ is replaced by $x+h_i$, only the $N-1$ factors involving $x_i$
change. Moreover, $F$ is even, so the factors with $j<i$ can also be
written in terms of $x_i-x_j$. We can therefore apply the weighted
arithmetic-geometric mean inequality, with weights $\theta/(N-1)$
for these $N-1$ ratios and weight $1-\theta$ for the constant $1$:
\[
 \frac{\varphi_{\alpha,s,\varepsilon}(x+h_i)}{\varphi_{\alpha,s,\varepsilon}(x)}
 =\prod_{j\ne i}
   \left[\frac{F(x_i-x_j+h)}{F(x_i-x_j)}\right]^{\theta/(N-1)}
 \le1-\theta+\frac{\theta}{N-1}\sum_{j\ne i}
       \frac{F(x_i-x_j+h)}{F(x_i-x_j)}.
\]
Subtracting this inequality from $1$ and using that there are exactly
$N-1$ summands yields
\[
 1-\frac{\varphi_{\alpha,s,\varepsilon}(x+h_i)}{\varphi_{\alpha,s,\varepsilon}(x)}
 \ge\frac{\theta}{N-1}\sum_{j\ne i}
       \frac{F(x_i-x_j)-F(x_i-x_j+h)}{F(x_i-x_j)}.
\]
At this point the differences have the sign required in the
principal-value representation of the fractional Laplacian. For each $\delta>0$, the truncated
integrals are absolutely convergent, so we may integrate the preceding
inequality over $|h|>\delta$ against the positive measure
$c_{d,\alpha}|h|^{-d-\alpha}\,dh$. We obtain
\[
 \frac{L_{i,\delta}\varphi_{\alpha,s,\varepsilon}(x)}{\varphi_{\alpha,s,\varepsilon}(x)}
 \ge\frac{\theta}{N-1}\sum_{j\ne i}
       \frac{L_\delta F(x_i-x_j)}{F(x_i-x_j)}.
\]
Passing to the principal-value limits and using $(-\Delta)^{\alpha/2}1=0$, we obtain
\begin{equation}
\label{ineq:main}
 \frac{(-\Delta_{x_i})^{\alpha/2}\varphi_{\alpha,s,\varepsilon}(x)}
      {\varphi_{\alpha,s,\varepsilon}(x)}
 \ge\frac{\theta s^{(d-\alpha)/(2\alpha)}}{N-1}\sum_{j\ne i}
 \frac{(-\Delta)^{\alpha/2}|x_i-x_j|_\varepsilon^{-(d-\alpha)/2}}
      {1+s^{(d-\alpha)/(2\alpha)}|x_i-x_j|_\varepsilon^{-(d-\alpha)/2}}.
\end{equation}
We now apply Lemma~\ref{nl:power-lemma} to the right-hand side. Write
\[
 q_\varepsilon(z)=1+
       \frac{\alpha\varepsilon}{(d-\alpha)|z|_\varepsilon^2},
 \qquad 1\le q_\varepsilon(z)\le\frac d{d-\alpha}.
\]
Then each numerator on the right is bounded
below by
$\kappa_{d,\alpha}q_\varepsilon(z)|z|_\varepsilon^{-(d+\alpha)/2}$
with $z=x_i-x_j$. The identity
\[
 \frac{s^{(d-\alpha)/(2\alpha)}|z|_\varepsilon^{-(d-\alpha)/2}}
      {1+s^{(d-\alpha)/(2\alpha)}|z|_\varepsilon^{-(d-\alpha)/2}}
 =\frac1{1+(|z|_\varepsilon/s^{1/\alpha})^{(d-\alpha)/2}}
\]
then gives, after summation over $1 \leq i \leq N$,
\begin{equation}\label{nl:product-diffusion}
 \frac{A_\alpha\varphi_{\alpha,s,\varepsilon}}{\varphi_{\alpha,s,\varepsilon}}
 \ge\frac{\theta\kappa_{d,\alpha}}{N-1}
 \sum_i\sum_{j\ne i}
 \frac{q_\varepsilon(x_i-x_j)|x_i-x_j|_\varepsilon^{-\alpha}}
      {1+(|x_i-x_j|_\varepsilon/s^{1/\alpha})^{(d-\alpha)/2}}.
\end{equation}

2.~We next estimate the contribution of the drift term. We have
\[
 \nabla_{x_i}\log\varphi_{\alpha,s,\varepsilon}
 =-\frac{\theta(d-\alpha)}{2(N-1)}\sum_{j\ne i}
   \frac{x_i-x_j}{|x_i-x_j|_\varepsilon^2}
   \frac1{1+(|x_i-x_j|_\varepsilon/s^{1/\alpha})^{(d-\alpha)/2}}.
\]
For a fixed $i$, the drift $b_{\alpha,\varepsilon,i}$ contains
one sum over $j\ne i$, while the gradient above contains another
sum which we index by $k\ne i$. Multiplying these two sums gives
\[
 \begin{split}
 b_{\alpha,\varepsilon,i}\cdot\nabla_{x_i}\log\varphi_{\alpha,s,\varepsilon}
 =-\frac{\theta\nu(d-\alpha)}{2N(N-1)}
 \sum_{j\ne i}\sum_{k\ne i}
 &\frac{(x_i-x_j)\cdot(x_i-x_k)}
 {|x_i-x_j|_\varepsilon^\alpha|x_i-x_k|_\varepsilon^2}\,\frac1{1+(|x_i-x_k|_\varepsilon/s^{1/\alpha})^{(d-\alpha)/2}}.
 \end{split}
\]
The indices $j$ and $k$ vary independently, so the sum includes the
terms with $j=k$. By Cauchy-Schwarz and $|z|\le |z|_\varepsilon$,
\[
 |(x_i-x_j)\cdot(x_i-x_k)|
 \le |x_i-x_j|\,|x_i-x_k|
 \le |x_i-x_j|_\varepsilon|x_i-x_k|_\varepsilon.
\]
Moreover, the last factor in the double sum above is between
zero and one. The triangle inequality, first within each double sum
and then over $i$, therefore yields
\[
 \begin{split}
 |b_{\alpha,\varepsilon}\cdot\nabla\log\varphi_{\alpha,s,\varepsilon}|
 &\le\sum_i
       |b_{\alpha,\varepsilon,i}\cdot\nabla_{x_i}\log\varphi_{\alpha,s,\varepsilon}|\\
 &\le\frac{\theta\nu(d-\alpha)}{2N(N-1)}
       \sum_i\sum_{j\ne i}\sum_{k\ne i}
       |x_i-x_j|_\varepsilon^{1-\alpha}
       |x_i-x_k|_\varepsilon^{-1}.
 \end{split}
\]
To compare this expression with the pair sum in the $A_\alpha$ diffusion
estimate from the first step, we use Young's inequality with conjugate exponents
$\alpha/(\alpha-1)$ and $\alpha$:
\[
 r^{1-\alpha}t^{-1}
 \le\frac{\alpha-1}{\alpha}r^{-\alpha}
       +\frac1\alpha t^{-\alpha}.
\]
For fixed $i$, we apply the inequality with
$r=|x_i-x_j|_\varepsilon$ and $t=|x_i-x_k|_\varepsilon$.
The independent sums over $j$ and $k$ can then be evaluated as follows:
\[
 \begin{aligned}
 &\sum_{j\ne i}\sum_{k\ne i}
       |x_i-x_j|_\varepsilon^{1-\alpha}|x_i-x_k|_\varepsilon^{-1}\\
 &\quad\le\frac{\alpha-1}{\alpha}
       \sum_{j\ne i}\sum_{k\ne i}|x_i-x_j|_\varepsilon^{-\alpha}
       +\frac1\alpha
       \sum_{j\ne i}\sum_{k\ne i}|x_i-x_k|_\varepsilon^{-\alpha}\\
 &\quad=\frac{\alpha-1}{\alpha}(N-1)
       \sum_{j\ne i}|x_i-x_j|_\varepsilon^{-\alpha}
       +\frac1\alpha(N-1)
       \sum_{k\ne i}|x_i-x_k|_\varepsilon^{-\alpha}\\
 &\quad=(N-1)\sum_{j\ne i}|x_i-x_j|_\varepsilon^{-\alpha}.
 \end{aligned}
\]
Substitution into the preceding bound
cancels the denominator $N-1$ and yields
\begin{equation}\label{nl:weight-drift}
 |b_{\alpha,\varepsilon}\cdot\nabla\log\varphi_{\alpha,s,\varepsilon}|
 \le\frac{\theta\nu(d-\alpha)}{2N}
       \sum_i\sum_{j\ne i}|x_i-x_j|_\varepsilon^{-\alpha}.
\end{equation}

3.~The remaining term is the divergence of the drift.
Using the same factor $q_\varepsilon$ as at the first step,
we can write it as
\[
 \operatorname{div}b_{\alpha,\varepsilon}
 =\frac{\nu(d-\alpha)}N\sum_i\sum_{j\ne i}
       q_\varepsilon(x_i-x_j)|x_i-x_j|_\varepsilon^{-\alpha}.
\]

4.~Combining the former with \eqref{nl:product-diffusion},
\eqref{nl:weight-drift}, and $q_\varepsilon\ge1$, we find
\begin{equation}\label{nl:adjoint-comparison}
 \begin{split}
 \frac{\Lambda_{\alpha,\varepsilon}^*\varphi_{\alpha,s,\varepsilon}}
      {\varphi_{\alpha,s,\varepsilon}}
 &\ge\sum_i\sum_{j\ne i}
 q_\varepsilon(x_i-x_j)|x_i-x_j|_\varepsilon^{-\alpha}\\
 &\quad\times\left[
 \frac{\theta\kappa_{d,\alpha}}
 {(N-1)(1+(|x_i-x_j|_\varepsilon/s^{1/\alpha})^{(d-\alpha)/2})}
 -\frac{\nu(d-\alpha)(2+\theta)}{2N}\right].
 \end{split}
\end{equation}
Since $\theta<1$, its definition gives
$
 \frac{\theta\kappa_{d,\alpha}}{N-1}
 =\frac{3\nu(d-\alpha)}{2N}
 >\frac{\nu(d-\alpha)(2+\theta)}{2N}.
$
Therefore, the function
\[
 \rho\longmapsto\rho^{-\alpha}
 \left[\frac{\theta\kappa_{d,\alpha}}
 {(N-1)(1+\rho^{(d-\alpha)/2})}
 -\frac{\nu(d-\alpha)(2+\theta)}{2N}\right]
\]
is nonnegative near zero. Away from zero it is continuous, and it
tends to zero as $\rho\to\infty$, so it has a finite lower bound. Writing the middle term in \eqref{nl:adjoint-comparison} $|x_i-x_j|_\varepsilon^{-\alpha}=\frac{1}{s}(|x_i-x_j|_\varepsilon/s^{1/\alpha})^{-\alpha}$
and recalling that
$1\le q_\varepsilon\le d/(d-\alpha)$, we see that each term in
\eqref{nl:adjoint-comparison} is bounded below by $-C/s$. Summing up and re-denoting $C$, we obtain \eqref{nl:adjoint-bound}. 
\end{proof}

\begin{remark}
The first step in the proof can be handled entirely within the semigroup framework. 
Since $e^{-t(-\Delta_{x_i})^{\alpha/2}}$ acts only on $x_i$, all pair factors independent of $x_i$
can be taken outside the semigroup. We are therefore left with the $N-1$ factors corresponding to $j\ne i$. 
The semigroup is positivity preserving and satisfies $e^{-t(-\Delta_{x_i})^{\alpha/2}}1=1$. We apply H\"older's inequality with exponent $(N-1)/\theta$ for each of the $N-1$ pair factors, and exponent $1/(1-\theta)$ for the additional constant factor $1$. These exponents satisfy
\[
\sum_{j\ne i}\frac{\theta}{N-1}+(1-\theta)=1.
\]
This gives
\[
\begin{aligned}
&e^{-t(-\Delta_{x_i})^{\alpha/2}}
\left[
\prod_{j\ne i}
\left(
1+s^{(d-\alpha)/(2\alpha)}
|x_i-x_j|_\varepsilon^{-(d-\alpha)/2}
\right)^{\theta/(N-1)}
\right]
\\
&\quad\le
\left(
e^{-t(-\Delta_{x_i})^{\alpha/2}}1
\right)^{1-\theta}
\prod_{j\ne i}
\left[
e^{-t(-\Delta_{x_i})^{\alpha/2}}
\left(
1+s^{(d-\alpha)/(2\alpha)}
|x_i-x_j|_\varepsilon^{-(d-\alpha)/2}
\right)
\right]^{\theta/(N-1)}.
\end{aligned}
\]
The first factor on the right equals $1$ because the semigroup preserves constants. By linearity
and translation invariance,
\[
e^{-t(-\Delta_{x_i})^{\alpha/2}} \left( 1+s^{(d-\alpha)/(2\alpha)} |x_i-x_j|_\varepsilon^{-(d-\alpha)/2}\right)
=
1+s^{(d-\alpha)/(2\alpha)}\left(e^{-t(-\Delta)^{\alpha/2}}|\cdot|_\varepsilon^{-(d-\alpha)/2} \right)(x_i-x_j).
\]
Multiplying by the pair factors independent of $x_i$ and dividing by $\varphi_{\alpha,s,\varepsilon}(x)$, we obtain
\begin{equation}\label{nl:semigroup-holder}
\begin{aligned}
&\frac{
\bigl(e^{-t(-\Delta_{x_i})^{\alpha/2}}
\varphi_{\alpha,s,\varepsilon}\bigr)(x)}
{\varphi_{\alpha,s,\varepsilon}(x)}
\\
&\quad\le
\prod_{j\ne i}
\left[
\frac{
1+s^{(d-\alpha)/(2\alpha)}
\bigl(
e^{-t(-\Delta)^{\alpha/2}}
|\cdot|_\varepsilon^{-(d-\alpha)/2}
\bigr)(x_i-x_j)}
{
1+s^{(d-\alpha)/(2\alpha)}
|x_i-x_j|_\varepsilon^{-(d-\alpha)/2}}
\right]^{\theta/(N-1)}.
\end{aligned}
\end{equation}

We are going to differentiate both sides of \eqref{nl:semigroup-holder}. Note that the denominator on the
left-hand side of \eqref{nl:semigroup-holder} does not depend on $t$,
so the definition of the generator gives
\[
 \left.\frac{d}{dt}
 \frac{\bigl(e^{-t(-\Delta_{x_i})^{\alpha/2}}
             \varphi_{\alpha,s,\varepsilon}\bigr)(x)}
      {\varphi_{\alpha,s,\varepsilon}(x)}
 \right|_{t=0+}
 =-\frac{\bigl((-\Delta_{x_i})^{\alpha/2}
               \varphi_{\alpha,s,\varepsilon}\bigr)(x)}
        {\varphi_{\alpha,s,\varepsilon}(x)}.
\]
For each factor on the right-hand side, the ratio inside the brackets
equals $1$ at $t=0$, and its denominator is again independent of $t$.
Therefore,
\[
 \begin{aligned}
 &\left.\frac{d}{dt}
 \left[
 \frac{1+s^{(d-\alpha)/(2\alpha)}
       \bigl(e^{-t(-\Delta)^{\alpha/2}}
       |\cdot|_\varepsilon^{-(d-\alpha)/2}\bigr)(x_i-x_j)}
      {1+s^{(d-\alpha)/(2\alpha)}
       |x_i-x_j|_\varepsilon^{-(d-\alpha)/2}}
 \right]^{\theta/(N-1)}\right|_{t=0+}
 \\
 &\qquad=-\frac{\theta s^{(d-\alpha)/(2\alpha)}}{N-1}
 \frac{\bigl((-\Delta)^{\alpha/2}
       |\cdot|_\varepsilon^{-(d-\alpha)/2}\bigr)(x_i-x_j)}
      {1+s^{(d-\alpha)/(2\alpha)}
       |x_i-x_j|_\varepsilon^{-(d-\alpha)/2}}.
 \end{aligned}
\]
Indeed, differentiating the outer power contributes the factor
$\theta/(N-1)$, since the remaining power of the ratio equals $1$
at $t=0$. Differentiating the semigroup contributes the minus sign.
By the product rule, the derivative of the right-hand side of
\eqref{nl:semigroup-holder} is the sum of these expressions over
$j\ne i$: in each term, all the factors that are not differentiated
equal $1$ at $t=0$.

Since both sides of \eqref{nl:semigroup-holder} equal $1$
at $t=0$, subtracting $1$, dividing by $t>0$, and letting
$t\downarrow0$ shows that the derivative of the left-hand side is
at most the derivative of the right-hand side. Substituting the
expressions above and multiplying by $-1$ reverses the inequality
and yields \eqref{ineq:main}. The rest of the proof stays the same.
\end{remark}

We are in a position to verify condition \eqref{S4} of Theorem A.
We argue as in the proof of Lemma~\ref{planar:weighted-lemma}.
Fix $s,\varepsilon>0$ and write
$\varphi=\varphi_{\alpha,s,\varepsilon}$. Let
$0\le f\in C_c^\infty(\mathbb R^{dN})$ and put
$u(t)=e^{-t\Lambda_{\alpha,\varepsilon}}\varphi^{-1}f$.
Thus $u(t)\ge0$ and
\[
 \partial_tu=-A_\alpha u-b_{\alpha,\varepsilon}\cdot\nabla u,
 \qquad u(0)=\varphi^{-1}f.
\]
Our goal is to estimate $\int_{\mathbb R^{dN}}\varphi u(t)\,dx$
whose value at $t=0$ is $\|f\|_1$. Note that, since $b_{\alpha,\varepsilon}$ is bounded and smooth, we are in the setting of the classical theory. Multiplying the equation by $\varphi$ and integrating, we 
obtain
\[
 \begin{aligned}
 \frac d{dt}\int_{\mathbb R^{dN}}\varphi u(t)\,dx
 &=-\int_{\mathbb R^{dN}}\varphi
       (A_\alpha u(t)+b_{\alpha,\varepsilon}\cdot\nabla u(t))\,dx\\
 &=\int_{\mathbb R^{dN}}u(t)
       \bigl(-A_\alpha\varphi+
             \operatorname{div}(\varphi b_{\alpha,\varepsilon})\bigr)\,dx\\
 &=-\int_{\mathbb R^{dN}}u(t)
                         \Lambda_{\alpha,\varepsilon}^*\varphi\,dx\\
 &\le\frac Cs\int_{\mathbb R^{dN}}\varphi u(t)\,dx,
 \end{aligned}
\]
where the last inequality follows from \eqref{nl:adjoint-bound}.
The Gronwall inequality now yields 
\begin{equation*}
 \|\varphi e^{-t\Lambda_{\alpha,\varepsilon}}\varphi^{-1}f\|_1
 =\int_{\mathbb R^{dN}}\varphi u(t)\,dx
 \le e^{Ct/s}\|f\|_1,\qquad t\ge0.
\end{equation*}
Using
$|\varphi e^{-t\Lambda_{\alpha,\varepsilon}}\varphi^{-1}f|
\le\varphi e^{-t\Lambda_{\alpha,\varepsilon}}\varphi^{-1}|f|$
and applying a standard density argument, we extend this estimate
to all $f\in L^1$.

\subsection{Application of Theorem A}
For fixed
$s,\varepsilon>0$, take $a=\frac{dN}{\alpha}$, 
 $\varphi=\varphi_{\alpha,s,\varepsilon}$ and apply Theorem A on $0<t\le s$ - its conditions \eqref{S2}, \eqref{S3} are straightforward, and conditions \eqref{S1}, \eqref{S4} were verified above. We obtain
\begin{equation*}
 0\le p_{\alpha,\varepsilon}(t,x,y)
 \le C t^{-dN/\alpha}\varphi_{\alpha,s,\varepsilon}(y),
 \qquad 0<t\le s.
\end{equation*}
We now choose $s=t$. This gives \eqref{nl:kernel-eps} and ends the proof of assertion (\textit{i}) of Theorem \ref{nl:main}

\subsection{A posteriori heat kernel bound} We now prove assertions (\textit{ii}) and (\textit{iii}). 
Once again, we will use Trotter's
approximation theorem (Appendix~\ref{app_B}). We will again use a resolvent representation to show the convergence. However,  this resolvent representation is different from the one in the proof of Theorem \ref{thm1}. Namely, we can no longer simply get rid of the gradient without introducing a dependence on $N$ in the condition on $\nu$. Instead, we need to take advantage of the first inequality \eqref{nl:drift-estimate} in Lemma \ref{nl:drift-lemma} which, recall, can be viewed as a many-particle Hardy-type inequality.

For $\mu>0$, define operator $K_{\varepsilon,\mu}$ ($=$ denominator of the geometric series in the resolvent representation), initially on
the dense subspace $(\mu+A_\alpha)^{1/2}C_c^\infty$, by the formula
\[
 K_{\varepsilon,\mu}\bigl((\mu+A_\alpha)^{1/2}f\bigr)
 =(\mu+A_\alpha)^{-1/2}(b_{\alpha,\varepsilon}\cdot\nabla f),
 \qquad f\in C_c^\infty.
\]
The density assertion follows because $C_c^\infty$ is dense in
$W^{\alpha/2,2}$ for the norm $\|(\mu+A_\alpha)^{1/2}\cdot\|_2$, while
$(\mu+A_\alpha)^{1/2}:W^{\alpha/2,2}\to L^2$ is onto. We extend this operator to $L^2$. Note that $\mathcal E_\alpha(u)=\|A_\alpha^{\frac{1}{2}}u\|^2_2$ and estimate
\begin{align*}
\|(\mu+A_\alpha)^{-\frac{1}{2}}b_{\alpha,\varepsilon} \cdot \nabla f\|_2 & = \sup_{\|h\|_2\leq 1} \langle (\mu+A_\alpha)^{-\frac{1}{2}}b_{\alpha,\varepsilon} \cdot \nabla f,h \rangle \\
& = \sup_{\|h\|_2\leq 1} \langle b_{\alpha,\varepsilon} \cdot \nabla f,(\mu+A_\alpha)^{-\frac{1}{2}} h \rangle \\
& (\text{use \eqref{nl:drift-estimate} under $\sup$}) \\
& \leq \sup_{\|h\|_2\leq 1}  C_{d,\alpha}\nu\frac{N-1}{N}\|A_\alpha^{\frac{1}{2}}f\|^2 \|A^{\frac{1}{2}}(\mu+A_\alpha)^{-\frac{1}{2}} h\|_2 \\
& (\text{apply $\|A_\alpha^{1/2}(\mu+A_\alpha)^{-1/2}\|_{2\to2}\le1$ in the last term}) \\
& \leq C_{d,\alpha}\nu\frac{N-1}{N}\|A_\alpha^{\frac{1}{2}}f\|^2 \\
& \leq C_{d,\alpha}\nu\frac{N-1}{N}
       \|(\mu+A_\alpha)^{1/2}f\|_2.
\end{align*}
This estimate yields that $K_{\varepsilon,\mu}$ extends to a bounded operator on
$L^2$, with a bound uniform in $\mu,\varepsilon>0$. To invert $I+K_{\varepsilon,\mu}$, as will be needed to construct the resolvent representation, we use the coercivity estimate in Lemma \ref{nl:drift-lemma}.
For $g=(\mu+A_\alpha)^{1/2}f$, $f\in C_c^\infty$, it gives
\[
 \begin{split}
 \operatorname{Re}\langle(I+K_{\varepsilon,\mu})g,g\rangle
 &=\mu\|f\|_2^2+
   \operatorname{Re}\langle\Lambda_{\alpha,\varepsilon}f,f\rangle\\
 &\ge \mu\|f\|_2^2+\tfrac12\mathcal E_\alpha(f)
 \ge\tfrac12\|g\|_2^2.
 \end{split}
\]
By density this inequality holds for every $g\in L^2$. It also holds
for the real part of the adjoint of $I+K_{\varepsilon,\mu}$. The first
bound implies that the operator is bounded below and has closed range, and
the bound for its adjoint makes that range dense (see e.g.\,the argument in \cite[Ch.IX, Sect.\,8]{Y}). The range is therefore
all of $L^2$, and we have
\[
 (I+K_{\varepsilon,\mu})^{-1}\in\mathcal B(L^2),
 \qquad\|(I+K_{\varepsilon,\mu})^{-1}\|_{2\to2}\le2.
\]
With this estimate at hand, the regularized
resolvent can be written as
\begin{equation*}
 (\mu+\Lambda_{\alpha,\varepsilon})^{-1}
 =(\mu+A_\alpha)^{-1/2}(I+K_{\varepsilon,\mu})^{-1}
                         (\mu+A_\alpha)^{-1/2}.
\end{equation*}
Indeed, if $u=(\mu+\Lambda_{\alpha,\varepsilon})^{-1}f\in W^{\alpha,2}$,
then approximation in $W^{\alpha,2}$ by smooth compactly supported
functions gives
\[
 (I+K_{\varepsilon,\mu})(\mu+A_\alpha)^{1/2}u
 =(\mu+A_\alpha)^{-1/2}(\mu+\Lambda_{\alpha,\varepsilon})u
 =(\mu+A_\alpha)^{-1/2}f
\]
so applying
the inverse of $I+K_{\varepsilon,\mu}$ yields the representation.

We now verify conditions of Trotter's approximation theorem:

$1^\circ)$ The resolvent power bound follows simply from
$L^2$ contractivity of $e^{-t\Lambda_{\alpha,\varepsilon}}$:
\[
 \sup_{\varepsilon>0}
 \|(\mu+\Lambda_{\alpha,\varepsilon})^{-m}\|_{2\to2}
 \le\mu^{-m},\qquad\mu>0,\quad m=1,2,\ldots.
\]

$2^\circ)$ As in the proof of Theorem \ref{thm1}, we need to verify strong convergence of the regularized
resolvents at some fixed $\mu>0$.
For $f\in C_c^\infty$, we have
\begin{equation*}
 (b_{\alpha,\varepsilon}-b_\alpha)\cdot\nabla f\longrightarrow0
 \quad\hbox{in }L^2.
\end{equation*}
In terms of the definition of $K_{\varepsilon,\mu}$, this is
strong convergence on $(\mu+A_\alpha)^{1/2}C_c^\infty$. The
subspace is dense and the operators are uniformly bounded, therefore the
convergence extends to all of $L^2$. We denote the resulting
bounded operator by $K_\mu$.
The coercivity inequality passes to the strong limit:
\[
 \operatorname{Re}\langle(I+K_\mu)g,g\rangle
 \ge\tfrac12\|g\|_2^2,\qquad g\in L^2.
\]
We can consequently repeat the closed-range and adjoint argument
used for the regularized operators: $I+K_\mu$ is invertible,
with inverse norm at most $2$. To pass from convergence of the
operators to convergence of their inverses, we use the identity
\[
 (I+K_{\varepsilon,\mu})^{-1}-(I+K_\mu)^{-1}
 =(I+K_{\varepsilon,\mu})^{-1}
       (K_\mu-K_{\varepsilon,\mu})(I+K_\mu)^{-1}
\]
which, together with the uniform bound on the inverses, proves their
strong convergence. Hence, for every
$\mu>0$, the regularized resolvents converge strongly to operator
\begin{equation*}
 R_\mu=(\mu+A_\alpha)^{-1/2}
       (I+K_\mu)^{-1}(\mu+A_\alpha)^{-1/2}.
\end{equation*}

$3^\circ)$ We verify the uniform (in $\varepsilon$) approximation of the identity as $\mu \rightarrow \infty$. It is enough to begin on a fixed dense subset.
It is easily seen that (since $|b|$ is bounded from above by the double sum of $|x_i-x_j|^{1-\alpha}$, which is in $L_{loc}^2$), for
every $f\in C_c^\infty$,
\[
 C_f:=\sup_{\varepsilon>0}
       \|\Lambda_{\alpha,\varepsilon}f\|_2<\infty.
\]
The resolvent identity then yields
\[
 \begin{split}
 \sup_{\varepsilon>0}
 \|\mu(\mu+\Lambda_{\alpha,\varepsilon})^{-1}f-f\|_2
 &=\sup_{\varepsilon>0}
 \|(\mu+\Lambda_{\alpha,\varepsilon})^{-1}
               \Lambda_{\alpha,\varepsilon}f\|_2\\
 &\le \frac{C_f}{\mu}\longrightarrow0\qquad(\mu\to\infty).
 \end{split}
\]
The uniform bound
$\|\mu(\mu+\Lambda_{\alpha,\varepsilon})^{-1}\|_{2\to2}\le1$
allows us to extend this convergence from $f \in C_c^\infty$ to
every $f\in L^2$. 

We have thus verified all three conditions of
Trotter's approximation theorem.
The theorem gives a strongly continuous contraction semigroup
$e^{-t\Lambda_\alpha}$ on $L^2$ whose resolvent is $R_\mu$ such that
\begin{equation*}
 \lim_{\varepsilon,\eta\downarrow0}\sup_{0\le t\le T}
 \|e^{-t\Lambda_{\alpha,\varepsilon}}f-e^{-t\Lambda_{\alpha,\eta}}f\|_2=0,
 \qquad f\in L^2,\quad T<\infty.
\end{equation*}
It is easily seen that the limiting semigroup is positivity preserving and an $L^\infty$ contraction. The same interpolation argument as in the construction of $e^{-t\Lambda_{\alpha,\varepsilon}}$ allows us to extend $e^{-t\Lambda_\alpha}$ to a strongly continuous semigroup on $L^r$, $2 \leq r<\infty$, so that also $e^{-t\Lambda_\alpha}f=\lim_{\varepsilon\downarrow0}e^{-t\Lambda_{\alpha,\varepsilon}}f$ in $L^r$.
This ends the proof of (\textit{ii}).

Finally, (\textit{iii}) follows by applying the Dunford-Pettis theorem, to establish the existence of the integral kernel of $e^{-t\Lambda_\alpha}$, and then using the Lebesgue differentiation theorem in order to derive the a posteriori heat kernel bound from the a priori heat kernel bound of (\textit{i}). \hfill \qed

\bigskip

\section{Proof of Theorem \ref{thm2}}
\label{proof:higher}

The proof of the a priori ($\varepsilon>0$) upper bound is a somewhat more detailed variant of the proof in \cite{MS}. The construction of the limiting semigroup and of the a posteriori heat kernel as $\varepsilon \downarrow 0$ follows \cite{Ki-Osaka}, see also \cite{KiS-theory}.

\subsection{A priori upper bound}
Fix $s,\varepsilon>0$ and put
$$
 \bar b_\varepsilon \equiv \bar b_{s,\varepsilon}:=-\frac{\nabla\varphi}{\varphi}, \quad \text{ where }\varphi \equiv \varphi_{s,\varepsilon} =\frac{\psi_\varepsilon}{\psi_{\varepsilon+s}}. 
$$
This gives 
\begin{equation}\label{higher:weight-drift}
 \bar b_\varepsilon=b_\varepsilon-b_{\varepsilon+s}.
\end{equation}
Denote
$$
\langle f \rangle_\varphi:=\int_{\mathbb{R}{dN}} f(x)\varphi(x) dx, \qquad \langle f,g\rangle:=\langle f\bar{g}\rangle.
$$

\subsubsection*{Reduction to the Dirichlet form setting}

\begin{lemma}\label{higher:drift-lemma}
\begin{equation}\label{bb_est}
 |b_{\varepsilon+s}|^2\leq \frac{\nu^2(N-1)}{8s}.
\end{equation}
\end{lemma}
\begin{proof}
We first note that
$$
 \sum_{i<j}|b_{\varepsilon+s,i}-b_{\varepsilon+s,j}|^2  =N|b_{\varepsilon+s}|^2 -\left|\sum_{i=1}^N b_{\varepsilon+s,i}\right|^2
 \leq N|b_{\varepsilon+s}|^2.
$$
By the definition of $b_{\varepsilon+s}$, we obtain 
\[
\begin{aligned}
 |b_{\varepsilon+s}|^2 = \sum_{i=1}^N b_{\varepsilon+s,i}\cdot b_{\varepsilon+s,i}
 & =\frac{\nu}{N}\sum_{i=1}^N\sum_{j\ne i}
   \frac{x_i-x_j}{|x_i-x_j|^2+\varepsilon+s}
       \cdot b_{\varepsilon+s,i}\\ 
 &=\frac{\nu}{N}\sum_{i<j}
   \frac{x_i-x_j}{|x_i-x_j|^2+\varepsilon+s}
   \cdot(b_{\varepsilon+s,i}-b_{\varepsilon+s,j})\\
 &\leq\frac{\nu}{N}
   \left(\sum_{i<j}
   \frac{|x_i-x_j|^2}
        {(|x_i-x_j|^2+\varepsilon+s)^2}\right)^{1/2}
   \left(\sum_{i<j}
   |b_{\varepsilon+s,i}-b_{\varepsilon+s,j}|^2\right)^{1/2}\\
 &\leq\frac{\nu}{\sqrt N}
   \left(\sum_{i<j}
   \frac{|x_i-x_j|^2}
        {(|x_i-x_j|^2+\varepsilon+s)^2}\right)^{1/2}
   |b_{\varepsilon+s}|.
\end{aligned}
\]
We thus obtain 
\[
 |b_{\varepsilon+s}|^2  \leq\frac{\nu^2}{N}\sum_{i<j} \frac{|x_i-x_j|^2} {(|x_i-x_j|^2+\varepsilon+s)^2},
\]
with the case $|b_{\varepsilon+s}|=0$ being immediate. Since
$
 \frac{|x_i-x_j|^2}
      {(|x_i-x_j|^2+\varepsilon+s)^2}
 \leq\frac{1}{4(\varepsilon+s)},
$
summing over the $N(N-1)/2$ pairs yields
$$
 |\bar b_\varepsilon-b_\varepsilon|^2  =|b_{\varepsilon+s}|^2  \leq\frac{\nu^2(N-1)}{8(\varepsilon+s)} \leq\frac{\nu^2(N-1)}{8s}.
$$
\end{proof}
For fixed $\varepsilon>0$, the weight $\varphi$ is smooth, bounded from above and satisfies $\varphi\geq1$. Define sesquilinear form
\[
 a_\varepsilon^\prime[u,w]:=\langle\nabla u,\nabla w\rangle_\varphi,  \qquad D(a^\prime)=W^{1,2}(\mathbb R^{dN}),
\]
and let $A_\varepsilon$ be its self-adjoint operator in $L^2_\varphi$. By \eqref{higher:weight-drift},
\[
 A_\varepsilon=-\varphi^{-1}\text{div}(\varphi\nabla)  =-\Delta+\bar b_\varepsilon\cdot\nabla,
 \qquad \Lambda_\varepsilon=A_\varepsilon+b_{\varepsilon+s}\cdot\nabla.
\]
Then the sesquilinear form associated to the operator $\Lambda_\varepsilon$ is
\begin{equation}\label{form}
  a_\varepsilon[u,w] :=\langle\Lambda_\varepsilon u,w\rangle_\varphi  = a_\varepsilon^\prime[u,w]+\langle b_{\varepsilon+s}\cdot\nabla u,w\rangle_\varphi.
\end{equation}
Let $q_{\varepsilon}(t,x,y):=e^{-t \Lambda_\varepsilon}(x,y)$ denote the heat kernel of the semigroup $e^{-t \Lambda_\varepsilon}$ acting in $L^2_{\varphi}$, so that
\begin{equation}\label{higher:q-def}
 q_\varepsilon(t,x,y)\varphi(y)=p_\varepsilon(t,x,y).
\end{equation}
Hence, the proof of the upper bound reduces to  establishing Gaussian upper bound 
\begin{equation*}
 q_\varepsilon(t,x,y) \leq  c_0 e^{\frac{c t}{s}}\Gamma_{c_2t}(x-y).
\end{equation*}
This is what we do next.
In what follows, all constants are independent of $\varepsilon>0$.

\subsubsection*{Weighted Sobolev embedding} Denote by 
$
\ell:=\frac{dN}{dN-2}
$
the Sobolev exponent in $\mathbb R^{dN}$.
Set $$
\|u\|_{p, \varphi}:=\langle |u|^p \varphi\rangle^{1/p}.
$$
The next estimates will be used in the proof of the weighted Sobolev embedding (Lemma \ref{sob_emb}). 

\begin{lemma}\label{comp_lem}
For $d\geq3$, 
\begin{equation}\label{higher:div-form}
0 \leq \langle \operatorname{div}b_\varepsilon\,|f|^2 \rangle \le  \frac{2\nu}{d-2} \langle |\nabla f|^2 \rangle \qquad f\in W^{1,2}(\mathbb R^{dN}).
\end{equation}
and
\begin{equation}\label{phi_calc}
 \langle -\Delta\log \varphi\,|f|^2\rangle \le  \frac{2\nu}{d-2} \langle |\nabla f|^2 \rangle \qquad f\in W^{1,2}(\mathbb R^{dN}).
\end{equation}
\end{lemma}

\begin{proof}
We have
$$
 \operatorname{div}b_\varepsilon  =\frac{2\nu}{N}\sum_{i<j}  \frac{(d-2)|x_i-x_j|^2+d\varepsilon} {(|x_i-x_j|^2+\varepsilon)^2}
 \geq0.
$$
Arguing as in Lemma~\ref{higher:drift-lemma}, we obtain
$$
 |b_\varepsilon|^2 \leq\frac{\nu^2}{N}\sum_{i<j} \frac{|x_i-x_j|^2}{(|x_i-x_j|^2+\varepsilon)^2}
 \leq\frac{\nu}{2(d-2)}\operatorname{div}b_\varepsilon.
$$
For $f\in C_c^\infty(\mathbb R^{dN})$, integration by parts, Cauchy-Schwarz, and the quadratic inequality yield
\[
\begin{aligned}
 \langle\operatorname{div}b_\varepsilon\,|f|^2\rangle
 &=-2\operatorname{Re}
       \langle b_\varepsilon\cdot\nabla f, \overline f\rangle\\
 &\leq2\langle|b_\varepsilon|^2|f|^2\rangle^{1/2}
          \langle|\nabla f|^2\rangle^{1/2}\\
 &\leq\sqrt{\frac{2\nu}{d-2}}\,
       \langle\operatorname{div}b_\varepsilon\,|f|^2\rangle^{1/2}
       \langle|\nabla f|^2\rangle^{1/2}\\
 &\leq\frac12
       \langle\operatorname{div}b_\varepsilon\,|f|^2\rangle
       +\frac{\nu}{d-2}\langle|\nabla f|^2\rangle.
\end{aligned}
\]
Absorbing the first term on the right proves \eqref{higher:div-form}. Finally, since
$\varphi=\psi_\varepsilon/\psi_{\varepsilon+s}$,
\[
 -\Delta\log\varphi =\operatorname{div}b_\varepsilon  -\operatorname{div}b_{\varepsilon+s} \leq\operatorname{div}b_\varepsilon.
\]
Multiplying by $|f|^2$, integrating, and applying \eqref{higher:div-form} proves \eqref{phi_calc}.
Both inequalities extend to $W^{1,2}(\mathbb R^{dN})$ by density.
\end{proof}

The next lemma is proved under a somewhat less restrictive condition
on $\nu$ compared to Theorem~\ref{thm2}.

\begin{lemma}\label{sob_emb}
If $0\leq\nu<2\ell(d-2)$, then
\begin{equation*}
 \|u\|_{2\ell,\varphi}^2\leq C a^{\prime}[u,u],
 \qquad u\in W^{1,2}(\mathbb R^{dN}),
\end{equation*}
where $C$ depends only on $d,N,\nu$. 
\end{lemma}

\begin{proof}
It suffices to consider real-valued $u\in C_c^\infty(\mathbb R^{dN})$ and then use a standard density argument. By the usual Sobolev inequality on
$\mathbb R^{dN}$,
\begin{equation}\label{int_0}
 \|u\|_{2\ell,\varphi}^2
 =\|u\varphi^{1/(2\ell)}\|_{2\ell}^2
 \leq C_S\big\langle
          |\nabla(u\varphi^{1/(2\ell)})|^2
        \big\rangle.
\end{equation}
Integration by parts gives
\begin{equation*}
\begin{aligned}
\big\langle |\nabla(u\varphi^{1/(2\ell)})|^2  \big\rangle  &=\big\langle
     \varphi^{1/(2\ell)}\nabla u,
     \nabla(u\varphi^{1/(2\ell)})
   \big\rangle 
 +\frac{1}{4\ell}  \big\langle  \nabla\log\varphi,  \nabla(u^2\varphi^{1/\ell})   \big\rangle\\
 &=\big\langle   \varphi^{1/(2\ell)}\nabla u,   \nabla(u\varphi^{1/(2\ell)})   \big\rangle
 + \frac{1}{4\ell}   \big\langle   -\Delta\log\varphi,    |u\varphi^{1/(2\ell)}|^2   \big\rangle.
\end{aligned}
\end{equation*}
Applying Cauchy-Schwarz to the first term in the RHS and \eqref{phi_calc} to the second gives
\[
\begin{aligned}
 \big\langle|\nabla(u\varphi^{1/(2\ell)})|^2\big\rangle
 &\leq
   \big\langle|\nabla u|^2\varphi^{1/\ell}\big\rangle^{1/2}
   \big\langle|\nabla(u\varphi^{1/(2\ell)})|^2
   \big\rangle^{1/2}
 +\frac{\nu}{2\ell(d-2)}
   \big\langle|\nabla(u\varphi^{1/(2\ell)})|^2\big\rangle.
\end{aligned}
\]
Since $\nu<2\ell(d-2)$, we conclude that
\begin{equation*}
 \big\langle|\nabla(u\varphi^{1/(2\ell)})|^2\big\rangle
 \leq
 \left(1-\frac{\nu}{2\ell(d-2)}\right)^{-2}
 \big\langle|\nabla u|^2\varphi^{1/\ell}\big\rangle.
\end{equation*}
Finally, $\varphi\geq1$ and $\ell>1$ imply
$\varphi^{1/\ell}\leq\varphi$. Hence
\[
 \big\langle|\nabla(u\varphi^{1/(2\ell)})|^2\big\rangle
 \leq C\big\langle|\nabla u|^2\varphi\big\rangle
 =C a^{\prime}[u,u].
\]
This and \eqref{int_0}, give us the assertion of the lemma.
\end{proof}

\begin{remark}
The proof given above is short but suboptimal. As it pointed out in \cite{FKS}, one can prove the weighted Sobolev embedding directly, via spectral gap and doubling inequalities, without passing through the usual Sobolev embedding on the Euclidean space. This allows to relax the conditions on $\nu$ and also to treat dimension $d=2$. See discussion in the introduction.
\end{remark}

\subsubsection*{Gaussian upper bound on $q(t,x,y)$}
We now apply the ``Davies' device'' and Moser's iterations in the weighted setting following Fabes-Stroock \cite{FS}. For $\theta \in\mathbb R^{dN}$, define
weight $
 \phi_\theta(x):=e^{\theta\cdot x}, x\in\mathbb R^{dN},
$
and put
$$
\begin{aligned}
 \Lambda_{\theta}
 &:=\phi_\theta\Lambda \phi_{-\theta}\\
 &=\Lambda +2(\theta\cdot\nabla)  +\theta\cdot\frac{\nabla\varphi}{\varphi}   -\theta\cdot b_{\varepsilon+s}-|\theta|^2.
\end{aligned}
$$
The sesquilinear form of the operator $\Lambda_{\theta}$ is
\begin{equation}\label{A_exp}
 \langle\Lambda_{\theta}u,w\rangle_\varphi
 =a[\phi_{-\theta}u,\phi_\theta w].
\end{equation}
Indeed, by \eqref{form},
\[
\begin{aligned}
 a[\phi_{-\theta}u,\phi_\theta w]
 &=a^\prime[\phi_{-\theta}u,\phi_\theta w] +\big\langle  b_{\varepsilon+s}\cdot(\nabla u-\theta u),w
   \big\rangle_\varphi.
\end{aligned}
\]
\begin{lemma}[Moser's lemma]\label{moser_lem} For all $t>0$,
\[
 \|e^{-t\Lambda_{\theta}}\|_{L^2_\varphi\to L^\infty}
 \leq C_1t^{-dN/4}e^{C_2t/s}e^{C_3|\theta|^2t},
\]
where $C_1,C_2$, and $C_3$ are independent of $\varepsilon$. The same bound holds for the adjoint semigroup
$e^{-t\Lambda_\theta^*}$.  
\end{lemma}

\begin{proof}
Define function
$$
 u(t):=e^{-t\Lambda_{\theta}}f, \qquad 0\leq f\in C_c^\infty(\mathbb R^{dN}),  \quad f\not\equiv0.
$$
Our goal is to prove
\begin{equation}\label{sbd}
 \|u(t)\|_\infty
 \leq Ct^{-dN/4}e^{Ct/s}e^{C|\theta|^2t}
       \|f\|_{L_\varphi^2},
 \qquad t>0.
\end{equation}

Let $p\geq2$.  We multiply in $L_\varphi^2$ the parabolic equation for $u$ by $u^{p-1}$, and setting for brevity $v:=u^{p/2}$, obtaining
\begin{equation}\label{A_first_eq}
 \frac1p\partial_t\langle v,v\rangle_\varphi +\langle\Lambda_{\theta}u,u^{p-1}\rangle_\varphi  =0.
\end{equation}
Using \eqref{A_exp}, we obtain
\begin{equation}\label{form_exp}
\begin{aligned}
 \langle\Lambda_{\theta}u,u^{p-1}\rangle_\varphi
 &=a[\phi_{-\theta}u,\phi_\theta u^{p-1}]\\
 &=\frac{4(p-1)}{p^2}a^\prime[v,v]
   -\frac{2(p-2)}p
      \langle\theta\cdot\nabla v,v\rangle_\varphi\\
 &\quad+\frac{2}{p}
      \langle b_{\varepsilon+s}\cdot\nabla v,v\rangle_\varphi
   -|\theta|^2\langle v,v\rangle_\varphi\\
 &\quad-\langle(\theta\cdot b_{\varepsilon+s})v,v
          \rangle_\varphi.
\end{aligned}
\end{equation}
Using \eqref{bb_est} and the quadratic inequality, we estimate the mixed terms as follows:
$$
\begin{aligned}
 &\left|
 -\frac{2(p-2)}p\langle\theta\cdot\nabla v,v\rangle_\varphi
 +\frac2p\langle b_{\varepsilon+s}\cdot\nabla v,v\rangle_\varphi
 \right|\\
 &\qquad\leq
 2\bigl(|\theta|+\|b_{\varepsilon+s}\|_\infty\bigr)
       \langle|\nabla v|,v\rangle_\varphi\\
 &\qquad\leq
 \frac1p a^\prime[v,v]
 +Cp\left(|\theta|^2+\frac1s\right)
       \langle v,v\rangle_\varphi.
\end{aligned}
$$
For the remaining terms, we first use the quadratic inequality and \eqref{bb_est} to obtain
\[
\begin{aligned}
 \left|
 \langle(\theta\cdot b_{\varepsilon+s})v,v\rangle_\varphi
 \right|
 &\leq|\theta|\,\|b_{\varepsilon+s}\|_\infty
       \langle v,v\rangle_\varphi\\
 &\leq\left(
       \frac12|\theta|^2
       +\frac12\|b_{\varepsilon+s}\|_\infty^2
       \right)\langle v,v\rangle_\varphi\\
 &\leq\left(
       \frac12|\theta|^2
       +\frac{\nu^2(N-1)}{16s}
       \right)\langle v,v\rangle_\varphi.
\end{aligned}
\]
Hence,
$$
\begin{aligned}
 \left|
 -|\theta|^2\langle v,v\rangle_\varphi
 -\langle(\theta\cdot b_{\varepsilon+s})v,v\rangle_\varphi
 \right|
& \leq
 \left(\frac32|\theta|^2+\frac{\nu^2(N-1)}{16s}\right)
 \langle v,v\rangle_\varphi\\
 &\leq
 C\left(|\theta|^2+\frac1s\right)
 \langle v,v\rangle_\varphi. 
\end{aligned}
$$

Since $4(p-1)/p^2\geq2/p$, it follows that
\[
 \langle\Lambda_{\theta}u,u^{p-1}\rangle_\varphi
 \geq\frac1p a^\prime[v,v]
      -Cp\left(|\theta|^2+\frac1s\right)
           \langle v,v\rangle_\varphi.
\]
Substituting into \eqref{A_first_eq} and applying Lemma~\ref{sob_emb}, i.e.the weighted Sobolev embedding, we obtain
\begin{equation}\label{L2 bound}
 -\partial_t\|v\|_{2,\varphi}^2
 \geq c\|v\|_{2\ell,\varphi}^2
      -Cp^2\left(|\theta|^2+\frac1s\right)
            \|v\|_{L_\varphi^2}^2.
\end{equation}
Using interpolation inequality,
$$
 \|v\|_{2,\varphi}
 \leq\|v\|_{1,\varphi}^{2/(dN+2)}
       \|v\|_{2\ell,\varphi}^{dN/(dN+2)}.
$$
Thus,
\begin{equation}\label{energy_eque}
 -\partial_t\|v\|_{2,\varphi}^2
 \geq c\|v\|_{2,\varphi}^{2+4/(dN)}
          \|v\|_{1,\varphi}^{-4/(dN)}
      -Cp^2\left(|\theta|^2+\frac1s\right)
            \|v\|_{L_\varphi^2}^2.
\end{equation}
From here, one can run the standard Moser's iterations (see, e.g., \cite[Sect.\,4.2]{KiS_MAAN}) to obtain the desired bound. To make the paper self-contained, we include the details of the argument below. So, \eqref{energy_eque} gives
\[
\begin{aligned}
 \frac{d}{dt}\|v\|_{2,\varphi}^{-4/(dN)}
 &\geq \frac{2C_1}{dN}\|v\|_{1,\varphi}^{-4/(dN)} -\frac{2C}{dN}p^2
       \left(|\theta|^2+\frac1s\right)   \|v\|_{2,\varphi}^{-4/(dN)}.
\end{aligned}
\]
Let $p\geq4$. The preceding inequality is linear in
$$
 w_p(t):=\|v(t)\|_{2,\varphi}^{-4/(dN)}    =\|u(t)\|_{p,\varphi}^{-2p/(dN)}.
$$
Therefore, setting
$$
 c:=\frac{2C_1}{dN},
 \qquad
 \beta:=\frac{2C}{dN}\left(|\theta|^2+\frac1s\right),
 \qquad
 \mu_p(t):=\beta p^2t,
$$
we have
$$
\begin{aligned}
 w_p(t)
 &\geq ce^{-\mu_p(t)}   \int_0^t e^{\mu_p(r)}w_{p/2}(r)^2\,dr\\
 &\geq ce^{-\mu_p(t)}    \int_0^t e^{\mu_p(r)}r^q\,dr\;V_{p/2}(t),
\end{aligned}
$$
where $q=p/2-2$ and
\[
\begin{aligned}
 V_{p/2}(t)
 &:=\inf\left\{
       r^{-q}w_{p/2}(r)^2:\ 0<r\leq t
       \right\}\\
 &=\left\{
       \sup_{0<r\leq t}
       r^{qdN/(2p)}\|u(r)\|_{p/2,\varphi}
   \right\}^{-2p/(dN)}.
\end{aligned}
\]
Since
\[
 e^{-\mu_p(t)}\int_0^t e^{\mu_p(r)}r^q\,dr
 =e^{-\beta p^2t}\int_0^t e^{\beta p^2r}r^q\,dr
\]
and
\[
\begin{aligned}
 \int_0^t e^{\beta p^2r}r^q\,dr
 &=\left(\frac{t}{\beta p^2}\right)^{q+1}
   \int_0^{\beta p^2}e^{tr'}(r')^q\,dr'\\
 &\geq
   \left(\frac{t}{\beta p^2}\right)^{q+1}
   e^{\beta(p^2-1)t}
   \int_{\beta p^2(1-p^{-2})}^{\beta p^2}(r')^q\,dr'\\
 &=t^{(p-2)/2}\frac{2}{p-2}
   \left[1-(1-p^{-2})^{(p-2)/2}\right]
   e^{\beta(p^2-1)t}\\
 &\geq Kp^{-2}t^{(p-2)/2}e^{\beta(p^2-1)t},
\end{aligned}
\]
where $  K:=2\inf\left\{ p\left[1-(1-p^{-2})^{(p-2)/2}\right] \mid p\geq4 \right\}>0$, we obtain
$$
 w_p(t)\geq cKp^{-2}e^{-\beta t}t^{(p-2)/2}V_{p/2}(t),
$$
or, equivalently,
$$
 t^{-(p-2)/2}w_p(t)  \geq cKp^{-2}e^{-\beta t}V_{p/2}(t).
$$
Setting
\[
 W_p(t):=
 \sup_{0<r\leq t}
 r^{dN(p-2)/(4p)}\|u(r)\|_{p,\varphi}
 =V_p(t)^{-dN/(4p)},
\]
we conclude that
\[
 W_p(t)\leq
 (cK)^{-dN/(2p)}p^{dN/p}
 e^{C(|\theta|^2+s^{-1})t/p}W_{p/2}(t),
 \qquad p=2^k,\quad k=2,3,\ldots.
\]
Iterating this inequality, starting with $k=2$, yields
$$
 t^{dN/4}\|u(t)\|_\infty  \leq C_2e^{C(|\theta|^2+s^{-1})t}W_2(t).
$$
Finally, \eqref{L2 bound} implies
\[
 \frac{d}{dt}\|v\|_{2,\varphi}
 \leq\frac{C}{2}p^2
       \left(|\theta|^2+\frac1s\right)
       \|v\|_{2,\varphi}.
\]
Hence
\[
\begin{aligned}
 \|v(t)\|_{2,\varphi}
 &\leq e^{\frac{C}{2}p^2(|\theta|^2+s^{-1})t}
       \|v(0)\|_{2,\varphi}\\
 &=e^{\frac{C}{2}p^2(|\theta|^2+s^{-1})t}
       \langle|u(0)|^p,\varphi\rangle^{1/2}\\
 &=e^{\frac{C}{2}p^2(|\theta|^2+s^{-1})t}
       \langle|f|^p,\varphi\rangle^{1/2}.
\end{aligned}
\]
That is,
$$
 \|v(t)\|_{2,\varphi} \leq e^{\frac{C}{2}p^2(|\theta|^2+s^{-1})t}  \|f\|_{p,\varphi}^{p/2}, 
$$
so we obtain,
$$
 t^{dN/4}\|u(t)\|_\infty  \leq C_2e^{C_3t/s}e^{C_3|\theta|^2t}  \|f\|_{2,\varphi} \quad \Rightarrow \quad \eqref{sbd}.
$$
The same estimate holds for the adjoint semigroup in $L^2_\varphi$. Indeed, let
$
 u(t):=e^{-t\Lambda_\theta^*}f$, $0\leq f\in C_c^\infty(\mathbb R^{dN})$, and set $v=u^{p/2}$, $p\geq2$. We have
\[
 \frac1p\partial_t\langle v,v\rangle_\varphi
 +\langle\Lambda_\theta^*u,u^{p-1}\rangle_\varphi=0.
\]
Moreover,
\[
\begin{aligned}
 \langle\Lambda_\theta^*u,u^{p-1}\rangle_\varphi
 &=a[\phi_{-\theta}u^{p-1},\phi_\theta u]\\
 &=\frac{4(p-1)}{p^2}a^\prime[v,v]
   +\frac{2(p-2)}p
      \langle\theta\cdot\nabla v,v\rangle_\varphi\\
 &\quad+\frac{2(p-1)}p
      \langle b_{\varepsilon+s}\cdot\nabla v,v\rangle_\varphi
   -|\theta|^2\langle v,v\rangle_\varphi\\
 &\quad-\langle(\theta\cdot b_{\varepsilon+s})v,v\rangle_\varphi.
\end{aligned}
\]
Estimating the mixed terms as above, we obtain
$$
 \langle\Lambda_\theta^*u,u^{p-1}\rangle_\varphi
 \geq\frac1p a^\prime[v,v]
      -Cp\left(|\theta|^2+\frac1s\right)
            \langle v,v\rangle_\varphi.
$$
Thus, \eqref{L2 bound} also holds for the adjoint evolution. The same interpolation and iteration argument yields
$$
 \|e^{-t\Lambda_\theta^*}\|_{L^2_\varphi\to L^\infty}  \leq C_1t^{-dN/4}e^{C_2t/s}e^{C_3|\theta|^2t}.
$$

\end{proof}
From Lemma~\ref{moser_lem} and the corresponding estimate for
the adjoint semigroup, we obtain by duality
\[
 \|e^{-t\Lambda_{\theta}}\|_{L^1_\varphi
       \rightarrow L^2_\varphi}
 =
 \|(e^{-t\Lambda_{\theta}})^*\|_{L^2_\varphi
       \rightarrow L^\infty}
 \leq Ct^{-\frac{dN}{4}}
       e^{\frac{C_2}{s}t}e^{C_3|\theta|^2t}.
\]
Using the semigroup property, we therefore obtain
\[
 \|e^{-t\Lambda_{\theta}}\|_{L^1_\varphi
       \rightarrow L^\infty}
 \leq Ct^{-\frac{dN}{2}}
       e^{2\frac{C_2}{s}t}e^{2C_3|\theta|^2t}.
\]
Therefore, recalling the definition of the weights $\phi_\theta$ and $\phi_{-\theta}$, the integral kernel
$q(t,x,y)$ of $e^{-t\Lambda}$, with respect to the measure $\varphi(y)\,dy$, satisfies
\begin{equation}\label{q- Gaussian upper bound}
 q(t,x,y)
 \leq Ct^{-\frac{dN}{2}}
       e^{2\frac{C_2}{s}t}
       e^{\theta\cdot(x-y)+2C_3|\theta|^2t},
 \qquad t>0,\quad x,y\in\mathbb R^{dN}.
\end{equation}
Indeed, since
$$
 e^{-t\Lambda_{\theta}} =\phi_\theta e^{-t\Lambda}\phi_{-\theta},
$$
we have
$$
 (e^{-t\Lambda_{\theta}}f)(x) =  \big\langle e^{\theta\cdot(x-\cdot)} q(t,x,\cdot)f
 \big\rangle_\varphi.
$$
Thus, the integral kernel of $e^{-t\Lambda_{\theta}}$ is $e^{\theta\cdot(x-y)}q(t,x,y)$, and the $L^1_\varphi\rightarrow L^\infty$ bound gives
\[
 q(t,x,y)
 \leq Ct^{-\frac{dN}{2}}
       e^{2\frac{C_2}{s}t}
       e^{-\theta\cdot(x-y)+2C_3|\theta|^2t},
 \qquad t>0,\quad x,y\in\mathbb R^{dN}.
\]
Since $\theta$ is arbitrary, replacing it by $-\theta$ gives \eqref{q- Gaussian upper bound}.

\medskip

Selecting $\theta=(y-x)/(4C_3t)$ in \eqref{q- Gaussian upper bound}, we obtain
\begin{equation*}
 q(t,x,y)
 \leq Ct^{-dN/2}e^{2C_2t/s}
       e^{-|x-y|^2/(8C_3t)}.
\end{equation*}
This proves the desired Gaussian upper bound on $q(=q_\varepsilon)$. By \eqref{higher:q-def}, and select $s=t$  in the resulting upper bound on $p_{\varepsilon}(t,x,y)$. This ends the proof of assertion (\textit{i}).

\subsection{Limiting semigroup}
\label{semigroup_app}

\begin{lemma}\label{b_est11}
For every $\varepsilon>0$,
\[
 \langle|b_\varepsilon|^2,f^2\rangle
 \leq
 \frac{2\nu^2}{(d-2)^2}\frac{N-1}{N}
 \langle|\nabla f|^2\rangle,
 \qquad f\in W^{1,2}(\mathbb R^{dN}).
\]
Moreover,
\begin{equation}\label{b_est_sharp}
 \langle|b_\varepsilon|^2,f^2\rangle
 \leq\frac{\nu^2}{(d-2)^2}
       \langle|\nabla f|^2\rangle.
\end{equation}
\end{lemma}

\begin{proof}
By the definition of $b_\varepsilon$, we have
\[
\begin{aligned}
 |b_\varepsilon(x)|^2
 =\sum_{i=1}^N|b_{\varepsilon,i}(x)|^2 & =\nu^2\sum_{i=1}^N
   \left|
   \frac1N\sum_{j=1,\,j\ne i}^N
   \frac{x_i-x_j}{|x_i-x_j|^2+\varepsilon}
   \right|^2\\
 &\leq\nu^2\sum_{i=1}^N
   \left(
   \frac1N\sum_{j=1,\,j\ne i}^N
   \frac1{|x_i-x_j|}
   \right)^2\\
 &\leq\nu^2\sum_{i=1}^N
   \frac{N-1}{N^2}
   \sum_{j=1,\,j\ne i}^N\frac1{|x_i-x_j|^2}.
\end{aligned}
\]
So, switching to the summation above the diagonal,
$$
 \langle|b_\varepsilon|^2,f^2\rangle
 \leq
 2\nu^2\frac{N-1}{N^2}
 \sum_{1\leq i<j\leq N}
 \left\langle\frac1{|x_i-x_j|^2},f^2\right\rangle.
$$
Invoking now the many-particle Hardy inequality \eqref{multi_hardy} (or, rather, simply \eqref{hardy}), we obtain
$$
\begin{aligned}
 \langle|b_\varepsilon|^2,f^2\rangle
 &\leq
 2\nu^2\frac{N-1}{N^2}C_{d,N}^{-1}
 \langle|\nabla f|^2\rangle\\
 &\leq
 2\nu^2\frac{N-1}{N^2}\frac{N}{(d-2)^2}
 \langle|\nabla f|^2\rangle,
\end{aligned}
$$
which is the first claimed estimate.

For \eqref{b_est_sharp}, we use the pointwise estimate from the proof of Lemma~\ref{comp_lem},
\[
 |b_\varepsilon|^2
 \leq\frac{\nu}{2(d-2)}\operatorname{div}b_\varepsilon.
\]
Applying \eqref{higher:div-form}, we obtain
\[
 \langle|b_\varepsilon|^2,f^2\rangle
 \leq\frac{\nu}{2(d-2)}
       \langle\operatorname{div}b_\varepsilon,f^2\rangle
 \leq\frac{\nu^2}{(d-2)^2}
       \langle|\nabla f|^2\rangle.
\]
\end{proof}

\begin{lemma}\label{b_est1}
For every $\varepsilon>0$,
\[
 \big|\langle b_\varepsilon\cdot\nabla f,f\rangle\big|
 \leq\frac{\nu}{d-2}\langle|\nabla f|^2\rangle,
 \qquad f\in W^{1,2}(\mathbb R^{dN}).
\]
\end{lemma}

\begin{proof}
Integrating by parts, we obtain
\[
 \langle b_\varepsilon\cdot\nabla f,f\rangle
 =-\frac12\langle\operatorname{div}b_\varepsilon,f^2\rangle.
\]
Now we simply use \eqref{higher:div-form} from Lemma~\ref{comp_lem}.
\end{proof}
We will also need a weighted variant of the previous estimate. Set
\begin{equation}\label{rho_def}
 \rho(x):=(1+\sigma^2|x|^2)^{-\gamma},
 \qquad 0<\sigma\leq1,
\end{equation}
where $\gamma>dN+2\,(>dN/2)$ is fixed, so that
$\langle\rho\rangle<\infty$ and also
$\langle|\nabla\sqrt{\rho}|^2\rangle<\infty$.
This weight has property
\begin{equation}\label{rho_est}
 |\nabla\rho|\leq C\sigma\rho.
\end{equation}
Let us emphasize that the choice of $\sigma$ is up to us, so using the previous estimate we can replace all occurrences
of $\nabla\rho$ by $\rho$ with a small coefficient.

\begin{lemma}\label{b_est2}
For every $\varepsilon>0$,
\[
 \big|\langle b_\varepsilon\cdot\nabla f,\rho f\rangle\big|
 \leq
 c\sigma\left(
 \langle\rho|\nabla f|^2\rangle+\langle\rho f^2\rangle
 \right)
 +\frac{\nu}{d-2}\langle\rho|\nabla f|^2\rangle
\]
for a constant $c$ independent of $\varepsilon$ and $\sigma$.
\end{lemma}

\begin{proof}
We integrate by parts:
\[
 \langle b_\varepsilon\cdot\nabla f,\rho f\rangle
 =-\frac12\langle\operatorname{div}b_\varepsilon,\rho f^2\rangle
  -\frac12\langle b_\varepsilon\cdot\nabla\rho,f^2\rangle.
\]
The first term in the right-hand side is estimated as
in the proof of Lemma~\ref{b_est1}:
\[
\begin{aligned}
 \big|\langle b_\varepsilon\cdot\nabla f,\rho f\rangle\big|
 &\leq
 \frac{\nu}{d-2}
 \langle|\nabla(\sqrt{\rho}f)|^2\rangle
 +\frac12\langle|b_\varepsilon|,|\nabla\rho|f^2\rangle\\
 &\leq
 \frac{\nu}{d-2}
 \langle|\nabla(\sqrt{\rho}f)|^2\rangle
 +\frac{C\sigma}{2}\langle|b_\varepsilon|,\rho f^2\rangle\\
 &\leq
 \frac{\nu}{d-2}
 \langle|\nabla(\sqrt{\rho}f)|^2\rangle
 +
 \frac{C\sigma}{2}
 \left(
 \alpha\langle|b_\varepsilon|^2,\rho f^2\rangle
 +\frac1{4\alpha}\langle\rho f^2\rangle
 \right),
 \qquad\alpha>0.
\end{aligned}
\]
Lemma~\ref{b_est11} and, once again, \eqref{rho_est} now yield the sought inequality. 
\end{proof}

\begin{proposition}\label{prop_semigroup}
Let $d\geq3$, $N\geq2$. Assume that
\[
 0\leq\nu<2(d-2).
\]
Then, for every $2\leq r<\infty$ satisfying
\[
 r>\frac4{4-2\nu/(d-2)},
\]
there exists the limit
\[
 s\mbox{-}L^r\mbox{-}\lim_{\varepsilon\downarrow0}
 e^{-t\Lambda_\varepsilon}
 \qquad
 \text{(loc.\ uniformly in $t\geq0$)},
\]
which determines a strongly continuous semigroup on
$L^r(\mathbb R^{dN})$, say, $e^{-t\Lambda}$.
\end{proposition}

\begin{remark}
We assume $r\geq2$ since we only need large $r$ to reach
large strengths of attraction. In the case $1<r<2$ one
needs to be more careful with the integration by parts,
but this can be done, see e.g.\
\cite[Proposition 3.5]{MeS}.
\end{remark}

\begin{proof}[Proof of Proposition~\ref{prop_semigroup}]
Put
$u_\varepsilon(t)=e^{-t\Lambda_\varepsilon}u_0$,
$u_0\in C_c^\infty(\mathbb R^{dN})$, where, thus,
\begin{equation}\label{CP}
 \left\{
 \begin{array}{l}
 \partial_tu_\varepsilon-\Delta u_\varepsilon
 +b_\varepsilon\cdot\nabla u_\varepsilon=0,\\
 u_\varepsilon(0)=u_0,
 \end{array}
 \right.
\end{equation}
where drift $b_\varepsilon$ is bounded and smooth for
every $d\geq3$, see the beginning of this section
for its definition.

\medskip

Step 1 (\textit{Contraction estimate}). We have
\begin{equation}\label{contr}
 \|u_\varepsilon(t)\|_r\leq\|u_0\|_r,
 \qquad\text{for all }t>0.
\end{equation}
We multiply the equation in \eqref{CP} by
$u_\varepsilon|u_\varepsilon|^{r-2}$ and integrate by parts
using the identities $ \nabla|u_\varepsilon|^r =r|u_\varepsilon|^{r-1}\nabla|u_\varepsilon| $, $ \nabla|u_\varepsilon|^r
 =r|u_\varepsilon|^{r-2}u_\varepsilon\nabla u_\varepsilon$:
 $$
 0 =\langle\partial_tu_\varepsilon,   u_\varepsilon|u_\varepsilon|^{r-2}\rangle +
 \langle-\Delta u_\varepsilon,       u_\varepsilon|u_\varepsilon|^{r-2}\rangle
 + \langle b_\varepsilon\cdot\nabla u_\varepsilon,   u_\varepsilon|u_\varepsilon|^{r-2}\rangle,
$$
so
$$
 0=\frac1r\partial_t\langle|u_\varepsilon|^r\rangle +(r-1)  \langle|\nabla u_\varepsilon|^2  |u_\varepsilon|^{r-2}\rangle
 +
 \langle b_\varepsilon\cdot\nabla u_\varepsilon,  u_\varepsilon|u_\varepsilon|^{r-2}\rangle,
$$
and thus
$$
 \partial_t\langle|u_\varepsilon|^r\rangle +\frac{4(r-1)}r  \langle\nabla|u_\varepsilon|^{r/2},
    \nabla|u_\varepsilon|^{r/2}\rangle =  -2\langle b_\varepsilon\cdot\nabla|u_\varepsilon|^{r/2},    |u_\varepsilon|^{r/2}\rangle.
$$
Apply Lemma~\ref{b_est1}:
$$
 \partial_t\langle|u_\varepsilon|^r\rangle
 +\frac{4(r-1)}r
 \langle\nabla|u_\varepsilon|^{r/2},
        \nabla|u_\varepsilon|^{r/2}\rangle \leq
 \frac{2\nu}{d-2}
 \langle|\nabla|u_\varepsilon|^{r/2}|^2\rangle.
$$
We have
$
 \frac{4(r-1)}r-\frac{2\nu}{d-2}>0
$
since $r>4/(4-2\nu/(d-2))$. So,
$
 \partial_t\langle|u_\varepsilon|^r\rangle\leq0.
$
It remains to integrate in time to obtain the sought contraction estimate \eqref{contr}.

\medskip

Step 2 (\textit{Cauchy sequence in $L^\infty([0,T],L^r(\mathbb R^{dN}))$}). We fix any $\varepsilon_n\downarrow0$ and, armed with \eqref{contr}, show that solutions $u_n:=u_{\varepsilon_n}$ constitute a Cauchy sequence in
$L^\infty([0,T],L^r(\mathbb R^{dN}))$. Put $b_n:=b_{\varepsilon_n}$. In detail:

\medskip

(a) We have
\[
 \sup_n\int_0^T
 \langle|\nabla u_n|^2\rho\rangle\,ds<\infty,
\]
where weight $\rho$ is defined by \eqref{rho_def}.

Indeed, multiplying the parabolic equation in \eqref{CP}
by $u_n\rho$ and integrating by parts, we obtain
$$
 \frac12\langle u_n^2(t)\rho\rangle
 +\int_0^t\langle|\nabla u_n|^2\rho\rangle\,ds =\frac12\langle u_0^2\rho\rangle -  \int_0^t\langle\nabla u_n,u_n\nabla\rho\rangle\,ds -  \int_0^t\langle b_n\cdot\nabla u_n,u_n\rho\rangle\,ds.
$$
Using the Cauchy-Schwarz inequality, \eqref{rho_est},
and employing the contractivity estimate
$\|u_n\|_\infty\leq\|u_0\|_\infty$, we arrive at
\[
\begin{aligned}
 &\frac12\langle u_n^2(t)\rho\rangle
 +\int_0^t\langle|\nabla u_n|^2\rho\rangle\,ds\\
 &\leq
 \frac12\langle u_0^2\rho\rangle
 +\beta\int_0^t\langle|\nabla u_n|^2\rho\rangle\,ds
 +\frac{C\sigma^2}{4\beta}
      \int_0^t\langle u_n^2\rho\rangle\,ds\\
 &\quad+
 \|u_0\|_\infty
 \left(
 \gamma \int_0^t\langle|\nabla u_n|^2\rho\rangle\,ds
 +\frac1{4\gamma}
      \int_0^t\langle|b_n|^2\rho\rangle\,ds
 \right).
\end{aligned}
\]
It remains to choose $\beta$ and $\gamma$ small and to note
that
$
 \langle u_n^2\rho\rangle
 \leq\|u_0\|_\infty^2\langle\rho\rangle<\infty
$
and, crucially,
\[
 \sup_n\int_0^t\langle|b_n|^2\rho\rangle\,ds<\infty.
\]
The latter is an immediate consequence of Lemma~\ref{b_est11}, applied to $\sqrt{\rho}$, and the integrability properties of $\nabla\sqrt{\rho}$, cf.\ \eqref{rho_est}. This yields the assertion of (a).

\medskip

(b) Let us first show convergence with respect to weight
$\rho$:
\[
 \sup_{t\in[0,T]}
 \|(u_n(t)-u_m(t))\rho^{1/r}\|_r\longrightarrow0
 \qquad\text{as }n,m\to\infty.
\]
To this end, we take the difference between the equations
for $u_n$ and $u_m$, then multiply it by
$\rho h|h|^{r-2}$, where
\[
 h:=u_n-u_m,
\]
and integrate by parts:
\[
\begin{aligned}
 0
 &=\langle\partial_th,\rho h|h|^{r-2}\rangle +\langle-\Delta h,\rho h|h|^{r-2}\rangle
 +
 \langle(b_n-b_m)\cdot\nabla u_n,\rho h|h|^{r-2}\rangle\\
 &\quad+
 \langle b_m\cdot\nabla h,\rho h|h|^{r-2}\rangle,
\end{aligned}
\]
so
\[
\begin{aligned}
 0
 &=\partial_t\langle\rho|h|^r\rangle
   +\frac{4(r-1)}r
      \langle\rho|\nabla|h|^{r/2}|^2\rangle + 2\langle\nabla|h|^{r/2},(\nabla\rho)|h|^{r/2}\rangle\\
 &\quad+
 2\langle b_m\cdot\nabla|h|^{r/2},\rho|h|^{r/2}\rangle +  r\langle(b_n-b_m)\cdot\nabla u_n,\rho h|h|^{r-2}\rangle.
\end{aligned}
\]
Integrating in time and using $h(0)=0$, we obtain
\[
\begin{aligned}
 &\langle\rho|h(t)|^r\rangle
 +\frac{4(r-1)}r
   \int_0^t\langle\rho|\nabla|h|^{r/2}|^2\rangle\,ds\\
 &\leq
 2\int_0^t
 \left|
 \langle\nabla|h|^{r/2},(\nabla\rho)|h|^{r/2}\rangle
 \right|\,ds\\
 &\quad+
 2\int_0^t
 \left|
 \langle b_m\cdot\nabla|h|^{r/2},\rho|h|^{r/2}\rangle
 \right|\,ds\\
 &\quad+
 r\int_0^t
 \left|
 \langle(b_n-b_m)\cdot\nabla u_n,\rho h|h|^{r-2}\rangle
 \right|\,ds\\
 &\quad=:R_1+R_2+R_3.
\end{aligned}
\]
We use \eqref{rho_est} and Cauchy-Schwarz to estimate $R_1$:
\[
 R_1
 \leq
 C_1\sigma\int_0^t
 \langle\rho|\nabla|h|^{r/2}|^2\rangle\,ds
 +C_2\sigma\int_0^t\langle\rho|h|^r\rangle\,ds,
\]
where, recall, $\sigma$ can and will be chosen sufficiently small. Next, by Lemma~\ref{b_est2},
\[
\begin{aligned}
 R_2
 &=2\int_0^t
 \left|
 \langle b_m\cdot\nabla|h|^{r/2},\rho|h|^{r/2}\rangle
 \right|\,ds\\
 &\leq
 \left(\frac{2\nu}{d-2}+c\sigma\right)
 \int_0^t\langle\rho|\nabla|h|^{r/2}|^2\rangle\,ds
 +c\sigma\int_0^t\langle\rho|h|^r\rangle\,ds.
\end{aligned}
\]

Finally, we have
\begin{align*}
R_3 := &  r   \int_0^t \left\vert  \langle (b_{n} - b_{m}) \cdot \nabla u_n,  \rho h |h|^{r-2}\rangle \right\vert ds \\
&  \le r \int_0^t \int_{\mathbb{R}^{dN}}\vert b_{n} - b_{m} \vert \vert \nabla u_n \vert |h|^{r-1} \rho dx ds \\
& \le r \left( \int_0^t \int_{\mathbb{R}^{dN}}\vert b_{n} - b_{m} \vert^2 |h|^{2(r-1)} \rho dx ds\right)^{\frac{1}{2}}\left( \int_0^t \int_{\mathbb{R}^{dN}} \vert \nabla u_n \vert^2  \rho dx ds\right)^{\frac{1}{2}} \\ 
& \; (\text{we use $\Vert h \Vert_{\infty} \le 2\Vert u_0 \Vert_{\infty}$}) \\
& \le 2^{r-1} r \Vert u_0 \Vert_{\infty}^{r-1}  \left( \int_0^t \int_{\mathbb{R}^{dN}}\vert b_{n} - b_{m} \vert^2  \rho
 dx ds\right)^{\frac{1}{2}}\left( \int_0^t \int_{\mathbb{R}^{dN}} \vert \nabla u_n \vert^2  \rho dx ds\right)^{\frac{1}{2}}  \rightarrow 0\; \text{as}\; n, m \rightarrow \infty,
\end{align*}
where we estimate the last factor using (a), and apply $\|\sqrt{\rho}(b-b_\varepsilon)\|_2 \rightarrow 0$.
This gives us (b).

\medskip

(c) Next, we show that the ``tails'' of $u_m$ are small, that is, for arbitrarily small $\epsilon>0$, parameter
$\sigma$ in the definition of weight $\rho$ can be selected small enough, depending on $T$ and $u_0$, to give us
the inequality
\begin{equation}\label{tail_est}
 \sup_m\sup_{s\in[0,T]}
 \langle|u_m(s)|^r(1-\rho)\rangle<\epsilon.
\end{equation}
Indeed, let us put for brevity
\[
 \Upsilon:=
 \int_0^t
 \langle(1-\rho)|\nabla|u_m|^{r/2}|^2\rangle\,ds.
\]
Multiplying the parabolic equation \eqref{CP} by $r(1-\rho)u_m|u_m|^{r-2}$ and integrating, we obtain
\[
\begin{aligned}
 &\langle(1-\rho)|u_m(t)|^r\rangle
 -\langle(1-\rho)|u_0|^r\rangle
 +\frac{4(r-1)}r\Upsilon\\
 &=2\int_0^t
 \langle\nabla|u_m|^{r/2},
        (\nabla\rho)|u_m|^{r/2}\rangle\,ds -
 r\int_0^t
 \langle b_m\cdot\nabla u_m,
        (1-\rho)u_m|u_m|^{r-2}\rangle\,ds.
\end{aligned}
\]
Using \eqref{b_est_sharp}, we obtain
\[
\begin{aligned}
 -r\int_0^t
 \langle b_m\cdot\nabla u_m,
        (1-\rho)u_m|u_m|^{r-2}\rangle\,ds &=-2\int_0^t
 \langle b_m\cdot\nabla|u_m|^{r/2},
        (1-\rho)|u_m|^{r/2}\rangle\,ds\\
 &\leq
 \alpha\Upsilon
 +\frac1\alpha\int_0^t
 \langle|b_m|^2(1-\rho)|u_m|^r\rangle\,ds\\
 &\leq
 \alpha\Upsilon
 +\frac{\nu^2}{\alpha(d-2)^2}
 \int_0^t
 \left\langle
 \left|\nabla\left(|u_m|^{r/2}\sqrt{1-\rho}\right)\right|^2
 \right\rangle\,ds,
 \qquad\alpha>0.
\end{aligned}
\]
We have
\[
\begin{aligned}
 &\int_0^t
 \left\langle
 \left|\nabla\left(|u_m|^{r/2}\sqrt{1-\rho}\right)\right|^2
 \right\rangle\,ds\\
 &=\Upsilon
 +\int_0^t
 \langle|u_m|^r|\nabla\sqrt{1-\rho}|^2\rangle\,ds
 +\frac12\int_0^t\langle|u_m|^r,\Delta\rho\rangle\,ds\\
 &\leq\Upsilon
 +\int_0^t
 \left\langle|u_m|^r,
       \frac{|\nabla\rho|^2}{4(1-\rho)}\right\rangle\,ds
 +\frac12\int_0^t
       \langle|u_m|^r,|\Delta\rho|\rangle\,ds.
\end{aligned}
\]
Also,
\[
 2\int_0^t
 \langle\nabla|u_m|^{r/2},
        (\nabla\rho)|u_m|^{r/2}\rangle\,ds
 =-\int_0^t\langle|u_m|^r,\Delta\rho\rangle\,ds.
\]
Thus, we obtain
\[
\begin{aligned}
 &\langle(1-\rho)|u_m(t)|^r\rangle
 +\left(
 \frac{4(r-1)}r-\alpha-\frac{\nu^2}{\alpha(d-2)^2}
 \right)\Upsilon\\
 &\leq
 \langle(1-\rho)|u_0|^r\rangle
 +
 \frac{\nu^2}{\alpha(d-2)^2}
 \int_0^t
 \left\langle|u_m|^r,
       \frac{|\nabla\rho|^2}{4(1-\rho)}\right\rangle\,ds\\
 &\quad+
 \left(1+\frac{\nu^2}{2\alpha(d-2)^2}\right)
 \int_0^t\langle|u_m|^r,|\Delta\rho|\rangle\,ds.
\end{aligned}
\]
Now, fix $\alpha>0$ to have
$
 \frac{4(r-1)}r-\alpha-\frac{\nu^2}{\alpha(d-2)^2}>0.
$
Such a choice is possible since $4(r-1)/r>2\nu/(d-2)$. Then, we have
$$
 \left\{
    \begin{array}{ll}
      \frac{ \vert \nabla \rho\vert^2}{4(1-\rho)} \leq    \gamma \sigma^2 \rho, \\ \ \\
   \vert \Delta \rho  \vert  \leq 2 \gamma\sigma^2(2(\gamma+1)- dN)\rho,   \\ \ \\
      \int_0^t \langle \rho |u_{m}|^{r} \rangle  ds \le   t \|u_0\|_r^r.  
    \end{array}
\right.
$$
Hence, we get
\[
\begin{aligned}
 \sup_{s\in[0,t]}
 \langle(1-\rho)|u_m(s)|^r\rangle \leq
 \langle(1-\rho)|u_0|^r\rangle
 +
 t\sigma^2
 \left[
 \frac{\gamma\nu^2}{\alpha(d-2)^2}
 +\left(1+\frac{\nu^2}{2\alpha(d-2)^2}\right)
   2\gamma\bigl(2(\gamma+1)-dN\bigr)
 \right]\|u_0\|_r^r.
\end{aligned}
\]
Since $1-\rho\to0$ uniformly on the support of $u_0$
as $\sigma\to0$, the right-hand side of the inequality
can be made arbitrarily small by taking sufficiently
small $\sigma$. That is, we have \eqref{tail_est}.

\medskip

Combining (a)-(c), we obtain that there exists
the limit
$$
 T_t f:=s\mbox{-}L^r\mbox{-}\lim_{\varepsilon \downarrow 0}e^{-t \Lambda_\varepsilon}f \quad (\text{loc.\,uniformly in $t \geq 0$}).
 $$ 
The uniform convergence in  $t$ ensures the strong continuity of $T_t$ in $L^r$. We extend $T_t$ to all $t>0$ by postulating the semigroup property. Step 1 yields
$$
\|T_t u_0\|_r \leq \|u_0\|_r, \quad t>0, \quad u_0 \in C_c^\infty.
$$
Therefore, we can extend $T_t$ from $C_c^\infty$ to all $L^r$. The latter determines a strongly continuous semigroup, $e^{-t\Lambda}:=T_t$. 
\end{proof}

\subsection{A posteriori upper bound}

To end the proof, we need to estimate the integral kernel ($=:$\,heat kernel $p(t,x,y)$) of $e^{-t\Lambda}$. 
The fact that we have a semigroup of integral operators follows from the Dunford-Pettis theorem, via the $L^r \rightarrow L^q$ embedding property for $e^{-t\Lambda}$. The latter is proved e.g.\,by repeating the proof of \cite[Theorem 4.2]{KiS-theory} supplemented with Lemma \ref{b_est1}, i.e.\,a realization of Nash's classical argument. Alternatively, we can simply use the a priori upper bound established in (\textit{i}).
This a priori upper bound yields a posteriori upper bound on the integral kernel of  $e^{-t\Lambda}$ via the Lebesgue differentiation theorem.  \hfill \qed

\bigskip

\appendix

\section{Proof of Theorem A: abstract desingularization}
\label{app_A}

In the proof of Theorem A we use the following weighted variant of the Coulhon-Raynaud extrapolation theorem \cite[Prop.\,II.2.1, Prop.\,II.2.2]{VSC}.

\begin{theorem}[\cite{KSS}]
\label{thmE2}
Let $U^{t,\theta}$ be a two-parameter family of operators
\[U^{t,\theta}f = U^{t,\tau}U^{\tau,\theta}f,  \qquad f \in L^1 \cap L^\infty, \quad 0 \leq \theta < \tau < t \leq \infty.
\]
Suppose that for some $1 \leq p < q < r \leq \infty$, $\nu>0$
\begin{align*}
\| U^{t,\theta} f \|_p & \leq M_1 \| f \|_{p,\sqrt{\psi}}, \quad 0 \leq \psi \in L^1+L^\infty,  \quad \|f\|_{p,\sqrt{\psi}}:=\langle |f|^p \psi \rangle^{1/p},\\
 \| U^{t,\theta} f \|_r & \leq M_2 (t-\theta)^{-\nu} \|  f \|_q
\end{align*}
for all $(t,\theta)$ and $f \in L^1 \cap L^\infty.$ Then
\[
\| U^{t,\theta} f \|_r \leq M (t-\theta)^{-\nu/(1-\beta)} \| f \|_{p,\sqrt{\psi}} ,
\]
where $\beta = \frac{r}{q}\frac{q-p}{r-p}$ and $M = 2^{\nu/(1-\beta)^2} M_1 M_2^{1/(1-\beta)}.$
\end{theorem}

\begin{proof}[Proof of Theorem \ref{thmE2}]

We have ($t_\theta:=\frac{t+\theta}{2}$)
\begin{align*}
\| U^{t, \theta} f \|_r & \leq M_2 (t-t_\theta)^{-\nu} \| U^{t_\theta,\theta} f \|_q \\
& \leq M_2 (t-t_\theta)^{-\nu} \| U^{t_\theta,\theta} f \|_r^\beta \;\| U^{t_\theta,\theta} f \|_p^{1-\beta} \\
& \leq M_2 M_1^{1-\beta} (t-t_\theta)^{-\nu} \| U^{t_\theta,\theta} f \|_r^\beta \;\| f \|_{p,\sqrt{\psi}}^{1-\beta},
\end{align*}
and hence, since
$t-t_\theta=t_\theta-\theta=\frac{t-\theta}{2}$,
\begin{align*}
(t-\theta)^{\nu/(1-\beta)} \Vert  U^{t,\theta} f \Vert_r/ \Vert  f \Vert_{p,\sqrt{\psi}} &  \leq M_2 M_1^{1-\beta} (t-\theta)^{\nu/(1-\beta)}(t-t_\theta)^{-\nu}  \Vert U^{t_\theta,\theta}f \Vert_r^{\beta}/\Vert f \Vert_{p,\sqrt{\psi}}^{\beta}  \\
&= M_2 M_1^{1-\beta} 2^\nu (t-\theta)^{\nu\beta/(1-\beta)}
\Vert U^{t_\theta,\theta}f \Vert_r^{\beta}/\Vert f \Vert_{p,\sqrt{\psi}}^{\beta}  \\
&= M_2 M_1^{1-\beta} 2^{\nu+\nu\beta/(1-\beta)} \big [(t_\theta - \theta)^{\nu/(1-\beta)} \| U^{t_\theta,\theta} f \|_r\;/\| f \|_{p,\sqrt{\psi}} \big ]^\beta \\
&= M_2 M_1^{1-\beta} 2^{\nu/(1-\beta)} \big [(t_\theta - \theta)^{\nu/(1-\beta)} \| U^{t_\theta,\theta} f \|_r\;/\| f \|_{p,\sqrt{\psi}} \big ]^\beta.
\end{align*}
Setting $R_{2 T}: = \sup_{t-\theta \in ]0,2T]} \big [ (t-\theta)^{\nu/(1-\beta)} \| U^{t,\theta} f \|_r/\| f \|_{p,\sqrt{\psi}} \big ],$ we obtain from the last inequality that $R_{2 T} \leq M^{1-\beta} (R_T)^\beta.$ But $R_T \leq R_{2T}$, and so $R_{2T} \leq M.$
The proof of Theorem \ref{thmE2} is completed.
\end{proof}

\begin{proof}[Proof of Theorem A]
By \eqref{S4} and \eqref{S3},
\begin{align*}
\|e^{-t\Lambda}h\|_{1} &\leq c_0^{-1} \|\varphi e^{-t\Lambda} \varphi^{-1} \varphi h \|_1  \\
& \leq c_0^{-1}c_1 \|h\|_{1,\sqrt{\varphi}}, \qquad  h \in L^\infty_{\rm com}.
\end{align*}
The latter, \eqref{S1} and Theorem \ref{thmE2} with $\psi:=\varphi$ yield
$$
\|e^{-t\Lambda}f\|_{\infty} \leq Mt^{-a}\|\varphi f\|_1, \quad t \in ]0,T], \quad f \in  L^\infty_{\rm com}.
$$
Note that \eqref{S1} verifies the assumptions of the Dunford-Pettis Theorem, so, for every $t>0$, $e^{-t\Lambda}$ is an integral operator.
The previous estimate thus yields \eqref{nie} and ends the proof of Theorem A.
\end{proof}

\bigskip

\section{Trotter's approximation theorem}
\label{app_B}

Consider a sequence $\{e^{-tA_k}\}_{k=1}^\infty$ of strongly continuous semigroups on a Banach space $Y$.

\begin{theorem}[{H.F.\,Trotter \cite[Ch.\,IX, sect.\,2]{Ka}}]
Let
$\sup_k\|(\mu+A_k)^{-m}\|_{Y \rightarrow Y} \leq M(\mu-\omega)^{-m}$, $m=1,2,\dots$, $\mu>\omega$, and 
$s\mbox{-}\lim_{\mu \rightarrow \infty}\mu(\mu+ A_k)^{-1}=1$ uniformly in $k$, and let $s\mbox{-}\lim_{k}(\zeta+ A_k)^{-1}$
exist for some $\zeta$ with ${\rm Re\,} \zeta>\omega$. Then there exists a strongly continuous semigroup $e^{-tA}$ such that
$$
(z+ A_k)^{-1} \overset{s}{\rightarrow} (z+ A)^{-1} \quad \text{ for every } {\rm Re}\, z>\omega,
$$
and
$$
e^{-tA_k} \overset{s}{\rightarrow} e^{-tA}
$$
uniformly in any finite interval of $t \geq 0$.
\end{theorem}

The first condition of the theorem is satisfied if e.g.\,$\sup_k\|(z+A_k)^{-1}\|_{Y \rightarrow Y} \leq C|z-\omega|^{-1}$, ${\rm Re\,}z>\omega$.

\bigskip

\section{Some auxiliary calculations in the proofs of Theorems \ref{thm1} and \ref{nl:main}}
\label{app:elementary}

\subsection{Proof of Claim \ref{cl3}}

1.~First, let us note that, setting $K(z)=z/|z|^2$ on $\mathbb R^2$, we have
\[
 (e^{t\Delta_{\mathbb R^2}}K)(z)
 =\frac{z}{|z|^2}\left(1-e^{-|z|^2/(4t)}\right),\qquad t>0.
\]
A simple calculation in polar coordinates yields ${\rm div}\,K=2\pi \delta_0$. Hence, setting 
$u(z)= (e^{t\Delta_{\mathbb R^2}}K)(z)$ and recalling notation for the Gaussian density $\Gamma_t(x-y):=e^{t\Delta}(x-y)$, we have
$$
 {\rm div\,}u(z)=\int_{\mathbb R^2}
       K(y)\cdot\nabla_z\Gamma_t(z-y)\,dy=-\int_{\mathbb R^2}
       K(y)\cdot\nabla_y\Gamma_t(z-y)\,dy=2\pi\Gamma_t(z).
$$
The vector field $u$ is a radial one and has form $u(z)=a(r)\frac{z}{r}$. Indeed, using
$\Gamma_t(Rz)=\Gamma_t(z)$,
$\qquad K(Rz)=RK(z)$
for every orthogonal matrix $R$, and making the change
of variables $y=Rw$, yields
\[
\begin{aligned}
 u(Rz)
 &=\int_{\mathbb R^2}\Gamma_t(Rz-y)K(y)\,dy\\
 &=\int_{\mathbb R^2}\Gamma_t(R(z-w))K(Rw)\,dw\\
 &=R\int_{\mathbb R^2}\Gamma_t(z-w)K(w)\,dw
 =Ru(z),
\end{aligned}
\]
from which one deduces (first along each coordinate axis, and then applying rotation to reduce a general direction to a coordinate axis) $u(z)=a(r)\frac{z}{r}$.

So, one trivially has $ 2\pi r\,a(r)
=\int_{\partial B_r}
       u(y)\cdot\frac{y}{r}\,d\sigma_y$.
	   Since $\frac{y}{r}$ is also the outer normal, we have by the divergence theorem
\[
\begin{aligned}
 2\pi r\,a(r)=\int_{B_r}\operatorname{div}u(y)\,dy=2\pi\int_{B_r}\Gamma_t(y)\,dy=2\pi\int_0^r
       \frac{\rho}{2t}e^{-\rho^2/(4t)}\,d\rho=2\pi\left(1-e^{-r^2/(4t)}\right).
\end{aligned}
\]
Dividing by $2\pi r$ and substituting
into the expression for $u$ yields
\[
 (e^{t\Delta_{\mathbb R^2}}K)(z)
 =\frac{z}{|z|^2}
       \left(1-e^{-|z|^2/(4t)}\right),
 \qquad z\ne0.
\]
At $z=0$, the convolution vanishes by oddness.

2.~For a function of $z=x_i-x_j$, the configuration-space
Laplacian acts as $2\Delta_z$. Therefore,
\[
 \begin{split}
 &\int_0^\infty e^{-s}
       e^{\frac{\varepsilon}{8s}\Delta}K(x_i-x_j)\,ds\\
 &\quad=\frac{x_i-x_j}{|x_i-x_j|^2}
       \int_0^\infty e^{-s}
       \left(1-e^{-s|x_i-x_j|^2/\varepsilon}\right)\,ds
       =\frac{x_i-x_j}{|x_i-x_j|^2+\varepsilon}.
 \end{split}
\]
Now, summation over $i<j$ yields the sought identity for
$b_\varepsilon$. \hfill \qed

\subsection{Proof of Claim \ref{nl:hardy-lemma}}
We apply
the standard (one-particle) fractional Hardy inequality
\begin{equation*}
 \kappa_{d,\alpha}\int_{\mathbb R^d}
             \frac{|f(z)|^2}{|z|^\alpha}\,dz
 \le\|(-\Delta)^{\alpha/4}f\|_2^2,
 \qquad f\in W^{\alpha/2,2}(\mathbb R^d).
\end{equation*}
in each particle
variable $x^i$, $i=1,\dots,N$. Specifically, let $u\in C_c^\infty(\mathbb R^{dN})$ and
fix an ordered pair $(i,j)$ with $i\ne j$. The one-particle inequality
applies to the function
$
 z\longmapsto
 u(x_1,\ldots,x_{i-1},z+x_j,x_{i+1},\ldots,x_N).
$
The translation invariance of the fractional Laplacian allows the singularity of the potential
to be placed at $x_j$. Now, integration in the remaining $N-1$
coordinates gives us
\begin{equation*}
 \int_{\mathbb R^{dN}}
       \frac{|u(x)|^2}{|x_i-x_j|^\alpha}\,dx
 \le\kappa_{d,\alpha}^{-1}
       \|(-\Delta_{x_i})^{\alpha/4}u\|_2^2,
\end{equation*}
where the right-hand side no longer
depends on $j$. We may therefore sum over 
$j$~($\ne i$) and, further, sum over $i=1,\ldots,N$. We arrive at
\[
 \sum_i\sum_{j\ne i}
 \int_{\mathbb R^{dN}}\frac{|u(x)|^2}{|x_i-x_j|^\alpha}\,dx
 \le\frac{N-1}{\kappa_{d,\alpha}}
       \sum_i\|(-\Delta_{x_i})^{\alpha/4}u\|_2^2,
\]
which is \eqref{nl:many-Hardy} on $C_c^\infty$. The standard density argument allows now to pass from $u \in C_c^\infty$ to $u \in W^{\alpha/2,2}$. \hfill \qed

\subsection{Proof of Lemma \ref{nl:power-lemma}}
We first carry out the proof for $\varepsilon=1$. The general case will then be obtained using scaling. Put for brevity
\[
 a=\frac{d+\alpha}{4},\qquad b=\frac{d-\alpha}{4}.
\]
Since $d>\alpha$, both parameters are positive, and
$a-b=\alpha/2$, $a+b=d/2$. 

\begin{claim}
\label{cl_2b}
\[
 (-\Delta)^{\alpha/2}(1+|z|^2)^{-2b}
 =2^\alpha\frac{\Gamma(2a)}{\Gamma(2b)}
  (1+|z|^2)^{-2a}.
\]
\end{claim}

Note that the power is off by the factor of $2$, so this claim is only the point of departure for the proof.

\begin{proof}[Proof of Claim \ref{cl_2b}]
Let $f_p(z)=(1+|z|^2)^{-p}$, $p>0$. We have
\[
 f_p(z)=\frac1{\Gamma(p)}\int_0^\infty
 t^{p-1}e^{-t}e^{-t|z|^2}\,dt.
\]
Taking the Fourier transform of the previous identity, and using under the time integral the identity
$
 \widehat{e^{-t|\cdot|^2}}(\xi)
 =\left(\frac{\pi}{t}\right)^{d/2}e^{-|\xi|^2/(4t)},
$
we obtain
\[
 \widehat f_p(\xi)=\frac{\pi^{d/2}}{\Gamma(p)}
 \int_0^\infty
 t^{p-1-d/2}e^{-t-|\xi|^2/(4t)}\,dt.
\]
Now, we take $p=2b$ and evaluate, taking into account $2b-d/2=-\alpha/2$:
\[
 \begin{aligned}
 |\xi|^\alpha\widehat f_{2b}(\xi)
 &=\frac{\pi^{d/2}|\xi|^\alpha}{\Gamma(2b)}
   \int_0^\infty t^{-1-\alpha/2}
          e^{-t-|\xi|^2/(4t)}\,dt\\
 & (\text{we make a change of variable:
$t=|\xi|^2/(4u)$}) \\
 &=\frac{2^\alpha\pi^{d/2}}{\Gamma(2b)}
   \int_0^\infty u^{\alpha/2-1}
          e^{-u-|\xi|^2/(4u)}\,du\\
 &=2^\alpha\frac{\Gamma(2a)}{\Gamma(2b)}
   \widehat f_{2a}(\xi),
 \end{aligned}
\]
where the last equality uses $2a-d/2=\alpha/2$. Inverting the Fourier transform ends the proof of the claim.
\end{proof}

In what follows, we will use some algebraic identities for Euler's beta function:
\begin{equation}
\label{beta_int}
 B(a,b):=\int_0^1 r^{a-1}(1-r)^{b-1}\,dr
       =\frac{\Gamma(a)\Gamma(b)}{\Gamma(a+b)}.
\end{equation}
Clearly, $B(a,b)=B(b,a)$.

\begin{claim}
\label{cl_prelim_id}
\[
 \frac{(-\Delta)^{\alpha/2}(1+|z|^2)^{-b}}
      {\kappa_{d,\alpha}(1+|z|^2)^{-a}}
 =\frac1{B(a,a)}\int_0^1
   v^{a-1}(1-v)^{b-1}(1-qv)^{\alpha/2}\,dv.
\]
\end{claim}

With this identity at hand, the proof of Lemma \ref{nl:power-lemma} for $\varepsilon=1$ will be reduced to estimating appropriately the integral on the right.

Set $q(z)=|z|^2/(1+|z|^2)$.

\begin{proof}[Proof of Claim \ref{cl_prelim_id}]
First, note that
\begin{equation}
\label{id1}
 (1+|z|^2)^{-b}
 =\frac{\Gamma(2b)}{\Gamma(b)^2}
   \int_0^\infty t^{b-1}(1+t+|z|^2)^{-2b}\,dt.
\end{equation}
Indeed, a change of variable $t=(1+|z|^2)v$ produces the factor
$(1+|z|^2)^{-b}$ multiplied by
\[
 \int_0^\infty v^{b-1}(1+v)^{-2b}\,dv
 =B(b,b)=\frac{\Gamma(b)^2}{\Gamma(2b)},
\]
where the beta integral is obtained upon substituting in \eqref{beta_int} $r=v/(1+v)$.

Next, applying scaling 
\[
 (1+t+|z|^2)^{-2b}
 =(1+t)^{-2b}f_{2b}\bigl(z/\sqrt{1+t}\bigr) \quad (\text{where, recall, }f_{2b}(z)=(1+|z|^2)^{-2b}),
\]
in
the identity in Claim \ref{cl_2b} yields
\begin{equation}
\label{id2}
 (-\Delta)^{\alpha/2}(1+t+|z|^2)^{-2b}
 =2^\alpha\frac{\Gamma(2a)}{\Gamma(2b)}
  (1+t)^{\alpha/2}(1+t+|z|^2)^{-2a}.
\end{equation}
(In detail, applying $(-\Delta)^{\alpha/2}$ to $f_{2b}$ contributes the factor
$(1+t)^{-\alpha/2}$, whereas writing the resulting $f_{2a}$
in terms of $1+t+|z|^2$ contributes $(1+t)^{2a}$.
Their combined exponent is
$-2b-\alpha/2+2a=\alpha/2$.)

We now apply the fractional Laplacian to \eqref{id1} and use \eqref{id2} under the time integral:
\[
 (-\Delta)^{\alpha/2}(1+|z|^2)^{-b}
 =\frac{2^\alpha\Gamma(2a)}{\Gamma(b)^2}
   \int_0^\infty
   \frac{t^{b-1}(1+t)^{\alpha/2}}
        {(1+t+|z|^2)^{2a}}\,dt.
\]
Finally, put $r=(1+t)^{-1}$. Then
\[
 t=\frac{1-r}{r},\qquad
 1+t+|z|^2=\frac{1+|z|^2r}{r},\qquad
 dt=-\frac{dr}{r^2}.
\]
The resulting power of $r$ is
$2a-b-\alpha/2-1=a-1$, and hence
\begin{equation}\label{nl:power-integral}
 (-\Delta)^{\alpha/2}(1+|z|^2)^{-b}
 =\frac{2^\alpha\Gamma(2a)}{\Gamma(b)^2}
   \int_0^1
   \frac{r^{a-1}(1-r)^{b-1}}{(1+|z|^2r)^{2a}}\,dr.
\end{equation}
The integral is finite since $a,b>0$. Our goal now is to extract
the factor $(1+|z|^2)^{-a}$ and then estimate the remaining
integral from below.
So, let us fix $z$ and make another change of variables
\[
 r=\frac{v}{1+|z|^2(1-v)},\qquad 0\le v\le1.
\]
This is an increasing change of variables from $[0,1]$ onto itself.
The factors in the integrand and the differential become
\[
 \begin{aligned}
 1-r&=\frac{(1+|z|^2)(1-v)}{1+|z|^2(1-v)},\\
 1+|z|^2r&=\frac{1+|z|^2}{1+|z|^2(1-v)},\\
 dr&=\frac{1+|z|^2}{[1+|z|^2(1-v)]^2}\,dv.
 \end{aligned}
\]
Then
\[
 \begin{aligned}
 \frac{r^{a-1}(1-r)^{b-1}}{(1+|z|^2r)^{2a}}\,dr=(1+|z|^2)^{b-2a}
       v^{a-1}(1-v)^{b-1}
       [1+|z|^2(1-v)]^{a-b}\,dv.
 \end{aligned}
\]
Writing $q=|z|^2/(1+|z|^2)$, we have
$1+|z|^2(1-v)=(1+|z|^2)(1-qv)$. Since $a-b=\alpha/2$,
the last expression is therefore
\[
 (1+|z|^2)^{-a}
 v^{a-1}(1-v)^{b-1}(1-qv)^{\alpha/2}\,dv.
\]
Finally, the definition of $\kappa_{d,\alpha}$ implies
\[
 \frac{2^\alpha\Gamma(2a)}
      {\kappa_{d,\alpha}\Gamma(b)^2}
 =\frac{\Gamma(2a)}{\Gamma(a)^2}
 =\frac1{B(a,a)}.
\]
Thus \eqref{nl:power-integral} becomes the assertion of the claim.
\end{proof}

\begin{claim}
\[
 \begin{aligned}
 \frac{(-\Delta)^{\alpha/2}(1+|z|^2)^{-b}}
      {\kappa_{d,\alpha}(1+|z|^2)^{-a}} \geq 1+\frac{\alpha}{(d-\alpha)(1+|z|^2)}.
 \end{aligned}
\]
\end{claim}

\begin{proof}We estimate the right-hand side of the inequality in the previous claim.
The
function $t\mapsto t^{\alpha/2}$ is concave on $[0,\infty[$, so applying concavity to
$1-qv=(1-q)\cdot1+q(1-v)$ gives
\[
 (1-qv)^{\alpha/2}\ge(1-q)+q(1-v)^{\alpha/2}.
\]
Multiplication by $v^{a-1}(1-v)^{b-1}$ and integration over
$(0,1)$ yield
\[
 \begin{aligned}
 &\int_0^1 v^{a-1}(1-v)^{b-1}(1-qv)^{\alpha/2}\,dv\\
 &\quad\ge(1-q)B(a,b)+qB(a,b+\alpha/2)=(1-q)B(a,b)+qB(a,a).
 \end{aligned}
\]
We now compare $B(a,b)$ with $B(a,a)$. Since $0<b<a$,
we have $(1-v)^b\ge(1-v)^a$ for $0<v<1$. Multiplying by
$v^{a-1}$ and integrating gives
\[
 B(a,b+1)=\int_0^1v^{a-1}(1-v)^b\,dv
 \ge\int_0^1v^{a-1}(1-v)^a\,dv=B(a,a+1).
\]
The identity $\Gamma(t+1)=t\Gamma(t)$ implies
\[
 B(a,b+1)=\frac b{a+b}B(a,b),\qquad
 B(a,a+1)=\frac12B(a,a).
\]
Therefore,
\[
 \frac{B(a,b)}{B(a,a)}
 \ge\frac{a+b}{2b}=\frac{d}{d-\alpha}.
\]
The preceding estimates and
$1-q=(1+|z|^2)^{-1}$ yield
\[
 \begin{aligned}
 \frac{(-\Delta)^{\alpha/2}(1+|z|^2)^{-b}}
      {\kappa_{d,\alpha}(1+|z|^2)^{-a}}
 &\ge(1-q)\frac{B(a,b)}{B(a,a)}+q\\
 &\ge(1-q)\frac{d}{d-\alpha}+q\\
 &=1+\frac{\alpha}{(d-\alpha)(1+|z|^2)},
 \end{aligned}
\]
which ends the proof of the claim.
\end{proof}

To pass to general $\varepsilon>0$, write
$h_\varepsilon(z)=|z|_\varepsilon^{-(d-\alpha)/2}$ and observe that
\[
 h_\varepsilon(z)=\varepsilon^{-b}h_1(z/\sqrt\varepsilon).
\]
The fractional Laplacian has order $\alpha$, so its scaling
property gives
\[
 (-\Delta)^{\alpha/2}h_\varepsilon(z)
 =\varepsilon^{-b-\alpha/2}
   \bigl[(-\Delta)^{\alpha/2}h_1\bigr](z/\sqrt\varepsilon)
 =\varepsilon^{-a}
   \bigl[(-\Delta)^{\alpha/2}h_1\bigr](z/\sqrt\varepsilon).
\]
Applying the estimate just proved at $z/\sqrt\varepsilon$,
and noting that
\[
 \varepsilon^{-a}(1+|z|^2/\varepsilon)^{-a}
       =(|z|^2+\varepsilon)^{-a},
 \qquad
 \frac1{1+|z|^2/\varepsilon}
       =\frac{\varepsilon}{|z|^2+\varepsilon},
\]
we obtain \eqref{nl:power-bound}. \hfill \qed

\bigskip

\section{Bounded skew-symmetric perturbations}
\label{C_app}

We mentioned in the comments after Theorem~\ref{thm2} that the proof of its  weighted Gaussian upper bound  extends to operator \eqref{ks7}. Here are the details. Let $C$ be a bounded skew-symmetric matrix field.
As before, fix $s,\varepsilon>0$ and put
\[
 \varphi=\frac{\psi_\varepsilon}{\psi_{\varepsilon+s}}.
\]
Recall the notation $\phi_\theta(x)=e^{\theta\cdot x}$.
Put
\[
 \Lambda_\theta^c
 :=\phi_\theta\Lambda^c\phi_{-\theta},
\]
where
$$
 \Lambda^c  :=A^c+b_{\varepsilon+s}\cdot(I+C)\nabla  =-\nabla\cdot(I+C)\nabla +b_\varepsilon\cdot(I+C)\nabla,
$$
and 
$$
\begin{aligned}
 A^c
 &:=-\varphi^{-1}\nabla\cdot   \bigl(\varphi(I+C)\nabla\bigr) =-\nabla\cdot(I+C)\nabla  -\frac{\nabla\varphi}{\varphi}\cdot(I+C)\nabla.
\end{aligned}
$$
Set
$$
 u(t)=e^{-t\Lambda_\theta^c}f,  \qquad f\in C_c^\infty(\mathbb R^{dN}).
$$
We assume without loss of generality that $f\geq0$, so $u(t)\geq0$ for all $t>0$. The counterpart of \eqref{A_first_eq} in the proof of Theorem~\ref{thm2} is
$$
 \frac1p\partial_t  \langle u^{p/2},u^{p/2}\rangle_\varphi +\langle\Lambda_\theta^c u,u^{p-1}\rangle_\varphi=0.
$$
For $p\geq2$, we evaluate
$$
\begin{aligned}
 \langle\Lambda_\theta^c u,u^{p-1}\rangle_\varphi
 &=\langle A^c\phi_{-\theta}u, \phi_\theta u^{p-1}\rangle_\varphi  + \langle b_{\varepsilon+s}\cdot(I+C) (\nabla u-\theta u),u^{p-1}\rangle_\varphi\\
 &=\langle\Lambda_\theta u,u^{p-1}\rangle_\varphi
   +L-J + \langle b_{\varepsilon+s}\cdot C (\nabla u-\theta u),u^{p-1}\rangle_\varphi,
\end{aligned}
$$
where
\[
\begin{aligned}
 L&:=
 \left\langle
 -\nabla\cdot C\nabla(\phi_{-\theta}u),
 \phi_\theta u^{p-1}\varphi
 \right\rangle,\\
 J&:=
 \left\langle
 \frac{\nabla\varphi}{\varphi}\cdot
       C\nabla(\phi_{-\theta}u),
 \phi_\theta u^{p-1}\varphi
 \right\rangle.
\end{aligned}
\]
The term $\langle\Lambda_\theta u,u^{p-1}\rangle_\varphi$ was already evaluated in \eqref{form_exp}. Next, integrating by parts, we obtain
$$
\begin{aligned}
 L
 &=\left\langle  C\nabla(\phi_{-\theta}u),  \nabla(\phi_\theta u^{p-1}\varphi)
 \right\rangle =\left\langle C\nabla(\phi_{-\theta}u), \nabla(\phi_\theta u^{p-1}) \right\rangle_\varphi
 + \left\langle  C\nabla(\phi_{-\theta}u),  \phi_\theta u^{p-1}\nabla\varphi  \right\rangle,
\end{aligned}
$$
where the last term coincides with $J$. Therefore,
\[
 L-J=
 \left\langle
 C\nabla(\phi_{-\theta}u),
 \nabla(\phi_\theta u^{p-1})
 \right\rangle_\varphi.
\]
Taking into account the skew-symmetry of $C$, we obtain
\[
\begin{aligned}
 L-J
 &=\left\langle
 C(\nabla u-\theta u),
 (p-1)u^{p-2}\nabla u+\theta u^{p-1}
 \right\rangle_\varphi\\
 &=-(p-1)
   \langle C\theta,u^{p-1}\nabla u\rangle_\varphi
   +\langle C\nabla u,\theta u^{p-1}\rangle_\varphi\\
 &=p\langle C\nabla u,\theta u^{p-1}\rangle_\varphi\\
 &=2\langle C\nabla v,\theta v\rangle_\varphi,
\end{aligned}
\]
where we have put $v=u^{p/2}$ as in the proof of Moser's Lemma~\ref{moser_lem}. Therefore,
\[
\begin{aligned}
 |L-J|
 &\leq
 2\|C\|_\infty|\theta|\,
 a^\prime[v,v]^{1/2}
 \langle v,v\rangle_\varphi^{1/2}\\
 &\leq
 \varepsilon_0 a^\prime[v,v]
 +\frac{\|C\|_\infty^2|\theta|^2}{\varepsilon_0}
       \langle v,v\rangle_\varphi,
 \qquad \varepsilon_0>0,
\end{aligned}
\]
where
\[
 a^\prime[v,v]=\langle\nabla v,\nabla v\rangle_\varphi.
\]
Taking $\varepsilon_0=1/(4p)$ gives
\[
 |L-J|
 \leq\frac1{4p}a^\prime[v,v]
     +4p\|C\|_\infty^2|\theta|^2
          \langle v,v\rangle_\varphi.
\]

It remains to estimate the term containing $b_{\varepsilon+s}$. We have
$$
\langle b_{\varepsilon+s}\cdot C  (\nabla u-\theta u),u^{p-1}\rangle_\varphi = \frac2p\langle b_{\varepsilon+s}\cdot C\nabla v,  v\rangle_\varphi -\langle(b_{\varepsilon+s}\cdot C\theta)v,  v\rangle_\varphi.
$$
Using \eqref{bb_est}, we obtain
\[
\begin{aligned}
 \left|
 \langle b_{\varepsilon+s}\cdot C  (\nabla u-\theta u),u^{p-1}\rangle_\varphi
 \right|
& \leq
 2\|C\|_\infty\|b_{\varepsilon+s}\|_\infty
 a^\prime[v,v]^{1/2}
 \langle v,v\rangle_\varphi^{1/2}\\
 &\quad+
 \|C\|_\infty\|b_{\varepsilon+s}\|_\infty|\theta|
 \langle v,v\rangle_\varphi\\
 &\leq
 \frac1{4p}a^\prime[v,v]
 +cp(1+\|C\|_\infty)^2
      \left(|\theta|^2+\frac1s\right)
      \langle v,v\rangle_\varphi,
\end{aligned}
\]
where $c$ depends only on $d,N,\nu$. Combining these estimates with the bound obtained in
the proof of Lemma~\ref{moser_lem},
$$
 \langle\Lambda_\theta u,u^{p-1}\rangle_\varphi \geq\frac1p a^\prime[v,v]  -cp\left(|\theta|^2+\frac1s\right) \langle v,v\rangle_\varphi,
$$
we arrive at
$$
 \langle\Lambda_\theta^c u,u^{p-1}\rangle_\varphi  \geq\frac1{2p}a^\prime[v,v] -cp(1+\|C\|_\infty)^2 \left(|\theta|^2+\frac1s\right) \langle v,v\rangle_\varphi.
$$
Now we can plug this estimate in the proof of Moser’s Lemma~\ref{moser_lem} and argue as before.

\bigskip

\section{Two-particle case}
\label{app:two-particles}

\subsection{Brownian case}
In this part of the appendix, we specify Theorems~\ref{thm1} and~\ref{thm2} to the heat kernel of the operator
$$
 \Lambda:=-\Delta+\frac{\nu}{2}\frac{x}{|x|^2}\cdot\nabla  =-\Delta-\frac{\nabla\psi}{\psi}\cdot\nabla
 \quad\text{on }\mathbb R^d,  \qquad  \psi(x)=|x|^{-\nu/2},
 \qquad d\geq2.
$$
The upper bounds on its heat kernel allow us to estimate from above the density of the law of the relative displacement
$Z_t:=X_t^1-X_t^2$ of the two-particle system,
$$
 \left\{
 \begin{aligned}
 dX_t^1
 &=-\frac{\nu}{2}  \frac{X_t^1-X_t^2}{|X_t^1-X_t^2|^2}\,dt  +\sqrt2\,dB_t^1,\\
 dX_t^2
 &=-\frac{\nu}{2}   \frac{X_t^2-X_t^1}{|X_t^2-X_t^1|^2}\,dt    +\sqrt2\,dB_t^2,
 \end{aligned}
 \right.
$$
where $\{B_t^1\}$ and $\{B_t^2\}$ are independent $d$-dimensional Brownian motions; this can be easily seen by subtracting the corresponding SDEs and using the fact that $\psi$ is an even function. Since the generator of $Z_t$ is $-2\Lambda$, its transition density is given by the heat kernel of $\Lambda$ evaluated at time $2t$.

We define the $\varepsilon$-regularized drift and operator $\Lambda_\varepsilon$ in the same way as we did in
Theorems~\ref{thm1} and~\ref{thm2}:
$$
 b_\varepsilon(x) :=\frac{\nu}{2}\frac{x}{|x|^2+\varepsilon},  \qquad \Lambda_\varepsilon:=-\Delta+b_\varepsilon\cdot\nabla.
$$
We denote by $p_\varepsilon$ in dimension $d=2$ and by $k_\varepsilon$ in dimensions $d\geq3$ the corresponding
heat kernels := the integral kernels of $e^{-t\Lambda_\varepsilon}$.

Define, in every dimension $d\geq2$,
\[
 \psi_\varepsilon(y):=|y|_\varepsilon^{-\nu/2},
 \qquad
 |y|_\varepsilon:=\sqrt{|y|^2+\varepsilon},
\]
and
\[
 \varphi_{t,\varepsilon}(y)
 :=\frac{\psi_\varepsilon(y)}{\psi_{\varepsilon+t}(y)}
 =\left(1+\frac{t}{|y|^2+\varepsilon}\right)^{\nu/4}.
\]
We set $\psi=\psi_0$ and
$\varphi_t=\varphi_{t,0}$ outside $\{0\}$.

\begin{theorem}
The following are true:

\begin{enumerate}
\item[\rm ($d=2$)] Assume that the coupling constant $\nu$ satisfies
$$
 \nu<  \max_{1\leq\beta<2}
 \left[
 \frac1{2^{2\beta}}
 \frac{\Gamma\left(\frac12-\frac{\beta}{4}\right)^2}
      {\Gamma\left(\frac12+\frac{\beta}{4}\right)^2}
 \right]^{-1/\beta}
 = 4\frac{\Gamma\left(\frac34\right)^2}  {\Gamma\left(\frac14\right)^2}.
$$
This is precisely condition~\eqref{alpha_max} with $N=2$. Then, for every $\varepsilon>0$,
\[
 p_\varepsilon(t,x,y)
 \leq Ct^{-1}\varphi_{t,\varepsilon}(y)
\]
for all $t\in]0,T]$, $x,y\in\mathbb R^2$, for a constant
$C=C(T)$ independent of $\varepsilon$.
Furthermore, there exists the limit
\[
 e^{-t\Lambda}
 :=
 s\mbox{-}L^2\mbox{-}\lim_{\varepsilon\downarrow0}
 e^{-t\Lambda_\varepsilon}
 \qquad
 \text{(loc.\ uniformly in $t\geq0$)}.
\]
This is a semigroup of integral operators,
\[
 e^{-t\Lambda}f(x)
 =:\int_{\mathbb R^2}p(t,x,y)f(y)\,dy,
\]
whose integral kernel $p(t,x,y)$ satisfies, for every
$t\in]0,T]$,
\[
 p(t,x,y)\leq Ct^{-1}\varphi_t(y)
\]
a.e.\ on $\mathbb R^2\times\mathbb R^2$.

\medskip

\item[\rm ($d\geq3$)]
Assume that
$$
 \nu<2(d-2).
$$
Then, for every $\varepsilon>0$,
\[
 k_\varepsilon(t,x,y)
 \leq c_1\Gamma_{c_2t}(x-y)\varphi_{t,\varepsilon}(y)
\]
on $]0,T]\times\mathbb R^d\times\mathbb R^d$,
where constants $c_1,c_2$ are independent of $\varepsilon$.
Further, for every $2\leq r<\infty$ satisfying
\[
 r>\frac4{4-2\nu/(d-2)},
\]
there exists the limit
\[
 e^{-t\Lambda}
 :=
 s\mbox{-}L^r\mbox{-}\lim_{\varepsilon\downarrow0}
 e^{-t\Lambda_\varepsilon}
 \qquad
 \text{(loc.\ uniformly in $t\geq0$)},
\]
a semigroup of integral operators,
\[
 e^{-t\Lambda}f(x)
 =:\int_{\mathbb R^d}k(t,x,y)f(y)\,dy,
\]
whose integral kernel satisfies, for every $t\in]0,T]$,
\[
 k(t,x,y)
 \leq c_1\Gamma_{c_2t}(x-y)\varphi_t(y)
\]
a.e.\ on $\mathbb R^d\times\mathbb R^d$.

\end{enumerate}
\end{theorem}

\begin{remark}
In dimensions $d \geq 3$ the result was obtained earlier in \cite{MS} and \cite{MSS} using Davies's device and Moser iterations. See also \cite{MNS} where this estimate, among many other results, is obtained using spherical harmonics.  
\end{remark}

\begin{proof}
We repeat the proofs of Theorems~\ref{thm1}, \ref{thm2} and Proposition~\ref{prop_semigroup},
using the ordinary (fractional where needed) Hardy inequality instead of the many-particle Hardy inequality. Since
$$
 \nabla\psi_\varepsilon(x)  =-\frac{\nu}{2}|x|_\varepsilon^{-\nu/2-2}x,
$$
we have, in every dimension,
$$
 b_\varepsilon(x) =-\frac{\nabla\psi_\varepsilon(x)}{\psi_\varepsilon(x)} =\frac{\nu}{2}\frac{x}{|x|_\varepsilon^2}.
$$

\medskip
\noindent
\textit{The case $d=2$.}
For every $1\leq\beta<2$,
$$
 |b_\varepsilon(x)|^\beta
 \leq\left(\frac{\nu}{2}\right)^\beta|x|^{-\beta}.
$$
Thus, by the fractional Hardy inequality in $\mathbb R^2$, i.e., \eqref{frac2},
\[
 \langle|b_\varepsilon|^\beta g,g\rangle
 \leq
 \frac{\nu^\beta}{4^\beta}
 \frac{\Gamma\left(\frac12-\frac{\beta}{4}\right)^2}
      {\Gamma\left(\frac12+\frac{\beta}{4}\right)^2}
 \|(-\Delta)^{\beta/4}g\|_2^2,
 \qquad
 g\in  W^{\beta/2,2}(\mathbb R^2).
\]
In particular, $b_\varepsilon$ is $\beta$-form-bounded with form-bound independent of $\varepsilon$. To carry out the argument in the proof of Theorem~\ref{thm1},  we need
$$
 \frac{\nu^\beta}{4^\beta}
 \frac{\Gamma\left(\frac12-\frac{\beta}{4}\right)^2}  {\Gamma\left(\frac12+\frac{\beta}{4}\right)^2} <1, 
$$
hence the condition on $\nu$.

\medskip
\noindent
\textit{The case $d\geq3$.}
The ordinary Hardy inequality gives
$$
\begin{aligned}
 \langle|b_\varepsilon|^2g,g\rangle \leq\frac{\nu^2}{4}  \left\langle\frac1{|x|^2},|g|^2\right\rangle \leq\frac{\nu^2}{(d-2)^2}
       \langle|\nabla g|^2\rangle,
 \qquad g\in W^{1,2}(\mathbb R^d).
\end{aligned}
$$
Moreover,
\[
 \operatorname{div}b_\varepsilon
 =\frac{\nu}{2}
   \frac{(d-2)|x|^2+d\varepsilon}
        {(|x|^2+\varepsilon)^2}
 \geq0,
\]
and
\[
 |b_\varepsilon|^2
 \leq\frac{\nu}{2(d-2)}
       \operatorname{div}b_\varepsilon.
\]
Arguing as in Lemma~\ref{comp_lem}, we obtain
\[
 \langle\operatorname{div}b_\varepsilon\,|g|^2\rangle
 \leq\frac{2\nu}{d-2}\langle|\nabla g|^2\rangle.
\]
Since $ -\Delta\log\varphi_{s,\varepsilon} =\operatorname{div}b_\varepsilon -\operatorname{div}b_{\varepsilon+s}
 \leq\operatorname{div}b_\varepsilon$, it follows that
$$
 \langle-\Delta\log\varphi_{s,\varepsilon}\,|g|^2\rangle  \leq\frac{2\nu}{d-2}\langle|\nabla g|^2\rangle.
$$
Also,
$$
 |b_{\varepsilon+s}|^2  =\frac{\nu^2}{4}   \frac{|x|^2}{(|x|^2+\varepsilon+s)^2}  \leq\frac{\nu^2}{16s}.
$$
With these estimates, we can repeat the weighted Sobolev and Moser arguments in the proof of Theorem~\ref{thm2}, to obtain the asserted Gaussian upper bound on $k_\varepsilon$. Since
$
 \frac{4(r-1)}r-\frac{2\nu}{d-2}>0,
$
the argument of Proposition~\ref{prop_semigroup} yields strong convergence in $L^r$, locally uniformly for $t\geq0$.
The Gaussian upper bound for the limiting kernel then follows by passing to the limit as in the proof of Theorem~\ref{thm2}.
\end{proof}

\subsection{Stable-driven case}
\label{app:Non-local two particles}

Now, we adapt the proof of Theorem~\ref{nl:main} to the heat kernel of the operator
\begin{equation}\label{nl:two-operator}
 \Lambda_\alpha=A_\alpha+b_\alpha\cdot\nabla,
 \qquad
 A_\alpha=2(-\Delta)^{\alpha/2},
 \qquad
 b_\alpha(x)=\nu\frac{x}{|x|^\alpha},
 \quad x\in\mathbb R^d,
\end{equation}
where $d\ge2$, $1<\alpha<2$, and $\nu\ge0$. The upper bounds on its heat kernel allow us to estimate
from above the density of the law of the relative displacement $Y_t:=X_t^1-X_t^2$ of the two-particle system,
\[
 \left\{
 \begin{aligned}
 dX_t^1
 &=-\frac{\nu}{2}
   \frac{X_t^1-X_t^2}{|X_t^1-X_t^2|^\alpha}\,dt
   +dZ_t^1,\\
 dX_t^2
 &=-\frac{\nu}{2}
   \frac{X_t^2-X_t^1}{|X_t^2-X_t^1|^\alpha}\,dt
   +dZ_t^2,
 \end{aligned}
 \right.
\]
where $\{Z_t^1\}$ and $\{Z_t^2\}$ are independent $d$-dimensional isotropic $\alpha$-stable Lévy processes,
each with generator $-(-\Delta)^{\alpha/2}$. Subtracting the two equations gives
$$
 dY_t=-\nu\frac{Y_t}{|Y_t|^\alpha}\,dt  +dZ_t^1-dZ_t^2.
$$
Since $Z_t^1-Z_t^2$ has generator $-2(-\Delta)^{\alpha/2}$, the generator of $Y_t$ is $-\Lambda_\alpha$.
Thus, its transition density is the heat kernel of $\Lambda_\alpha$ evaluated at time $t$. For $\varepsilon>0$, we define the regularized drift and operator by
$$
 b_{\alpha,\varepsilon}(x) :=\nu\frac{x}{(|x|^2+\varepsilon)^{\alpha/2}},
 \qquad \Lambda_{\alpha,\varepsilon}
 :=2(-\Delta)^{\alpha/2}+b_{\alpha,\varepsilon}\cdot\nabla.
$$
Recall 
$$
 \kappa_{d,\alpha}  =2^\alpha
  \left[
   \frac{\Gamma((d+\alpha)/4)}
        {\Gamma((d-\alpha)/4)}
  \right]^2,
$$
and 
$$
 D_{d,\alpha}
 =\sup_{r\ge0}
   \frac{d-\alpha+dr}{(1+r)^{1+\alpha/2}}
 =\begin{cases}
  2\bigl(d/(\alpha+2)\bigr)^{1+\alpha/2},
       &d<\alpha+2,\\
  d-\alpha,&d\ge\alpha+2.
  \end{cases}
$$
For $s>0$, put
\begin{equation*}
 \varphi_{\alpha,s,\varepsilon}(x)
 =\left(1+\left(\frac{s^{1/\alpha}}{|x|_\varepsilon}\right)^{(d-\alpha)/2}
   \right)^{\frac{3\nu(d-\alpha)}{4\kappa_{d,\alpha}}},
 \qquad
 \varphi_{\alpha,s}=\varphi_{\alpha,s,0}.
\end{equation*}

\begin{theorem}\label{nl:two-main}
Assume that the strength of attraction $\nu$ satisfies
\begin{equation*}
 \nu<  \frac{4\kappa_{d,\alpha}}  {2D_{d,\alpha}+d-\alpha}.
\end{equation*}
Then assertions \textit{(i)}-\textit{(iii)} of
Theorem~\ref{nl:main} hold for operator $\Lambda_\alpha$ in \eqref{nl:two-operator} with $dN$ replaced by $d$ in the kernel and dispersion estimates.
\end{theorem}

\bigskip

\makeatletter
\let\savedtocwrite\@tocwrite
\let\@tocwrite\@gobbletwo
\section*{AI usage }
\let\@tocwrite\savedtocwrite
\makeatother

OpenAI ChatGPT helped us to carry out calculations needed to strengthen the results presented in the first version of the paper (written in 2025 without any AI), i.e.\,making the condition on $\nu$ in dimension $d=2$ non-degenerating as $N$ goes to infinity, and extending our approach to stable-driven particles. The approach itself and the key calculations are due to the authors and stem from their earlier works.

\end{document}